\documentclass[a4, 12pt]{amsart}
\usepackage{amssymb,amstext,amscd,amsthm,amsfonts,mathdots,mathrsfs}
\usepackage{pdflscape}
\usepackage{hyperref}
\usepackage{pdflscape}
\usepackage{setspace}
\usepackage[all]{xy}
\usepackage{enumerate}
\usepackage{indentfirst}
\usepackage{ytableau}
\usepackage{color}
\usepackage{colortbl}
\usepackage{amsmath}
\usepackage{tikz}
\usetikzlibrary{shapes.geometric, arrows}
\usetikzlibrary{calc}
\numberwithin{equation}{section}

\usepackage[thicklines]{cancel}

\usepackage{cleveref}
\crefname{theorem}{Theorem}{Theorems}
\crefname{definition}{Theorem}{Definitions}
\crefname{proposition}{Theorem}{Propositions}
\crefname{corollary}{Corollary}{Corollaries}
\crefname{alphatheorem}{Theorem}{Theorems}

\newtheorem{theorem}{Theorem}[section]
\newtheorem{lemma}[theorem]{Lemma}
\newtheorem{proposition}[theorem]{Proposition}
\newtheorem{definition}[theorem]{Definition}

\newtheorem{corollary}[theorem]{Corollary}
\newtheorem{example}[theorem]{Example}

\theoremstyle{remark}
\newtheorem{rem}[theorem]{Remark}

\newtheorem{alphatheorem}{\bf Theorem}

\newtheorem{alphaproposition}[alphatheorem]{Proposition}

\newcommand{\clr}{rgb:black,3;blue,2;red,0}

\tikzset{anchorbase/.style={baseline={([yshift=-0.5ex]current bounding box.center)}}}
\tikzset{ 
    centerzero/.style={>=To,baseline={([yshift=-0.5ex](#1))}},
    centerzero/.default={0,0}
}
\tikzset{wipe/.style={white,line width=4pt}}
\newcommand{\cred}{rgb:black,.5;blue,0;red,1.5}

\DeclareFontFamily{OT1}{pzc}{}
\DeclareFontShape{OT1}{pzc}{m}{it}{ <-> s*[1.2] pzcmi7t }{}
\DeclareMathAlphabet{\mathpzc}{OT1}{pzc}{m}{it}
\newcommand{\AH}{\mathpzc{AH}}

\newcommand{\qW}{\mathpzc{WeB}}  
 
\newcommand{\qAW}{\mathpzc{WeB}^\bullet} 
\newcommand{\qAWC}{\mathpzc{WeB}^{\bullet\prime}}
\newcommand{\qSch}{\mathpzc{SchuR}}
\newcommand{\qASch}{\mathpzc{SchuR}^{{\hspace{-.03in}}\bullet}}
\newcommand{\qSchuDJM}{\mathpzc{SchuR}^{{\hspace{-.03in}}\text{DJM}}_{\bfu}}
\newcommand{\qSchu}{\mathpzc{SchuR}_{\bfu}}
\newcommand{\qASchC}{\mathpzc{SchuR}^{{\hspace{-.03in}}\bullet\prime}} 
\newcommand{\bT}{\mathbf{T}}
\newcommand{\bS}{\mathbf{S}}

\newcommand{\Smu}{\Sc_{m,\bfu}}
\newcommand{\Hu}{\mathcal{H}_{m,\bfu}}
\newcommand{\brY}{\blue{\rot{Y}}}
\newcommand{\bY}{\blue{\text{Y}}}
 
\newcommand{\Sym}{\text{Sym}}
\newcommand{\Mat}{\text{Mat}}
\newcommand{\PMat}{\text{ParMat}}
\newcommand{\RPMat}{\text{RParMat}}
\newcommand{\Par}{\text{Par}}

\newcommand{\RPar}{\text{RPar}}
\newcommand{\SST}{\text{SST}}

\newcommand{\x}{\textsc{X}}

\newcommand{\bfu}{\red{\mathbf{u}}}
\newcommand{\cdeg}{\mathrm{deg}^{\times}}
\newcommand{\ddeg}{\mathrm{deg}^{\bullet}}

\newcommand{\Homleqk}{\Hom_{\qASch}(\mu,\lambda)_{\leq k}}
\newcommand{\Homlesk}{\Hom_{\qASch}(\mu,\lambda)_{< k}}

\newcommand\C{\mathbb{C}}
\newcommand\Z{\mathbb{Z}}

\newcommand\N{\mathbb{N}}

\newcommand\eps{(q^{-1}-q)}

\newcommand{\qbinom}[2]{\begin{bmatrix} #1\\#2 \end{bmatrix} }
\newcommand{\qbinomsm}{\textstyle\genfrac{[}{]}{0pt}{}}

\newcommand\kk{\Bbbk}
\newcommand\la{\lambda}
\newcommand{\Hom}{{\rm Hom}}
\newcommand{\End}{{\rm End}}
\newcommand{\rot}{\rotatebox[origin=c]{180}}
\newcommand{\arxiv}[1]{\href{http://arxiv.org/abs/#1}{\tt arXiv:\nolinkurl{#1}}}

\newcommand{\equivc}{\equiv_{\times}}

\def\std{\text{ST}}

\def\s{\mathfrak s}

\def\t{\mathfrak t}
\def\Sc{\mathcal S}

\newcommand{\upliftuip}{\begin{tikzpicture}[baseline = 10pt, scale=.8, color=\clr]
 \draw[-,line width=1.2pt] (-0.3,.3) to (.3,1);
\draw[-,line width=1pt,color=\cred] (0.3,.3) to (-.3,1);
\draw(0.35,.15) node {$\scriptstyle \red{u_{i+1}}$};
\draw (-.3,0.15) node{$\scriptstyle a$};
\draw (.3,1.15) node{$\scriptstyle {a}$};
\end{tikzpicture}}

\newcommand{\downliftuip}{\begin{tikzpicture}[baseline = 10pt, scale=.8, color=\clr]
 \draw[-,line width=1pt,color=\cred] (-0.3,.3) to (.3,1);
 \draw(-.35,.15) node {$\scriptstyle \red{u_{i+1}}$};
\draw[-,line width=1.2pt] (0.3,.3) to (-.3,1);
\draw (-.3,1.15) node{$\scriptstyle a$};
\draw (.3,0.15) node{$\scriptstyle {a}$};
\end{tikzpicture}}

\newcommand{\str}{\begin{tikzpicture}[baseline = 10pt, scale=0.4, color=\clr]
            \draw[-,thick] (0,0.5)to[out=up,in=down](0,1.7);
            \draw (0,0.27) node{$\scriptstyle 1$};
\end{tikzpicture} 
}

\newcommand{\stra}{\begin{tikzpicture}[baseline = 10pt, scale=0.4, color=\clr]
            \draw[-,line width=1.2pt] (0,0.7)to[out=up,in=down](0,1.8);
            \draw (0,0.49) node{$\scriptstyle a$};
\end{tikzpicture} 
}

\newcommand{\stru}{\begin{tikzpicture}[baseline = 10pt, scale=0.4, color=\cred]
\draw[-,line width=1pt] (0,0.4) to (0,1.7);
\draw(0,0.1) node {$\scriptstyle u$};
\end{tikzpicture}
}

\newcommand{\zdot}{ node[circle,fill=white,draw,   
thin, inner sep=0pt, minimum width=3.2pt]{}
} 

\newcommand{\bdot}{ node[circle, draw, fill=\clr, thick, inner sep=0pt, minimum width=2.8pt]{}
}

\newcommand{\xdota}{
\begin{tikzpicture}[baseline = 3pt, scale=0.4, color=\clr]
\draw[-,line width=1.2pt] (0,0) to[out=up, in=down] (0,1.4);
\draw(0,0.6) \bdot;
\node at (0,-.22) {$\scriptstyle a$};
\end{tikzpicture}
}

\newcommand{\zdotgen}{
\begin{tikzpicture}[baseline = 3pt, scale=0.4, color=\clr]
\draw[-,thick] (0,0) to[out=up, in=down] (0,1.4);
\draw(0,0.6) \zdot;
\node at (0,-.3) {$\scriptstyle 1$};
\end{tikzpicture}
}
\newcommand{\zdota}{
\begin{tikzpicture}[baseline = 3pt, scale=0.4, color=\clr]
\draw[-,line width=1.2pt] (0,0) to[out=up, in=down] (0,1.4);
\draw(0,0.6) \zdot;
\node at (0,-.22) {$\scriptstyle a$};
\end{tikzpicture}
}

\newcommand{\merge}
{\begin{tikzpicture}[baseline = -.5mm,scale=.8, color=\clr]
	\draw[-,line width=1pt] (0.28,-.3) to (0.08,0.04);
	\draw[-,line width=1pt] (-0.12,-.3) to (0.08,0.04);
	\draw[-,line width=1.5pt] (0.08,.4) to (0.08,0);
        \node at (-0.18,-.4) {$\scriptstyle a$};
        \node at (0.35,-.4) {$\scriptstyle b$};
        \node at (0.05,.55){$\scriptstyle a+b$};
        \end{tikzpicture} }
        
\newcommand{\splits}
{\begin{tikzpicture}[baseline = -.5mm,scale=.8, color=\clr]
	\draw[-,line width=1.5pt] (0.08,-.3) to (0.08,0.04);
	\draw[-,line width=1pt] (0.28,.4) to (0.08,0);
	\draw[-,line width=1pt] (-0.12,.4) to (0.08,0);
        \node at (-0.2,.5) {$\scriptstyle a$};
        \node at (0.36,.5) {$\scriptstyle b$};
        \node at (0.1,-.41){$\scriptstyle a+b$};
\end{tikzpicture}}

\newcommand{\dotgen}
{\begin{tikzpicture}[baseline = 3pt, scale=0.4, color=\clr]
\draw[-,thick] (0,0) to[out=up, in=down] (0,1.4);
\draw(0,0.6) \bdot;
\node at (0,-.3) {$\scriptstyle 1$};
\end{tikzpicture}}

\newcommand{\crossingpos}{
\begin{tikzpicture}[baseline=-1mm, scale=.8, color=\clr]
	\draw[-,line width=1pt] (0.3,-.3) to (-.3,.4);
	\draw[-,line width=4pt,white] (-0.3,-.3) to (.3,.4);
	\draw[-,line width=1pt] (-0.3,-.3) to (.3,.4);
        \node at (-0.3,-.45) {$\scriptstyle a$};
        \node at (0.26,-.45) {$\scriptstyle b$};
\end{tikzpicture}
}

\newcommand{\crossingneg}{
\begin{tikzpicture}[baseline=-1mm, scale=.8, color=\clr]
 \draw[-,line width=1pt] (-0.3,-.3) to (.3,.4);
	\draw[-,line width=4pt,white] (0.3,-.3) to (-.3,.4);
 \draw[-,line width=1pt] (0.3,-.3) to (-.3,.4);
        \node at (-0.33,-.45) {$\scriptstyle b$};
        \node at (0.3,-.45) {$\scriptstyle a$};
\end{tikzpicture}
}

\newcommand{\rightcrossing}{\begin{tikzpicture}[baseline = 1mm, scale=.8, color=\clr]
 \draw[-,line width=1.2pt] (-0.3,0) to (.3,.7);
\draw[-,line width=1pt,color=\cred] (0.3,0) to (-.3,.7);
\draw(-.3,-0.1) node{$\scriptstyle a$};
\draw (.3, -0.1) node{$\scriptstyle \red{u}$};
\end{tikzpicture}}

\newcommand{\leftcrossing}{\begin{tikzpicture}[baseline = 1mm, scale=.8, color=\clr]
 \draw[-,line width=1pt,color=\cred] (-0.3,0) to (.3,.7);
\draw[-,line width=1.2pt] (0.3,0) to (-.3,.7);
\draw(-.3,-.1) node{$\scriptstyle \red{u}$};
\draw (.3, -.1) node{$\scriptstyle a$};
\end{tikzpicture}}

\newcommand{\upliftui}{\begin{tikzpicture}[baseline = 10pt, scale=.8, color=\clr]
 \draw[-,line width=1.2pt] (-0.3,.3) to (.3,1);
\draw[-,line width=1pt,color=\cred] (0.3,.3) to (-.3,1);
\draw(0.3,.15) node {$\scriptstyle u_i$};
\draw (-.3,0.15) node{$\scriptstyle a$};
\draw (.3,1.15) node{$\scriptstyle {a}$};
\end{tikzpicture}}

\newcommand{\downliftui}{\begin{tikzpicture}[baseline = 10pt, scale=.8, color=\clr]
 \draw[-,line width=1pt,color=\cred] (-0.3,.3) to (.3,1);
 \draw(-.3,.15) node {$\scriptstyle u_i$};
\draw[-,line width=1.2pt] (0.3,.3) to (-.3,1);
\draw (-.3,1.15) node{$\scriptstyle a$};
\draw (.3,0.15) node{$\scriptstyle {a}$};
\end{tikzpicture}}

\newcommand{\wkdota}{\begin{tikzpicture}[baseline = 3pt, scale=0.4, color=\clr]
\draw[-,line width=1.2pt] (0,0) to[out=up, in=down] (0,1.4);
\draw(0,0.6) \bdot; 
\draw (0.7,0.6) node {$\scriptstyle \omega_r$};
\node at (0,-.22) {$\scriptstyle a$};
\end{tikzpicture} }

\newcommand{\wkdotaa}{\begin{tikzpicture}[baseline = 3pt, scale=0.4, color=\clr]
\draw[-,line width=1.2pt] (0,0) to[out=up, in=down] (0,1.4);
\draw(0,0.6) \bdot; 
\node at (0,-.22) {$\scriptstyle a$};
\end{tikzpicture} }

\newcommand{\wkdotr}{
\begin{tikzpicture}[baseline = 3pt, scale=0.4, color=\clr]
\draw[-,line width=1.2pt] (0,0) to[out=up, in=down] (0,1.4);
\draw(0,0.6) \bdot; 
\draw (0.7,0.6) node {$\scriptstyle \omega_r$};
\node at (0,-.22) {$\scriptstyle r$};
\end{tikzpicture}
}

\newcommand{\wkzdotr}{
\begin{tikzpicture}[baseline = 3pt, scale=0.4, color=\clr]
\draw[-,line width=1.2pt] (0,0) to[out=up, in=down] (0,1.4);
\draw(0,0.6) \zdot; 
\draw (0.7,0.6) node {$\scriptstyle \omega_r$};
\node at (0,-.22) {$\scriptstyle r$};
\end{tikzpicture}
}

\newcommand{\blue}[1]{{\color{blue}#1}}
\newcommand{\red}[1]{{\color{red}#1}}

\newcommand{\Hd}{\widehat{\mathcal{H}}_{d}}

\newcommand{\unit}[1]{1_{#1}}

\begin{document}
\setlength{\baselineskip}{17pt}
\title{Affine and cyclotomic $q$-Schur categories via webs}

\author{Yaolong Shen}
\address{Department of Mathematics, University of Ottawa, Ottawa, ON, K1N 6N5, Canada} \email{yshen5@uottawa.ca}

\author{Linliang Song}
\address{School of Mathematical Science, Tongji University, Shanghai, 200092, China} \email{llsong@tongji.edu.cn}

\author{Weiqiang Wang}
 \address{Department of Mathematics, University of Virginia, Charlottesville, VA 22904, USA} \email{ww9c@virginia.edu}



\begin{abstract}
We formulate two new $\mathbb Z[q,q^{-1}]$-linear diagrammatic monoidal categories, the affine $q$-web category and the affine $q$-Schur category, as well as their respective cyclotomic quotient categories. Diagrammatic integral bases for the Hom-spaces of all these categories are established. In addition, we establish the following isomorphisms, providing diagrammatic presentations of these $q$-Schur algebras for the first time: (i)~ the path algebras of the affine $q$-web category to R.~Green's affine $q$-Schur algebras, (ii)~ the path algebras of the affine $q$-Schur category to Maksimau-Stroppel's higher level affine $q$-Schur algebras, and most significantly, (iii)~ the path algebras of the cyclotomic $q$-Schur categories to Dipper-James-Mathas' cyclotomic $q$-Schur algebras. 
\end{abstract}

\maketitle

\setcounter{tocdepth}{1}
\tableofcontents

%
\section{Introduction}

\subsection{The goal}

Let $\kk =\Z[q,q^{-1}]$ or $\C(q)$. In this sequel to \cite{SW24web,SW24Schur}, we shall introduce two new diagrammatic $\kk$-linear monoidal categories: affine $q$-web and affine $q$-Schur categories, denoted by $\qAW$ and $\qASch$, respectively. The category $\qAW$ is a $q$-deformation of the degenerate affine $q$-web category introduced in \cite{DKMZ} and independently in \cite{SW24web}, while $\qASch$ is a $q$-version of the degenerate affine $q$-Schur category initiated in \cite{SW24Schur}. 
Then we construct the cyclotomic quotients of these two categories: cyclotomic $q$-web categories $\qW_\bfu$ and cyclotomic $q$-Schur categories $\qSch_\bfu$, for $\bfu \in (\kk^*)^\ell$ and $\ell \ge 1$. 

For all these affine and cyclotomic categories, we shall establish integral bases for the corresponding Hom-spaces. Furthermore, we show that the path algebras of $\qAW$ are isomorphic to affine $q$-Schur algebras defined in \cite{Gre99}, the path algebras of $\qASch$ are isomorphic to the higher level $q$-Schur algebras defined in \cite{MS21}, and most significantly, the path algebras of $\qSch_\bfu$ are isomorphic to the cyclotomic $q$-Schur algebras introduced by Dipper-James-Mathas in \cite{DJM98}, which are an important family of quasi-hereditary algebras (cf. \cite{R08}). In this way, we shall provide diagrammatic presentations of all these variants of Schur algebras for the first time. 

Diagrammatic categories have provided powerful approaches to constructing important algebras which are well-suited for categorification; see \cite{KL3, Rou, Web17} for some well-known examples. The affine and cyclotomic $q$-Schur categories for $q$ generic and a root of $1$ will also be well suited for categorification; in addition, they will be building blocks for new diagrammatic categories suitable for $\imath$categorification related to $\imath$quantum groups. 

Let us explain our constructions in some detail. 

\subsection{Affine $q$-web and $q$-Schur categories}

The affine $q$-web category $\qAW$ admits generating objects $\stra$ ($a\in \Z_{\ge 1})$ and generating morphisms  
\[
\merge, \quad \splits, \quad \crossingpos, \quad \wkdotaa
\]
(where $\crossingpos$ is invertible with inverse $\crossingneg$ and $\xdota$ is invertible with inverse $\zdota$) subject to relations \eqref{webassoc}--\eqref{intergralballon}. The category $\qAW$ contains several notable monoidal subcategories. The subcategory generated by $\crossingpos$ and $\dotgen$ is the affine Hecke category whose endomorphism algebras are affine Hecke algebras of $GL$ type. The subcategory $\qW$ generated by $\merge, \splits, \crossingpos$ was studied recently by Brundan in \cite{Bru24} (who denoted it by $q$-{\bf Schur}); it is a variant of the Cautis-Kamntizer-Morrison web category \cite{CKM} and a $q$-version of the (polynomial) $\mathfrak{gl}$-web category in \cite{BEEO}. 

From our point of view, the category $\qAW$ is an intermediate step toward the construction of affine $q$-Schur category denoted $\qASch$. 
Besides the above generating morphisms for $\qAW$, the affine $q$-Schur category $\qASch$ admits additional generating morphisms 
\[
\rightcrossing, \quad \leftcrossing 
\qquad (a\in \Z_{\ge 1}, u\in \kk^*)
\]
subject to relations \eqref{webassoc}--\eqref{intergralballon} and \eqref{redslider}--\eqref{redbraid}. This is reminiscent of diagrammatics used by Webster in his diagrammatic algebras \cite{Web17, Web20}. 

Over the field $\C(q)$, the presentations of $\qAW$ and $\qASch$ can be much simplified. Indeed, for $\qAW$ (and respectively, $\qASch$) over $\C(q)$, we only need the generating morphisms $\merge, \splits, \crossingpos$ and $\dotgen$ (and respectively, with additional ones $\rightcrossing$ and $\leftcrossing$) together with the corresponding relations. We denote the $\C(q)$-linear monoidal categories with these generating morphisms and simpler relations by $\qAWC$ and $\qASchC$, respectively (see \cref{def:qAWebC,def:qSchurC}). 

\begin{alphatheorem}[\cref{th:qAWiso,th:qASchiso}]
  \label{thm:AffineC}
  The $\C(q)$-linear monoidal categories $\qAWC$ (respectively, $\qASchC$) and $\qAW$ (respectively, $\qASch$) are isomorphic. 
 \end{alphatheorem} 

Note that the definition of $\qAWC$ is rather straightforward (following the strategy advocated in \cite{SW24web}), as it contains two distinguished subcategories: $\qW$ from \cite{Bru24} and an affine Hecke category $\AH$, and the generating morphisms and relations of $\qAWC$ are those of the two subcategories. The key point here is that the remaining defining relations \eqref{dotmovecrossing} and \eqref{dotmovesplits+merge} in $\qAW$ can be derived over $\C(q)$ from these standard relations of $\qAWC$ once we make a crucial definition of the thick dots
\[
\xdota := \frac{1}{[a]!}
\begin{tikzpicture}[baseline = 1.5mm, scale=.4, color=\clr]
\draw[-, line width=1.2pt] (0.5,2) to (0.5,2.5);
\draw[-, line width=1.2pt] (0.5,0) to (0.5,-.4);
\draw[-,thin]  (0.5,2) to[out=left,in=up] (-.5,1)
to[out=down,in=left] (0.5,0);
\draw[-,thin]  (0.5,2) to[out=left,in=up] (0,1)
 to[out=down,in=left] (0.5,0);      
\draw[-,thin] (0.5,0)to[out=right,in=down] (1.5,1)
to[out=up,in=right] (0.5,2);
\draw[-,thin] (0.5,0)to[out=right,in=down] (1,1) to[out=up,in=right] (0.5,2);
\node at (0.5,.7){$\scriptstyle \cdots$};
\draw (-0.5,1) \bdot; 
\draw (0,1) \bdot; 
\node at (0.5,-.6) {$\scriptstyle a$};
\draw (1,1) \bdot;
\draw (1.5,1) \bdot; 
\node at (-.26,0) {$\scriptstyle 1$};
\node at (.3,0.34) {$\scriptstyle 1$};
\node at (1.25,0) {$\scriptstyle 1$};
\end{tikzpicture}
\qquad (a\ge 1)
\]
in $\qAWC$; see  \eqref{eq:dota}. Similar remarks apply to $\qASch$ versus $\qASchC$.

\subsection{Integral bases}

The basis results for diagrammatic categories are often very challenging to establish or simply unknown (see, e.g., Khovanov-Lauda \cite{KL3}, Cautis-Kamntizer-Morrison \cite{CKM}, Elias \cite{El15}, Bodish-Wu \cite{BWu23}, and Webster \cite{Web17}). 

Recall that bases for the Hom spaces in $\qW$ are parametrized by chicken foot ribbon diagrams \cite{Bru24}. Making an analogy with \cite{SW24web}, we expect that the algebra $\End_{\qAW}(\stra)$ is commutative and isomorphic to the center of the affine Hecke algebra $\widehat{\mathcal H}_a$, and hence admits a suitable dotted strand basis $\RPar_a$ in \eqref{eq:RPar}. Then bases for $\qAW$ are expected to be reduced dotted chicken foot ribbon diagrams where each strand of thickness $a$ is decorated by a dotted strand basis element for $\End_{\qAW}(\stra)$. 
The set of such $\la\times \mu$ reduced dotted chicken foot ribbon diagrams is denoted by $\RPMat_{\lambda,\mu}$; see \eqref{dottedreduced} for $\la,\mu \in \cup_{m\ge 0} \Lambda_{\text{st}}(m)$ and \eqref{RPMatblock}  for $\la,\mu \in \cup_{m\ge 0} \Lambda_{\text{st}}^{1+\ell}(m)$. (Recall here that objects in $\qAW$ can be identified with strict compositions $\cup_{m\ge 0} \Lambda_{\text{st}}(m)$ while objects in $\qASch$ can be identified with strict multi-compositions $\cup_{\ell,m\ge 0} \Lambda_{\text{st}}^{1+\ell}(m)$.)

Our next theorem confirms such expectations. 
 
 \begin{alphatheorem}  [\cref{thm:basisqAW,thm:basisASchur}]
  \label{thm:BasesAffine}
  Let $m\ge 1$. Then
  \begin{enumerate}
      \item 
$\RPMat_{\lambda,\mu}$ in \eqref{dottedreduced} forms a basis for $\Hom_{\qAW}(\mu,\lambda)$, for any $\lambda,\mu \in \Lambda_{\text{st}}(m)$. 
\item 
$\RPMat_{\lambda,\mu}$ in \eqref{RPMatblock} forms a basis for $\Hom_{\qASch}(\mu,\lambda)$, for any $\lambda,\mu \in \Lambda_{\text{st}}^{1+\ell}(m)$.    
\item 
$\qAW$ is a full subcategory of $\qASch$. 
  \end{enumerate}
 \end{alphatheorem} 

It turns out that the categories $\qAW$ and $\qASch$ are closely related to various $q$-Schur algebras in the literature. Denote by $\qASch (m)$ the path algebras of the category $\qASch$ and by $\qAW (m)$ the path algebras of the category $\qAW$, for $m\ge 1$; see \eqref{affineSchur_red} and  \eqref{affineSchur_black}.
Recall that the affine $q$-Schur algebras were defined in \cite{Gre99} and the so-called higher level affine $q$-Schur algebras were defined in \cite{MS21} (also cf. \cite{Web20}).

\begin{alphaproposition} 
[Proposition \ref{prop:path_MSchur}]
\label{cor:isoSchur}
Let $m\ge 1$. Then 
\begin{enumerate}
\item 
The path algebras $\qAW (m)$ are isomorphic to the affine $q$-Schur algebras.
\item 
The path algebras $\qASch (m)$ are isomorphic to the higher level affine $q$-Schur algebras.
\end{enumerate}
\end{alphaproposition}
In this way, we have obtained diagrammatic presentations for these affine $q$-Schur algebras for the first time. Some related generalized $q$-Schur algebras are studied in a recent paper \cite{LaiMinets25}.

 \subsection{Cyclotomic quotients}

Let $\bfu \in (\kk^*)^\ell$ with $\ell \ge 1$. We formulate the cyclotomic quotient categories (of level $\ell$): $\qW_\bfu$ and $\qSch_\bfu$ of $\qAW$ and $\qASch$, respectively; see Definitions \ref{def:cycWeb}--\ref{def:qcSchur}. These categories are naturally related, as illustrated by the following commutative diagrams.
\begin{align}  \label{diag:qHWS}
\xymatrix{
\qW \ar[r]  
& \qAW \ar[r]  \ar[d]  
& \qW_\bfu   \ar[d] 
\\
 & \qASch \ar[r] & \qSch_\bfu
}
\end{align}
The bases for the Hom spaces of these cyclotomic categories are certain bounded reduced dotted chicken foot ribbon diagrams. In particular, the level one quotients $\qW_\bfu$ and $\qSch_\bfu$ are all isomorphic to $\qW$ (see \cite{Bru24}). 

 \begin{alphatheorem}  [\cref{thm:basis-cyc-qweb,thm:basisqSchu}]
  \label{thm:CyclotBases}
  Let ${\bfu} \in (\kk^*)^\ell$ and $\ell \ge 1$. Then
  \begin{enumerate}
     \item 
    $\PMat_{\lambda,\mu}^\ell$ in \eqref{Def-spansetof-cycqweb} forms a basis for $\Hom_{\qW_{\bfu}}(\mu,\lambda)$, for any $\lambda,\mu \in\Lambda_{\text{st}}(m)$.   
    \item 
        $\PMat_{\nu,\mu}^\flat$ in \eqref{PM0} forms a basis for $\Hom_{\qSch_{\bfu}}(\mu,\nu)$, for any $\mu,\nu\in \Lambda_{\text{st}}^\ell(m)$.
 \item 
$\qW_\bfu$ is a full subcategory of $\qSch_\bfu$. 
 \end{enumerate}
 \end{alphatheorem} 

Cyclotomic $q$-Schur algebras were introduced by Dipper-James-Mathas in \cite{DJM98} as endomorphism algebras of a direct sum of permutation modules over cyclotomic Hecke algebras, and they admit cellular bases \eqref{basisofdjm}. They form quasi-hereditary covers of 
cyclotomic Hecke algebras which are intimately related to categorification of representations of quantum groups; cf. \cite{R08, RSVV}. A final main result of this paper will connect our diagrammatic cyclotomic quotients with the cyclotomic $q$-Schur algebras.

To that end, let us formally formulate a small $\kk$-linear category $\qSchuDJM$ with endomorphism algebras being the cyclotomic $q$-Schur algebras. A functor $\mathcal G: \qSchu\rightarrow \qSchuDJM$ is formulated in \cref{thm:G}.
 
 \begin{alphatheorem}  [\cref{thm:G,thm:SSTbasis,thm:isoDJM}]
  \label{th:DJMiso}
  Let $m, \ell \ge 1$, and ${\bfu} \in(\kk^*)^\ell$. Then
  \begin{enumerate}
    \item 
    $\Hom_{\qSch_{\bfu}}(\mu,\nu)$ admits a double SST basis \eqref{doubleSST Hom}, for any $\mu,\nu\in \Lambda_{\text{st}}^\ell(m)$.
   \item 
   The functor $\mathcal G: \qSchu\rightarrow \qSchuDJM$ is an isomorphism.
   \item There exists an isomorphism between the path algebras of $\qSchu$ and the cyclotomic $q$-Schur algebras, $\qSchu(m)\cong \Smu$; moreover, the isomorphism matches the double SST basis \eqref{doubleSST} with the cellular basis \eqref{basisofdjm}.
  \end{enumerate}
 \end{alphatheorem} 
 
In contrast to the degenerate setting \cite{SW24web, SW24Schur}, we need to keep track of the positive crossing $\crossingpos$ and negative crossing $\crossingneg$ in the current $q$-setting. On the other hand, several crucial relations (see \eqref{dotmovecrossing} and \eqref{dotmovesplitabr}) have become more homogeneous and simpler in the $q$-setting. Deriving various new identities for morphisms in $\qAW$ and $\qASch$ requires independent computations from {\em loc. cit.} On the other hand, proofs of several main results follow strategies similar to {\em loc. cit.} and we choose to refer to the earlier work to avoid repetition. 

\subsection{Other works}

Wada gave a presentation of the cyclotomic $q$-Schur algebras $\Smu$ in terms of some non-standard quantum groups. The diagrammatic presentation of $\Smu$ given by \cref{thm:isoDJM}(3) is more conceptual and much simpler than the one given by Wada \cite{Wad11}. 

In light of the connection developed in \cite{SW24web} between degenerate cyclotomic webs and certain Schur algebras arising from finite W-algebras, we expect that the path algebras of $\qW_\bfu$ should be closely related to finite $q$-$W$-algebras of type $A$ introduced by Sevostyanov \cite{Sev11}. This provides us a renewed motivation to studying finite $q$-$W$-algebras. 

Similar to \cite{KL3, Rou}, the categories constructed in this paper are well-suited for categorification, for $q$ generic or a root of unity, and this will be addressed elsewhere. We refer to the Introductions in \cite{SW24web, SW24Schur} for in-depth discussions on this research program of diagrammatic Schur categories and categorification.
Building on the categories {\em loc. cit.} and in this paper, we shall construct in a sequel new monoidal categories and cyclotomic quotients, which are closely related to representations of classical Lie algebras and  categorification related to $\imath$quantum groups.

\subsection{The organization}

The paper is organized as follows. 
In Section~\ref{sec:web}, we introduce the affine $q$-web categories $\qAW$ and $\qAWC$, and derive many additional identities from the defining relations. We establish the isomorphisms of $\qAWC$ and $\qAW$ over $\C(q)$. 

In Section~\ref{sec:Schur}, we introduce the affine $q$-Schur categories $\qASch$ and $\qASchC$, and again derive many additional identities from the defining relations. We establish the isomorphisms of $\qASchC$ and $\qASch$ over $\C(q)$. In addition, we formulate a polynomial representation of $\qASch$; see \cref{th:ASchRep}. 

In Section~\ref{sec:bases_affineWS}, we establish diagrammatic bases for Hom spaces in $\qASch$ and then in $\qAW$ in terms of reduced dotted chicken foot ribbon diagrams. Using the faithful polynomial representation of $\qASch$, we show that the path algebras $\qASch (m)$ (see \eqref{affineSchur_red} for definition) are isomorphic to the higher level $q$-Schur algebras defined in \cite{MS21}.

In Section~\ref{sec:cyclotomic}, we define the cyclotomic $q$-web category $\qW_\bfu$ and the cyclotomic $q$-Schur category $\qSch_\bfu$, and establish bases for their Hom spaces in terms of bounded reduced dotted chicken foot ribbon diagrams. For the Hom spaces of $\qSch_\bfu$, we establish the double SST bases. We show that the path algebras $\qSchu(m)$  (see \eqref{eq:qSchum} for definition) are isomorphic to the cyclotomic Schur algebras $\Smu$.

\vspace{2mm}

\noindent {\bf Acknowledgement.} 
LS is partially supported by NSFC (Grant No. 12071346), and he thanks Institute of Mathematical Science and Department of Mathematics at University of Virginia for hospitality and support, where this project was initiated. YS is partially supported by the Fields Institute. WW is partially supported by DMS--2401351. 
\section{The affine $q$-web category}
\label{sec:web}

In this section, we introduce a $\kk$-linear monoidal category, the affine $q$-web category $\qAW$, where $\kk =\Z[q,q^{-1}]$ or $\C(q)$. We also formulate a $\C(q)$-linear variant $\qAWC$ with fewer generating morphisms/relations, and show that it is isomorphic to $\qAW$ over $\C(q)$. Along the way, we derive various additional relations in $\qAW$ which will be used in later sections.

\subsection{Definition of $\qAW$}
For $n, a \in \N$, we define the $ q$-numbers and $q$-binomial coefficients to be
\begin{align}
\label{qbinom}
    [n] =[n]_q =\frac{q^n -q^{-n}}{q-q^{-1}},
    \qquad
    \qbinom{n}{a} =\frac{[n][n-1]\cdots [n-a+1]}{[a]!}.
\end{align}

Let $\kk =\Z[q,q^{-1}]$ or $\C(q)$.
\begin{definition}
  \label{def:qAWeb}
The affine $q$-web category $\qAW$ is the strict $\kk$-linear monoidal category generated by objects $a\in \mathbb Z_{\ge1}$. The object $a$ and its identity morphism  will be drawn as a vertical strand labeled by $a$,
$\stra.$
The generating morphisms are the merges, splits and (thick) positive crossings depicted as
\begin{align}
\label{merge+split+crossing}
\merge 
&:(a,b) \rightarrow (a+b),&
\splits
&:(a+b)\rightarrow (a,b),&
\crossingpos
&:(a,b) \rightarrow (b,a),
\end{align}
and 
\begin{equation}
\label{dotgenerator}
\wkdotaa , 
\end{equation}
for $a,b \in \Z_{\ge 1}$, where $\crossingpos$ is invertible with inverse being the negative crossing denoted $\crossingneg := \crossingpos^{-1}$ and $\xdota$ is invertible with inverse denoted by $\zdota$. 
These generating morphisms are subject to the following relations \eqref{webassoc}--\eqref{intergralballon}, for $a,b,c,d \in \Z_{\ge 1}$ with $d-a=c-b$:
\begin{align}
\label{webassoc}
\begin{tikzpicture}[baseline = 0,color=\clr]
	\draw[-,thick] (0.35,-.3) to (0.08,0.14);
	\draw[-,thick] (0.1,-.3) to (-0.04,-0.06);
	\draw[-,line width=1pt] (0.085,.14) to (-0.035,-0.06);
	\draw[-,thick] (-0.2,-.3) to (0.07,0.14);
	\draw[-,line width=1.5pt] (0.08,.45) to (0.08,.1);
        \node at (0.45,-.41) {$\scriptstyle c$};
        \node at (0.07,-.4) {$\scriptstyle b$};
        \node at (-0.28,-.41) {$\scriptstyle a$};
\end{tikzpicture}
&=
\begin{tikzpicture}[baseline = 0, color=\clr]
	\draw[-,thick] (0.36,-.3) to (0.09,0.14);
	\draw[-,thick] (0.06,-.3) to (0.2,-.05);
	\draw[-,line width=.8pt] (0.07,.14) to (0.19,-.06);
	\draw[-,thick] (-0.19,-.3) to (0.08,0.14);
	\draw[-,line width=1.5pt] (0.08,.45) to (0.08,.1);
        \node at (0.45,-.41) {$\scriptstyle c$};
        \node at (0.07,-.4) {$\scriptstyle b$};
        \node at (-0.28,-.41) {$\scriptstyle a$};
\end{tikzpicture}\:,
\qquad
\begin{tikzpicture}[baseline = -1mm, color=\clr]
	\draw[-,thick] (0.35,.3) to (0.08,-0.14);
	\draw[-,thick] (0.1,.3) to (-0.04,0.06);
	\draw[-,line width=1pt] (0.085,-.14) to (-0.035,0.06);
	\draw[-,thick] (-0.2,.3) to (0.07,-0.14);
	\draw[-,line width=1.5pt] (0.08,-.45) to (0.08,-.1);
        \node at (0.45,.4) {$\scriptstyle c$};
        \node at (0.07,.42) {$\scriptstyle b$};
        \node at (-0.28,.4) {$\scriptstyle a$};
\end{tikzpicture}
=\begin{tikzpicture}[baseline = -1mm, color=\clr]
	\draw[-,thick] (0.36,.3) to (0.09,-0.14);
	\draw[-,thick] (0.06,.3) to (0.2,.05);
	\draw[-,line width=1pt] (0.07,-.14) to (0.19,.06);
	\draw[-,thick] (-0.19,.3) to (0.08,-0.14);
	\draw[-,line width=1.5pt] (0.08,-.45) to (0.08,-.1);
        \node at (0.45,.4) {$\scriptstyle c$};
        \node at (0.07,.42) {$\scriptstyle b$};
        \node at (-0.28,.4) {$\scriptstyle a$};
\end{tikzpicture}\:,
\\
\label{mergesplit}
\begin{tikzpicture}[baseline = 7.5pt,scale=.8, color=\clr]
	\draw[-,line width=1pt] (0,0) to (.28,.3) to (.28,.7) to (0,1);
	\draw[-,line width=1pt] (.6,0) to (.31,.3) to (.31,.7) to (.6,1);
        \node at (0,1.13) {$\scriptstyle b$};
        \node at (0.63,1.13) {$\scriptstyle d$};
        \node at (0,-.1) {$\scriptstyle a$};
        \node at (0.63,-.1) {$\scriptstyle c$};
\end{tikzpicture}
&=
\sum_{\substack{0 \leq s \leq \min(a,b)\\0 \leq t \leq \min(c,d)\\t-s=d-a}}
q^{st}  
\begin{tikzpicture}[baseline = 7.5pt,scale=.8, color=\clr]
\draw[-,thick] (0.58,0) to (0.58,.2) to (.02,.8) to (.02,1);
	\draw[-,line width=4pt,white] (0.02,0) to (0.02,.2) to (.58,.8) to (.58,1);
	\draw[-,thick] (0.02,0) to (0.02,.2) to (.58,.8) to (.58,1);
	\draw[-,thick] (0,0) to (0,1);
	\draw[-,thick] (0.6,0) to (0.6,1);
        \node at (0,1.13) {$\scriptstyle b$};
        \node at (0.6,1.13) {$\scriptstyle d$};
        \node at (0,-.1) {$\scriptstyle a$};
        \node at (0.6,-.1) {$\scriptstyle c$};
        \node at (-0.1,.5) {$\scriptstyle s$};
        \node at (0.75,.5) {$\scriptstyle t$};
\end{tikzpicture},
\\
  \label{splitbinomial}
\begin{tikzpicture}[baseline = -1mm,scale=.8,color=\clr]
\draw[-,line width=1.5pt] (0.08,-.8) to (0.08,-.5);
\draw[-,line width=1.5pt] (0.08,.3) to (0.08,.6);
\draw[-,thick] (0.1,-.51) to [out=45,in=-45] (0.1,.31);
\draw[-,thick] (0.06,-.51) to [out=135,in=-135] (0.06,.31);
\node at (-.33,-.05) {$\scriptstyle a$};
\node at (.45,-.05) {$\scriptstyle b$};
\end{tikzpicture}
&= 
\qbinomsm{a+b}{a}\:
\begin{tikzpicture}[baseline = -1mm,scale=.6,color=\clr]
	\draw[-,line width=1.5pt] (0.08,-.8) to (0.08,.6);
        \node at (.08,-1) {$\scriptstyle a+b$};
\end{tikzpicture}, 
\\
\label{dotmovecrossing}
\begin{tikzpicture}[baseline=-1mm, scale=.8, color=\clr]
	\draw[-,line width=1pt] (0.3,-.3) to (-.3,.4);
	\draw[-,line width=4pt,white] (-0.3,-.3) to (.3,.4);
	\draw[-,line width=1pt] (-0.3,-.3) to (.3,.4);
        \node at (-0.3,-.4) {$\scriptstyle a$};
        \node at (0.3,-.4) {$\scriptstyle b$};
     \draw(-0.15,-0.12)   \bdot;
\end{tikzpicture}
 & =
\begin{tikzpicture}[baseline=-1mm, scale=.8, color=\clr]
 \draw[-,line width=1pt] (-0.3,-.3) to (.3,.4);
	\draw[-,line width=4pt,white] (0.3,-.3) to (-.3,.4);
  \draw[-,line width=1pt] (0.3,-.3) to (-.3,.4);
        \node at (-0.3,-.4) {$\scriptstyle a$};
        \node at (0.3,-.4) {$\scriptstyle b$};
     \draw(0.15,0.23)   \bdot;
\end{tikzpicture}
,  
    \qquad
\begin{tikzpicture}[baseline=-1mm, scale=.8, color=\clr]
	\draw[-,line width=1pt] (0.3,-.3) to (-.3,.4);
	\draw[-,line width=4pt,white] (-0.3,-.3) to (.3,.4);
	\draw[-,line width=1pt] (-0.3,-.3) to (.3,.4);
        \node at (-0.3,-.4) {$\scriptstyle a$};
        \node at (0.3,-.4) {$\scriptstyle b$};
     \draw(-0.15,0.23)   \bdot;
\end{tikzpicture}
  =
\begin{tikzpicture}[baseline=-1mm, scale=.8, color=\clr]
 \draw[-,line width=1pt] (-0.3,-.3) to (.3,.4);
	\draw[-,line width=4pt,white] (0.3,-.3) to (-.3,.4);
  \draw[-,line width=1pt] (0.3,-.3) to (-.3,.4);
        \node at (-0.3,-.4) {$\scriptstyle a$};
        \node at (0.3,-.4) {$\scriptstyle b$};
     \draw(0.15,-0.12)   \bdot;
\end{tikzpicture}
,
\\
 \label{dotmovesplits+merge}
\begin{tikzpicture}[baseline = -.5mm,scale=.8,color=\clr]
\draw[-,line width=1.5pt] (0.08,-.5) to (0.08,0.04);
\draw[-,line width=1pt] (0.34,.5) to (0.08,0);
\draw[-,line width=1pt] (-0.2,.5) to (0.08,0);
\node at (-0.22,.6) {$\scriptstyle a$};
\node at (0.36,.65) {$\scriptstyle b$};
\draw (0.08,-.2) \bdot;
\end{tikzpicture} 
& =
\begin{tikzpicture}[baseline = -.5mm,scale=.8,color=\clr]
\draw[-,line width=1.5pt] (0.08,-.5) to (0.08,0.04);
\draw[-,line width=1pt] (0.34,.5) to (0.08,0);
\draw[-,line width=1pt] (-0.2,.5) to (0.08,0);
\node at (-0.22,.6) {$\scriptstyle a$};
\node at (0.36,.65) {$\scriptstyle b$};
\draw (-.05,.24) \bdot;
\draw (.22,.24) \bdot;
\end{tikzpicture},
   \quad 
\begin{tikzpicture}[baseline = -.5mm, scale=.8, color=\clr]
\draw[-,line width=1pt] (0.3,-.5) to (0.08,0.04);
\draw[-,line width=1pt] (-0.2,-.5) to (0.08,0.04);
\draw[-,line width=1.5pt] (0.08,.6) to (0.08,0);
\node at (-0.22,-.6) {$\scriptstyle a$};
\node at (0.35,-.6) {$\scriptstyle b$};
\draw (0.08,.2) \bdot;
\end{tikzpicture}
 =~
\begin{tikzpicture}[baseline = -.5mm,scale=.8, color=\clr]
\draw[-,line width=1pt] (0.3,-.5) to (0.08,0.04);
\draw[-,line width=1pt] (-0.2,-.5) to (0.08,0.04);
\draw[-,line width=1.5pt] (0.08,.6) to (0.08,0);
\node at (-0.22,-.6) {$\scriptstyle a$};
\node at (0.35,-.6) {$\scriptstyle b$};
\draw (-.08,-.3) \bdot; \draw (.22,-.3) \bdot;
\end{tikzpicture} ,
\\
  \label{intergralballon}
\begin{tikzpicture}[baseline = 1.5mm, scale=.5, color=\clr]
\draw[-, line width=1.2pt] (0.5,2) to (0.5,2.5);
\draw[-, line width=1.2pt] (0.5,0) to (0.5,-.4);
\draw[-,thin]  (0.5,2) to[out=left,in=up] (-.5,1)
to[out=down,in=left] (0.5,0);
\draw[-,thin]  (0.5,2) to[out=left,in=up] (0,1)
 to[out=down,in=left] (0.5,0);      
\draw[-,thin] (0.5,0)to[out=right,in=down] (1.5,1)
to[out=up,in=right] (0.5,2);
\draw[-,thin] (0.5,0)to[out=right,in=down] (1,1) to[out=up,in=right] (0.5,2);
\node at (0.5,.7){$\scriptstyle \cdots$};
\draw (-0.5,1) \bdot; 
\draw (0,1) \bdot; 
\node at (0.5,-.6) {$\scriptstyle a$};
\draw (1,1) \bdot;
\draw (1.5,1) \bdot; 
\node at (-.22,0) {$\scriptstyle 1$};
\node at (1.2,0) {$\scriptstyle 1$};
\node at (.3,0.3) {$\scriptstyle 1$};
\node at (.7,0.3) {$\scriptstyle 1$};
\end{tikzpicture}
   & =
   [a]!  
\begin{tikzpicture}[baseline = 1.5mm, scale=.5, color=\clr]
\draw[-,line width=1.2pt] (0,-0.1) to[out=up, in=down] (0,1.4);
\draw(0,0.6) \bdot; 
\node at (0,-.3) {$\scriptstyle a$};
\end{tikzpicture} ,
\qquad
\begin{tikzpicture}[baseline = 1.5mm, scale=.5, color=\clr]
\draw[-, line width=1.2pt] (0.5,2) to (0.5,2.5);
\draw[-, line width=1.2pt] (0.5,0) to (0.5,-.4);
\draw[-,thin]  (0.5,2) to[out=left,in=up] (-.5,1)
to[out=down,in=left] (0.5,0);
\draw[-,thin]  (0.5,2) to[out=left,in=up] (0,1)
 to[out=down,in=left] (0.5,0);      
\draw[-,thin] (0.5,0)to[out=right,in=down] (1.5,1)
to[out=up,in=right] (0.5,2);
\draw[-,thin] (0.5,0)to[out=right,in=down] (1,1) to[out=up,in=right] (0.5,2);
\node at (0.5,.7){$\scriptstyle \cdots$};
\draw (-0.5,1) \zdot; 
\draw (0,1) \zdot; 
\node at (0.5,-.6) {$\scriptstyle a$};
\draw (1,1) \zdot;
\draw (1.5,1) \zdot; 
\node at (-.22,0) {$\scriptstyle 1$};
\node at (1.2,0) {$\scriptstyle 1$};
\node at (.3,0.3) {$\scriptstyle 1$};
\node at (.7,0.3) {$\scriptstyle 1$};
\end{tikzpicture}
    ~=~ [a]!  
\begin{tikzpicture}[baseline = 1.5mm, scale=.5, color=\clr]
\draw[-,line width=1.2pt] (0,-0.1) to[out=up, in=down] (0,1.4);
\draw(0,0.6) \zdot; 
\node at (0,-.3) {$\scriptstyle a$};
\end{tikzpicture} .
\end{align}
\end{definition}

\begin{rem}
    A degenerate version of the affine $q$-web category $\qAW$ was introduced independently in \cite{DKMZ} and \cite{SW24web}. 
\end{rem}

The following relations in $\qAW$ follow by \eqref{dotmovecrossing} and the inverse relations between solid and hollow dots:
\begin{align}
\label{eq:zdotmovecrossing}
    \begin{tikzpicture}[baseline=-1mm, scale=.8, color=\clr]
 \draw[-,line width=1pt] (-0.3,-.3) to (.3,.4);
	\draw[-,line width=4pt,white] (0.3,-.3) to (-.3,.4);
  \draw[-,line width=1pt] (0.3,-.3) to (-.3,.4);
        \node at (-0.3,-.4) {$\scriptstyle a$};
        \node at (0.3,-.4) {$\scriptstyle b$};
     \draw(-0.18,-0.15)   \zdot;
\end{tikzpicture}
=
\begin{tikzpicture}[baseline=-1mm, scale=.8, color=\clr]
	\draw[-,line width=1pt] (0.3,-.3) to (-.3,.4);
	\draw[-,line width=4pt,white] (-0.3,-.3) to (.3,.4);
	\draw[-,line width=1pt] (-0.3,-.3) to (.3,.4);
        \node at (-0.3,-.4) {$\scriptstyle a$};
        \node at (0.3,-.4) {$\scriptstyle b$};
     \draw(0.15,0.23)   \zdot;
\end{tikzpicture} ,
    \qquad
\begin{tikzpicture}[baseline=-1mm, scale=.8, color=\clr]
 \draw[-,line width=1pt] (-0.3,-.3) to (.3,.4);
	\draw[-,line width=4pt,white] (0.3,-.3) to (-.3,.4);
  \draw[-,line width=1pt] (0.3,-.3) to (-.3,.4);
        \node at (-0.3,-.4) {$\scriptstyle a$};
        \node at (0.3,-.4) {$\scriptstyle b$};
      \draw(-0.15,0.23)   \zdot;
\end{tikzpicture} 
=
\begin{tikzpicture}[baseline=-1mm, scale=.8, color=\clr]
	\draw[-,line width=1pt] (0.3,-.3) to (-.3,.4);
	\draw[-,line width=4pt,white] (-0.3,-.3) to (.3,.4);
	\draw[-,line width=1pt] (-0.3,-.3) to (.3,.4);
        \node at (-0.3,-.4) {$\scriptstyle a$};
        \node at (0.3,-.4) {$\scriptstyle b$};
    \draw(0.18,-0.15)   \zdot;
\end{tikzpicture} .
\end{align}

Similarly, the following relations in $\qAW$, for $a,b\ge 1$, follow by \eqref{dotmovesplits+merge}:
\begin{equation}
\label{eq:dotSM2}
\begin{tikzpicture}[baseline = -.5mm,scale=.8,color=\clr]
\draw[-,line width=1.5pt] (0.08,-.5) to (0.08,0.04);
\draw[-,line width=1pt] (0.34,.5) to (0.08,0);
\draw[-,line width=1pt] (-0.2,.5) to (0.08,0);
\node at (-0.22,.6) {$\scriptstyle a$};
\node at (0.36,.65) {$\scriptstyle b$};
\draw (0.08,-.2) \zdot;
\end{tikzpicture} 
=
\begin{tikzpicture}[baseline = -.5mm,scale=.8,color=\clr]
\draw[-,line width=1.5pt] (0.08,-.5) to (0.08,0.04);
\draw[-,line width=1pt] (0.34,.5) to (0.08,0);
\draw[-,line width=1pt] (-0.2,.5) to (0.08,0);
\node at (-0.22,.6) {$\scriptstyle a$};
\node at (0.36,.65) {$\scriptstyle b$};
\draw (-.05,.24) \zdot;
\draw (.22,.24) \zdot;
\end{tikzpicture},
   \qquad 
\begin{tikzpicture}[baseline = -.5mm, scale=.8, color=\clr]
\draw[-,line width=1pt] (0.3,-.5) to (0.08,0.04);
\draw[-,line width=1pt] (-0.2,-.5) to (0.08,0.04);
\draw[-,line width=1.5pt] (0.08,.6) to (0.08,0);
\node at (-0.22,-.6) {$\scriptstyle a$};
\node at (0.35,-.6) {$\scriptstyle b$};
\draw (0.08,.2) \zdot;
\end{tikzpicture}
 =~
\begin{tikzpicture}[baseline = -.5mm,scale=.8, color=\clr]
\draw[-,line width=1pt] (0.3,-.5) to (0.08,0.04);
\draw[-,line width=1pt] (-0.2,-.5) to (0.08,0.04);
\draw[-,line width=1.5pt] (0.08,.6) to (0.08,0);
\node at (-0.22,-.6) {$\scriptstyle a$};
\node at (0.35,-.6) {$\scriptstyle b$};
\draw (-.08,-.3) \zdot; \draw (.22,-.3) \zdot;
\end{tikzpicture} .
\end{equation}

\subsection{The $q$-web category}

To be consistent with our general terminology, the $q$-Schur category $q$-{\bf Schur} in \cite{Bru24} will be referred to as the {\em $q$-web category} and denoted by $\qW$. More explicitly, $\qW$ is the strict $\kk$-linear monoidal category generated by objects $a\in \mathbb Z_{\ge1}$. The object $a$ and its identity morphism  will be drawn as a vertical strand labeled by $a$:
$
\stra .  
$
The generating morphisms are the merges, splits and (thick) positive crossings depicted as \eqref{merge+split+crossing}, subject to the relations \eqref{webassoc}--\eqref{splitbinomial}, where $\crossingpos$ is invertible with inverse denoted by $\crossingneg$. 

The following additional relations hold for $\qW$ (and hence also for $\qAW$) as they can be derived from the defining relations of $\qW$, for $a,b,c,d \geq 0$ 
with $d \leq a$ and $c \leq b+d$:
\begin{align}
\label{squareswitch}
 ,
\\ 
\label{cross via square}
%
 .
\label{braid}
\end{align}
For details see \cite{CKM},\cite[(6.11)-(6.17), Cor. 6.2]{Bru24} (note a typo in \cite[Cor. 6.2]{Bru24} is corrected in \eqref{cross2}).
Note that \eqref{cross2} in case $a=b=1$ is simply the standard quadratic relation for Hecke algebras:
\begin{align} \label{Hecke}
   \Big( 
      %
 . 
\end{align}
It is not always necessary to include $\crossingpos$ as generating morphisms in $\qW$ and $\qAW$ (as it can be expressed in terms of merges and splits as in \eqref{cross via square}) if one includes additional relations  \eqref{squareswitch}--\eqref{squareswitch2} as defining relations.

\subsection{Additional relations in $\qAW$}

We define the following morphisms in $\qAW$
\begin{equation}
\label{eq:wkdota}
\omega_{a,r} :=
%
  
\qquad\quad (0 \le r\le a).
\end{equation}
In particular, we have 
$
%
=\zdota$. 
We adopt the convention throughout the paper that
\[ 
%
.
 \]

\begin{lemma}
    \label{lem:rightdot}
    The following relations hold in $\qAW$:
    \begin{align*}
    %
\quad \text{ for } a,b \ge 0.
    \end{align*}
\end{lemma}

\begin{proof}
Using \eqref{unfold} and \eqref{dotmovecrossing} we have that
\begin{align*}
        %
 .
\end{align*}
The first formula follows. The second formula is proved similarly using \eqref{eq:zdotmovecrossing}. 
\end{proof}

\begin{lemma}
\label{lem:dots2strands}
  The following relations hold in $\qAW$:
 \begin{align}
\label{dots 2str}
%
, \qquad \text{ for } 0\le r \le a, 0\le u \le b. 
\end{align}
In particular, we have
 \begin{align}
\label{splitmerge}
%
\end{align*}
which is equal to the RHS of \eqref{dots 2str} as desired  by applying \eqref{eq:wkdota} again. 
\end{proof}

\begin{lemma} 
\label{lem:dotmovesplitabr}
  The following relations hold in $\qAW$, for $0\le r \le a+b$:
 \begin{align}
 \label{dotmovesplitabr}
 \begin{split}
%
 .
\end{split}
\end{align}
The same formulas hold if we replace   %
 everywhere.  
\end{lemma}

\begin{proof}
  We provide a proof for the first formula only, as the second follows by an entirely similar argument. We have 
    \begin{align*}
%
.
\end{align*}
The lemma is proved. 
\end{proof}

\begin{lemma}
 \label{lem:dotcommutative}
The following relations hold in $\qAW$, for $1\le r, t \le a$: 
    \begin{equation*}
    %
.
\end{equation*}
\end{lemma}

\begin{proof}
We only prove the first identity, as the remaining two can be checked similarly. 
We prove by induction on $a+t+r$. The base case $a=r=t=1$ (or when $r=0$ or $t=0$) is clear.

We can assume that $1\le r \le t \le a$, but $r,t,a$ not all equal (otherwise, it trivially follows by definition). We have
    \begin{align*}
    %
 
\qquad \text{ by induction on $s+2r$}
\\
&\overset{\eqref{dotmovesplitabr}}{=}
    %
 .
\end{align*}
The lemma is proved. 
\end{proof}

\begin{lemma}
    \label{lem:cross+=-}
    The following relation holds in $\qAW$, for $a, b\ge 1$: 
\begin{equation*}
   %
 .
    \end{equation*}
\end{lemma}

\begin{proof}
Composing both sides of the double crossing formula \eqref{cross2} with $\crossingneg$ on the bottom, we have
\begin{equation*}
    %
 .
\end{equation*}
The last diagram above can be further simplified:
\begin{equation*}
 %
 .
\end{equation*}
Plugging this formula back to the earlier computation proves the lemma. 
\end{proof}

The following lemma extends the relation \eqref{dotmovecrossing}, which corresponds to the special case when $r = a$  below.

\begin{lemma}
    \label{lem:wrdotcross}
    The following relations hold in $\qAW$, for $0\le r \le a$: 
     \begin{align*}
%
.
\end{align*} 
\end{lemma}

\begin{proof}
We only prove the first formula, as the second one is proved similarly. We have
    \begin{align*}
        %
 \quad \text{ by  Lemma~\ref{lem:cross+=-}.}
\end{align*}
The last diagram can be simplified as follows:
\begin{align*}
    %
 .
\end{align*}
Plugging this formula back into the earlier computation proves the first formula in the lemma. 
\end{proof}

For any $a\in \Z_{>0}$, 
    let
    \begin{equation}
    \label{eq:RPar}
          \RPar_a:=\{\lambda=(\lambda_1,\ldots,\lambda_k)\in \Z^k\mid \lambda_1\ge \lambda_2\ge \ldots\ge \lambda_k, |\lambda_i|\le a,  \lambda_1-\lambda_k\le a, \forall 1\le i\le k \}.
    \end{equation}

We call any element in $ \RPar_a$ a rational partition.
 Let $\Par_a$ denote the subset of $\RPar_a$ consisting of partitions $\nu=(\nu_1,\ldots,\nu_k)$ such that $\nu_1\le a$.

For any $\lambda\in \RPar_a $, we define 
\begin{equation}
\label{eq:wkdotalambda}
\begin{tikzpicture}[baseline = -1mm,scale=1,color=\clr]
\draw[-,line width=2pt] (0.08,-.5) to (0.08,.5);
\node at (.08,-.65) {$\scriptstyle a$};
\draw(0.08,0) \bdot;
\draw(.45,0)node {$\scriptstyle \omega_{\lambda}$};\end{tikzpicture}
 := \begin{tikzpicture}[baseline = -1mm,scale=1,color=\clr]
\draw[-,line width=2pt] (0.08,-.7) to (0.08,.7);
\node at (.08,-.85) {$\scriptstyle a$};
\draw(0.08,.5) \bdot;
\draw(.55,.5)node {$\scriptstyle \omega_{\lambda_1}$};
\draw(0.08,.3) \bdot;
\draw(.55,.3)node {$\scriptstyle \omega_{\lambda_2}$};
\draw(.55,0) node{$\ldots$};
\draw(0.08,-.3) \bdot;
\draw(.55,-.3)node {$\scriptstyle \omega_{\lambda_k}$};
\end{tikzpicture}.   
\end{equation}
We may also denote the above diagram by the shorthand notation $\omega_{a,\la}$ or just $\omega_{\la}$ if there is no confusion on $a$. 
\begin{definition}
\label{def:Da}
    For any $a\in \N$, define 
    $D_a$ to be the (commutative) $\Z[q,q^{-1}]$-subalgebra of $ \End_{\qAW}(\stra)$ generated by $\omega_{a,r}$, for $-a\le r \le a$.
(The commutativity here follows by Lemma~\ref{lem:dotcommutative}.)
\end{definition}
It will be shown in \cref{cor:End1a} that $D_a = \End_{\qAW}(\stra)$; moreover, it is isomorphic to the ring of symmetric Laurent polynomials $\Z[q,q^{-1}] [X_1^\pm, \ldots, X_a^\pm]^{\mathfrak S_a}$.

\begin{lemma} \label{lem:spanDa}
     As a $\Z[q,q^{-1}]$-module, $D_a$ is
spanned by $\omega_{a,\lambda}$, for all $\lambda\in \RPar_a$.
\end{lemma}

\begin{proof}
    We first make the following.
    
\underline{Claim $(\star)$}. Let $a\ge b \ge 0$. For $r,s$ such that $|r|\le c$ and $|s| \le b$, we have 
    \[
    \begin{tikzpicture}
[baseline = -1mm,scale=.8,color=\clr]
\draw[-,line width=1.5pt] (0.08,-.8) to (0.08,-.5);
\draw[-,line width=1.5pt] (0.08,.3) to (0.08,.6);
\draw[-,thick] (0.1,-.51) to [out=45,in=-45] (0.1,.31);
\draw[-,thick] (0.06,-.51) to [out=135,in=-135] (0.06,.31);
\draw(-.1,0) \bdot;
\draw(-.4,0)node {$\scriptstyle \omega_r$};
\draw(.25,0) \bdot;
\draw(.7,0)node {$\scriptstyle \omega_{-s}$};
\node at (.1,-.95) {$\scriptstyle a$};
\node at (.42,-.35) {$\scriptstyle b$};
\end{tikzpicture}
\in 
\Omega_{a}:=\kk\text{-span}\{ \omega_{a,h}\omega_{a,-k} \mid h,k\ge 0, h+k\le a \} .
\]
The claim for $r\ge 0\ge s$ follows by \eqref{dots 2str}, and the claim for $r\le 0 \le s$ follows by a counterpart of \eqref{dots 2str}. It remains to prove the claim for $r,s \ge 0$ by induction on $r+s$. The cases with either $r=0$ or $s=0$ hold trivially. 
In general, for $r,s>0$, we have 
\begin{align*}
      \begin{tikzpicture}
[baseline = -1mm,scale=.8,color=\clr]
\draw[-,line width=1.5pt] (0.08,-.8) to (0.08,-.5);
\draw[-,line width=1.5pt] (0.08,.3) to (0.08,.6);
\draw[-,thick] (0.1,-.51) to [out=45,in=-45] (0.1,.31);
\draw[-,thick] (0.06,-.51) to [out=135,in=-135] (0.06,.31);
\draw(-.1,0) \bdot;
\draw(-.4,0)node {$\scriptstyle \omega_r$};
\draw(.25,0) \bdot;
\draw(.7,0)node {$\scriptstyle \omega_{-s}$};
\node at (.1,-.95) {$\scriptstyle a$};
\node at (.42,-.35) {$\scriptstyle b$};
\end{tikzpicture}
&=\qbinomsm{a-r-s}{b-s}
\begin{tikzpicture}
[baseline = -1mm,scale=.8,color=\clr]
\draw[-,line width=1.5pt] (0.08,-.8) to (0.08,-.5);
\draw[-,line width=1.5pt] (0.08,.3) to (0.08,.6);
\draw[-,thick] (0.1,-.51) to [out=45,in=-45] (0.1,.31);
\draw[-,thick] (0.06,-.51) to [out=135,in=-135] (0.06,.31);
\draw(-.1,0) \bdot;
\draw(-.4,0)node {$\scriptstyle \omega_r$};
\draw(.25,0) \zdot;
\node at (.42,-.35) {$\scriptstyle s$};
\node at (.1,-.95) {$\scriptstyle a$};
\end{tikzpicture}
\overset{\eqref{dotmovesplitabr}}=
\qbinomsm{a-r-s}{b-s} 
\Big( q^*\omega_{a,r}\omega_{a,-s}
-
\sum_{0\le t<r }
q^\star\begin{tikzpicture}
[baseline = -1mm,scale=.8,color=\clr]
\draw[-,line width=1.5pt] (0.08,-.8) to (0.08,-.5);
\draw[-,line width=1.5pt] (0.08,.3) to (0.08,.6);
\draw[-,thick] (0.1,-.51) to [out=45,in=-45] (0.1,.31);
\draw[-,thick] (0.06,-.51) to [out=135,in=-135] (0.06,.31);
\draw(-.1,0) \bdot;
\draw(-.4,0)node {$\scriptstyle \omega_t$};
\draw(.25,-.15) \zdot;
\draw(.25,.1) \bdot;
\draw(.8,0.1)node {$\scriptstyle \omega_{r-t}$};
\node at (.42,-.35) {$\scriptstyle s$};
\node at (.1,-.95) {$\scriptstyle a$};
\end{tikzpicture} \Big)
\\
&=\qbinomsm{a-r-s}{b-s}\Big( q^*\omega_{a,r}\omega_{a,-s}
-
\sum_{0\le t<r }
q^\star\begin{tikzpicture}
[baseline = -1mm,scale=.8,color=\clr]
\draw[-,line width=1.5pt] (0.08,-.8) to (0.08,-.5);
\draw[-,line width=1.5pt] (0.08,.3) to (0.08,.6);
\draw[-,thick] (0.1,-.51) to [out=45,in=-45] (0.1,.31);
\draw[-,thick] (0.06,-.51) to [out=135,in=-135] (0.06,.31);
\draw(-.1,0) \bdot;
\draw(-.4,0)node {$\scriptstyle \omega_t$};
\draw(.25,0) \bdot;
\draw(.94,0.1)node {$\scriptstyle \omega_{r-t-s}$};
\node at (.42,-.35) {$\scriptstyle s$};
\node at (.1,-.95) {$\scriptstyle a$};
\end{tikzpicture} \Big)
\end{align*}
where $q^*, q^\star$ are some powers of $q$. Then the claim holds since each summand of the summation
  is contained in $\Omega_{a}$ by using \eqref{dots 2str} if $r-t-s\ge 0$ or by induction assumption on $r+s$ if $r-t-s<0$ by noting that $s+2t-r<s+r $.
  This completes the proof of Claim ($\star$).

Next, we claim that
  \begin{align} \label{eq:omega2}
      \omega_{a,r}\omega_{a,-s}\in \Omega_a, 
      \quad \text{ for }0\le r,s\le a.
  \end{align}
The claim is trivial when $r+s\le a$. For $0\le r,s\le a$ with $r+s> a$, the claim \eqref{eq:omega2} follows from 
\begin{align*}
    \omega_{a,r}\omega_{a,-s}= \begin{tikzpicture}
[baseline = -1mm,scale=.8,color=\clr]
\draw[-,line width=1.5pt] (0.08,-.8) to (0.08,-.5);
\draw[-,line width=1.5pt] (0.08,.3) to (0.08,.6);
\draw[-,thick] (0.1,-.51) to [out=45,in=-45] (0.1,.31);
\draw[-,thick] (0.06,-.51) to [out=135,in=-135] (0.06,.31);
\draw(.08,0.45) \bdot;
\draw(-.25,.4)node {$\scriptstyle \omega_r$};
\draw(.25,0) \zdot;
\node at (.42,-.35) {$\scriptstyle s$};
\end{tikzpicture}
\overset{\eqref{dotmovesplitabr}}{\in}
\sum_{1\le t\le \min(r,s)}  q^\Z 
\begin{tikzpicture}
[baseline = -1mm,scale=.8,color=\clr]
\draw[-,line width=1.5pt] (0.08,-.8) to (0.08,-.5);
\draw[-,line width=1.5pt] (0.08,.3) to (0.08,.6);
\draw[-,thick] (0.1,-.51) to [out=45,in=-45] (0.1,.31);
\draw[-,thick] (0.06,-.51) to [out=135,in=-135] (0.06,.31);
\draw(-.1,0) \bdot;
\draw(-.6,0)node {$\scriptstyle \omega_{r-t}$};
\draw(.25,0) \bdot;
\draw(.8,0)node {$\scriptstyle \omega_{t-s}$};
\node at (.42,-.35) {$\scriptstyle s$};
\end{tikzpicture} 
\end{align*}
and then applying Claim ($\star$). 

By repeatedly applying \eqref{eq:omega2}, we see any monomial in $\{\omega_{a,r} \mid -a\le r \le a\}$ can be written as a $\kk$-linear combination of $\omega_{a,\lambda}$, for $\lambda\in \RPar_a$.
\end{proof}

For $k\ge 0$, we define
\[ 
D_{a}^{+,k}=\kk\text{-span}\{\omega_{a,\eta}\in D_{a}\mid \eta\in 
\Par_a, l(\eta)\le k \}. 
\]

\begin{lemma} 
\label{lem:blamu}
For any $\lambda \in \RPar_a$
and $\mu\in \RPar_b$ we have 
$
\begin{tikzpicture}
[baseline = -1mm,scale=.8,color=\clr]
\draw[-,line width=1.5pt] (0.08,-.8) to (0.08,-.5);
\draw[-,line width=1.5pt] (0.08,.3) to (0.08,.6);
\draw[-,thick] (0.1,-.51) to [out=45,in=-45] (0.1,.31);
\draw[-,thick] (0.06,-.51) to [out=135,in=-135] (0.06,.31);
\draw(-.1,0) \bdot;
\draw(-.4,0)node {$\scriptstyle \omega_\lambda$};
\draw(.25,0) \bdot;
\draw(.6,0)node {$\scriptstyle \omega_\mu$};
\node at (-.22,-.35) {$\scriptstyle a$};
\node at (.42,-.35) {$\scriptstyle b$};
\end{tikzpicture}
\in D_{a+b}$. Moreover, if $\lambda,\mu\in \Par_a$, then $
\begin{tikzpicture}
[baseline = -1mm,scale=.8,color=\clr]
\draw[-,line width=1.5pt] (0.08,-.8) to (0.08,-.5);
\draw[-,line width=1.5pt] (0.08,.3) to (0.08,.6);
\draw[-,thick] (0.1,-.51) to [out=45,in=-45] (0.1,.31);
\draw[-,thick] (0.06,-.51) to [out=135,in=-135] (0.06,.31);
\draw(-.1,0) \bdot;
\draw(-.4,0)node {$\scriptstyle \omega_\lambda$};
\draw(.25,0) \bdot;
\draw(.6,0)node {$\scriptstyle \omega_\mu$};
\node at (-.22,-.35) {$\scriptstyle a$};
\node at (.42,-.35) {$\scriptstyle b$};
\end{tikzpicture}
\in D_{a+b}^{+,k}$ where $k=\max\{l(\lambda),l(\mu)\}$. 
\end{lemma}

\begin{proof}
It follows from \eqref{dotmovesplitabr} by the same type inductive argument as for \cite[Lemma~ 2.13]{SW24web}. We skip the detail. 
\end{proof}

%

\subsection{Category $\qAWC$ over $\C(q)$}

We now introduce a $\C(q)$-linear monoidal category $\qAWC$, a simpler variant of $\qAW$. 

\begin{definition}
  \label{def:qAWebC}
The affine $q$-web category $\qAWC$ is the strict $\C(q)$-linear monoidal category generated by objects $a\in \mathbb Z_{\ge1}$. The object $a$ and its identity morphism  will be drawn as a vertical strand labeled by $a$, i.e. 
$
\stra .  
$
The generating morphisms are the merges, splits and (thick) positive crossings depicted as
\begin{align}
\label{merge+split+crossingC}
\merge 
&:(a,b) \rightarrow (a+b),&
\splits
&:(a+b)\rightarrow (a,b),&
\crossingpos
&:(a,b) \rightarrow (b,a),
\end{align}
and 
\begin{equation}
\label{dotgenerator1}
\dotgen, 
\end{equation}
for $a,b \in \Z_{\ge 1}$, where $\crossingpos$ is invertible with inverse being the negative crossing denoted $\crossingneg := \crossingpos^{-1}$ and $\dotgen$ is invertible with inverse $\zdotgen$, 
subject to the following relations \eqref{webassocC}--\eqref{dotmovecrossingC}, for $a,b,c,d \in \Z_{\ge 1}$ with $d-a=c-b$:
\begin{align}
\label{webassocC}
\begin{tikzpicture}[baseline = 0,color=\clr]
	\draw[-,thick] (0.35,-.3) to (0.08,0.14);
	\draw[-,thick] (0.1,-.3) to (-0.04,-0.06);
	\draw[-,line width=1pt] (0.085,.14) to (-0.035,-0.06);
	\draw[-,thick] (-0.2,-.3) to (0.07,0.14);
	\draw[-,line width=1.5pt] (0.08,.45) to (0.08,.1);
        \node at (0.45,-.41) {$\scriptstyle c$};
        \node at (0.07,-.4) {$\scriptstyle b$};
        \node at (-0.28,-.41) {$\scriptstyle a$};
\end{tikzpicture}
&=
\begin{tikzpicture}[baseline = 0, color=\clr]
	\draw[-,thick] (0.36,-.3) to (0.09,0.14);
	\draw[-,thick] (0.06,-.3) to (0.2,-.05);
	\draw[-,line width=1pt] (0.07,.14) to (0.19,-.06);
	\draw[-,thick] (-0.19,-.3) to (0.08,0.14);
	\draw[-,line width=1.5pt] (0.08,.45) to (0.08,.1);
        \node at (0.45,-.41) {$\scriptstyle c$};
        \node at (0.04,-.4) {$\scriptstyle b$};
        \node at (-0.28,-.41) {$\scriptstyle a$};
\end{tikzpicture}\:,
\qquad
\begin{tikzpicture}[baseline = -1mm, color=\clr]
	\draw[-,thick] (0.35,.3) to (0.08,-0.14);
	\draw[-,thick] (0.1,.3) to (-0.04,0.06);
	\draw[-,line width=1pt] (0.085,-.14) to (-0.035,0.06);
	\draw[-,thick] (-0.2,.3) to (0.07,-0.14);
	\draw[-,line width=1.5pt] (0.08,-.45) to (0.08,-.1);
        \node at (0.45,.4) {$\scriptstyle c$};
        \node at (0.07,.42) {$\scriptstyle b$};
        \node at (-0.28,.4) {$\scriptstyle a$};
\end{tikzpicture}
=\begin{tikzpicture}[baseline = -1mm, color=\clr]
	\draw[-,thick] (0.36,.3) to (0.09,-0.14);
	\draw[-,thick] (0.06,.3) to (0.2,.05);
	\draw[-,line width=1pt] (0.07,-.14) to (0.19,.06);
	\draw[-,thick] (-0.19,.3) to (0.08,-0.14);
	\draw[-,line width=1.5pt] (0.08,-.45) to (0.08,-.1);
        \node at (0.45,.4) {$\scriptstyle c$};
        \node at (0.07,.42) {$\scriptstyle b$};
        \node at (-0.28,.4) {$\scriptstyle a$};
\end{tikzpicture}\:,
\\
\label{mergesplitC}
\begin{tikzpicture}[baseline = 7.5pt,scale=.8, color=\clr]
	\draw[-,line width=1pt] (0,0) to (.28,.3) to (.28,.7) to (0,1);
	\draw[-,line width=1pt] (.6,0) to (.31,.3) to (.31,.7) to (.6,1);
        \node at (0,1.13) {$\scriptstyle b$};
        \node at (0.63,1.13) {$\scriptstyle d$};
        \node at (0,-.1) {$\scriptstyle a$};
        \node at (0.63,-.1) {$\scriptstyle c$};
\end{tikzpicture}
&=
\sum_{\substack{0 \leq s \leq \min(a,b)\\0 \leq t \leq \min(c,d)\\t-s=d-a}}
q^{st}  
\begin{tikzpicture}[baseline = 7.5pt,scale=.8, color=\clr]
\draw[-,thick] (0.58,0) to (0.58,.2) to (.02,.8) to (.02,1);
	\draw[-,line width=4pt,white] (0.02,0) to (0.02,.2) to (.58,.8) to (.58,1);
	\draw[-,thick] (0.02,0) to (0.02,.2) to (.58,.8) to (.58,1);
	\draw[-,thick] (0,0) to (0,1);
	\draw[-,thick] (0.6,0) to (0.6,1);
        \node at (0,1.13) {$\scriptstyle c$};
        \node at (0.6,1.13) {$\scriptstyle d$};
        \node at (0,-.1) {$\scriptstyle a$};
        \node at (0.57,-.1) {$\scriptstyle b$};
        \node at (-0.1,.5) {$\scriptstyle s$};
        \node at (0.75,.5) {$\scriptstyle t$};
\end{tikzpicture},
\\
  \label{splitbinomialC}
\begin{tikzpicture}[baseline = -1mm,scale=.8,color=\clr]
\draw[-,line width=1.5pt] (0.08,-.8) to (0.08,-.5);
\draw[-,line width=1.5pt] (0.08,.3) to (0.08,.6);
\draw[-,thick] (0.1,-.51) to [out=45,in=-45] (0.1,.31);
\draw[-,thick] (0.06,-.51) to [out=135,in=-135] (0.06,.31);
\node at (-.33,-.05) {$\scriptstyle a$};
\node at (.45,-.05) {$\scriptstyle b$};
\end{tikzpicture}
& = 
\qbinomsm{a+b}{a}\:
\begin{tikzpicture}[baseline = -1mm,scale=.6,color=\clr]
	\draw[-,line width=1.5pt] (0.08,-.8) to (0.08,.6);
        \node at (.08,-1) {$\scriptstyle a+b$};
\end{tikzpicture}, 
\\
\label{dotmovecrossingC}
\begin{tikzpicture}[baseline=-1mm, color=\clr]
	\draw[-,thick] (0.3,-.3) to (-.3,.4);
	\draw[-,line width=4pt,white] (-0.3,-.3) to (.3,.4);
	\draw[-,thick] (-0.3,-.3) to (.3,.4);
        \node at (-0.3,-.45) {$\scriptstyle 1$};
        \node at (0.3,-.45) {$\scriptstyle 1$};
     \draw(-0.15,-0.12)   \bdot;
\end{tikzpicture}
 & =
\begin{tikzpicture}[baseline=-1mm, color=\clr]
 \draw[-,thick] (-0.3,-.3) to (.3,.4);
	\draw[-,line width=4pt,white] (0.3,-.3) to (-.3,.4);
  \draw[-,thick] (0.3,-.3) to (-.3,.4);
        \node at (-0.3,-.45) {$\scriptstyle 1$};
        \node at (0.3,-.45) {$\scriptstyle 1$};
     \draw(0.15,0.23)   \bdot;
\end{tikzpicture} ,  
    \qquad
\begin{tikzpicture}[baseline=-1mm, color=\clr]
	\draw[-,thick] (0.3,-.3) to (-.3,.4);
	\draw[-,line width=4pt,white] (-0.3,-.3) to (.3,.4);
	\draw[-,thick] (-0.3,-.3) to (.3,.4);
        \node at (-0.3,-.45) {$\scriptstyle 1$};
        \node at (0.3,-.45) {$\scriptstyle 1$};
     \draw(-0.15,0.23)   \bdot;
\end{tikzpicture}
  =
\begin{tikzpicture}[baseline=-1mm, color=\clr]
 \draw[-,thick] (-0.3,-.3) to (.3,.4);
	\draw[-,line width=4pt,white] (0.3,-.3) to (-.3,.4);
  \draw[-,thick] (0.3,-.3) to (-.3,.4);
        \node at (-0.3,-.45) {$\scriptstyle 1$};
        \node at (0.3,-.45) {$\scriptstyle 1$};
     \draw(0.15,-0.12)   \bdot;
\end{tikzpicture} .
\end{align} 
\end{definition}

We define the following morphisms in $\qAWC$:
\begin{align}  \label{eq:dota}
\xdota := \frac{1}{[a]!}
\begin{tikzpicture}[baseline = 1.5mm, scale=.5, color=\clr]
\draw[-, line width=1.2pt] (0.5,2) to (0.5,2.5);
\draw[-, line width=1.2pt] (0.5,0) to (0.5,-.4);
\draw[-,thin]  (0.5,2) to[out=left,in=up] (-.5,1)
to[out=down,in=left] (0.5,0);
\draw[-,thin]  (0.5,2) to[out=left,in=up] (0,1)
 to[out=down,in=left] (0.5,0);      
\draw[-,thin] (0.5,0)to[out=right,in=down] (1.5,1)
to[out=up,in=right] (0.5,2);
\draw[-,thin] (0.5,0)to[out=right,in=down] (1,1) to[out=up,in=right] (0.5,2);
\node at (0.5,.7){$\scriptstyle \cdots$};
\draw (-0.5,1) \bdot; 
\draw (0,1) \bdot; 
\node at (0.5,-.6) {$\scriptstyle a$};
\draw (1,1) \bdot;
\draw (1.5,1) \bdot; 
\node at (-.22,0) {$\scriptstyle 1$};
\node at (1.2,0) {$\scriptstyle 1$};
\node at (.3,0.3) {$\scriptstyle 1$};
\node at (.7,0.3) {$\scriptstyle 1$};
\end{tikzpicture} ,
\qquad\qquad
\zdota := \frac{1}{[a]!}
\begin{tikzpicture}[baseline = 1.5mm, scale=.5, color=\clr]
\draw[-, line width=1.2pt] (0.5,2) to (0.5,2.5);
\draw[-, line width=1.2pt] (0.5,0) to (0.5,-.4);
\draw[-,thin]  (0.5,2) to[out=left,in=up] (-.5,1)
to[out=down,in=left] (0.5,0);
\draw[-,thin]  (0.5,2) to[out=left,in=up] (0,1)
 to[out=down,in=left] (0.5,0);      
\draw[-,thin] (0.5,0)to[out=right,in=down] (1.5,1)
to[out=up,in=right] (0.5,2);
\draw[-,thin] (0.5,0)to[out=right,in=down] (1,1) to[out=up,in=right] (0.5,2);
\node at (0.5,.7){$\scriptstyle \cdots$};
\draw (-0.5,1) \zdot; 
\draw (0,1) \zdot; 
\node at (0.5,-.6) {$\scriptstyle a$};
\draw (1,1) \zdot;
\draw (1.5,1) \zdot; 
\node at (-.22,0) {$\scriptstyle 1$};
\node at (1.2,0) {$\scriptstyle 1$};
\node at (.3,0.3) {$\scriptstyle 1$};
\node at (.7,0.3) {$\scriptstyle 1$};
\end{tikzpicture}. 
\end{align}

\subsection{Isomorphism of $\qAWC$ and $\qAW$ over $\C(q)$}

We allow the $\xdota$ and $\zdota$ (for $a\ge 1$) defined above to be included as generating morphisms in $\qAWC$ for the formulation of the theorem below. 

\begin{theorem}
\label{th:qAWiso}
    The $\C(q)$-linear monoidal categories $\qAWC$ and $\qAW$ are isomorphic by matching the generating objects and generating morphisms in the same symbols. 
\end{theorem}

\begin{proof}
    There is a natural functor $\Psi: \qAWC \rightarrow \qAW$ which matches the generating objects and generating morphisms (see \eqref{eq:dota}) in the same symbols. 
    
    The main difference between $\qAWC$ and $\qAW$ is that the defining relations \eqref{webassocC}--\eqref{dotmovecrossingC} for $\qAWC$ are weaker than the definition relations \eqref{webassoc}--\eqref{intergralballon} for $\qAW$. The relations \eqref{webassocC}--\eqref{splitbinomialC} are the same as \eqref{webassoc}--\eqref{splitbinomial}, and \eqref{dotmovecrossingC} is a special case of \eqref{dotmovecrossing}. 

    To show that $\Psi$ is an isomorphism, it remains to show that the other defining relations \eqref{dotmovecrossing}--\eqref{intergralballon} hold for $\qAWC$. Indeed, \eqref{intergralballon} follows by definition \eqref{eq:dota}. The other two relations \eqref{dotmovecrossing} and \eqref{dotmovesplits+merge} will be proved in Lemma~\ref{lem:dotcross} and Lemma~\ref{lem:dotSM} below.

    This completes the proof of the theorem.
\end{proof}

\begin{lemma}
    \label{lem:dotSM}
       The relations \eqref{dotmovesplits+merge} hold in $\qAWC$, that is, 
 \begin{equation*}
 .
\end{equation*}
\end{lemma}

\begin{proof}
We will only prove the first identity, as the proof of the second identity is entirely similar. 
    We proceed by induction on $a+b$, with the base case $a+b=1$ being trivial. Then we have 
\begin{align*}
[a+b] 
%
  \quad \text{ by induction on $a+b-1$}
\\
&
\overset{\eqref{webassocC}\eqref{splitbinomialC}}{\underset{\eqref{dotmovecrossingC}}{=}}
q^b [a] %
 .
\end{align*}
The lemma is proved. 
\end{proof}

\begin{lemma}
\label{lem:inversedot}
  The following relations hold in $\qAWC$, for $a\ge 1$:
    \begin{equation*}
    %
 .
    \end{equation*}
\end{lemma}

\begin{proof}
We will only prove the first formula. We proceed by an induction on $a$, where the base case for $a=1$ follows by definition of $\qAWC$. 

For $a\ge 2$, we have
\begin{equation*}
        [a]^2
    %
 ,
\end{equation*}
where, in the last two equations above, we have applied Lemmas~\ref{lem:dotSM} and \eqref{eq:dotSM2} and the inductive assumptions (for $a=1$ and for $a-1$) to cancel out the solid/open dots.

The second summand in the above formula can be further simplified as follows:
\begin{equation*}
%
 .
\end{equation*}
Plugging this formula to the calculation above and noting that $q^{a-1}[a] + q^{-1} [a][a-1] =[a]^2$, we have proved the first formula in the lemma.
\end{proof}

\begin{lemma}
    \label{lem:dotcross}
    The relations \eqref{dotmovecrossing} hold in $\qAWC$, that is, 
     \begin{equation*}
%
 .
\end{equation*} 
\end{lemma}

\begin{proof}
    We will prove the first identity only as the proof for the second identity is entirely similar. The third and fourth identities follow from the first two by Lemma~\ref{lem:inversedot}.

    We first verify the case for $a=1$ by induction on $b$ with the base case $b=1$ being \eqref{dotmovecrossingC}. For $b>1$, this follows by a direct computation: 
\begin{equation*}
%
  , 
\end{equation*}
where we have applied the inductive assumptions (for $b=1$ and for $b-1$) in the second equation above.

Now we proceed to prove the general case by an induction on $a$. Let $a>1$. Applying Lemma~\ref{lem:dotSM} and the inductive assumptions (for $a=1$ and for $a-1$), we compute that 
\begin{equation*}
   %
 .
\end{equation*}
The lemma is proved. 
\end{proof}

\section{The affine $q$-Schur category}
\label{sec:Schur}

In this section, we enlarge $\qAW$ to a new $\kk$-linear monoidal category, the affine $q$-Schur category $\qASch$, with additional generating morphisms/relations, where $\kk =\Z[q,q^{-1}]$ or $\C(q)$. We also formulate a $\C(q)$-linear variant $\qASchC$ with fewer generating morphisms/relations, and show that it is isomorphic to $\qASch$ over $\C(q)$. Various additional relations in $\qASch$ are derived. 

\subsection{Definition of $\qASch$}

Let $\kk =\Z[q,q^{-1}]$ or $\C(q)$.

\begin{definition}
\label{def:qSchur}
    The affine $q$-Schur category  $\qASch$ is a strict $\kk$-linear monoidal category with generating objects $a\in \Z_{\ge 1}$ and $u\in \kk$. We denote the object $a\in \N$ by $\stra$ and denote the object $u \in \kk$ by a red strand labeled by $u$ as $\stru .$
The morphisms are generated by 
\begin{equation} 
\label{generator-affschur}
  \begin{tikzpicture}[anchorbase, scale=0.4, color=\clr] 
                \draw[-,line width=1pt] (-2,0.2) to (-1.5,1);
	\draw[-,line width=1pt ] (-1,0.2) to (-1.5,1);
	\draw[-,line width=1.5pt] (-1.5,1) to (-1.5,1.8);
                \draw (-1,0) node{$\scriptstyle {b}$};
                \draw (-2,0) node{$\scriptstyle {a}$};
                \draw (-1.5,2.2) node{$\scriptstyle {a+b}$};
                \end{tikzpicture}, 
                \quad 
  \begin{tikzpicture}[anchorbase, scale=0.4, color=\clr]
                \draw[-,line width=1pt] (-2,1.8) to (-1.5,1);
	\draw[-,line width=1pt ] (-1,1.8) to (-1.5,1);
	\draw[-,line width=1.5pt] (-1.5,1) to (-1.5,0.2);
                \draw (-1,2.2) node{$\scriptstyle {b}$};
                \draw (-2,2.2) node{$\scriptstyle {a}$};
                \draw (-1.5,0) node{$\scriptstyle {a+b}$};
                \end{tikzpicture}, 
    \quad  \;
    \crossingpos , 
        \quad \;
    ~\xdota ,
 \end{equation}
and
\begin{equation}
\label{redcrossing}
\rightcrossing 
\;\; (\text{traverse-up}), 
                \qquad
                \leftcrossing 
 \;\; (\text{traverse-down}),    
\end{equation}
subject to relations \eqref{webassoc}--\eqref{intergralballon} for the generators in \eqref{generator-affschur} and additional relations \eqref{redslider}--\eqref{redbraid} involving also the traverse-ups and traverse-downs \eqref{redcrossing}:
\begin{align}
\label{redslider}
\begin{tikzpicture}[anchorbase,scale=.6,color=\clr]
	\draw[-,line width=1pt,color=\cred] (0.4,0) to (-0.6,1);
  \draw (-.6,1.2) node{$\scriptstyle \red{u}$};
	\draw[-,line width=1pt] (0.08,0) to (0.08,1);
	\draw[-,line width=1pt] (0.1,0) to (0.1,.6) to (.5,1);
        \node at (0.6,1.18) {$\scriptstyle a$};
        \node at (0.1,1.2) {$\scriptstyle b$};
\end{tikzpicture}
& =
\begin{tikzpicture}[anchorbase,scale=.6,color=\clr]
	\draw[-,line width=1pt,color=\cred] (0.7,0) to (-0.3,1);
  \draw (-.3,1.16) node{$\scriptstyle \red{u}$};
	\draw[-,line width=1pt] (0.08,0) to (0.08,1);
	\draw[-,line width=1pt] (0.1,0) to (0.1,.2) to (.9,1);
        \node at (0.9,1.18) {$\scriptstyle a$};
        \node at (0.1,1.2) {$\scriptstyle b$};
\end{tikzpicture},
\quad
\begin{tikzpicture}[anchorbase,scale=.6,color=\clr]
	\draw[-,line width=1pt,color=\cred] (-0.4,0) to (0.6,1);
  \draw (.6,1.16) node{$\scriptstyle \red{u}$};
	\draw[-,line width=1pt] (-0.08,0) to (-0.08,1);
	\draw[-,line width=1pt] (-0.1,0) to (-0.1,.6) to (-.5,1);
         \node at (-0.1,1.18) {$\scriptstyle b$};
        \node at (-0.6,1.2) {$\scriptstyle a$};
\end{tikzpicture}
\!\!=\!\!
\begin{tikzpicture}[anchorbase,scale=.6,color=\clr]
	\draw[-,line width=1pt,color=\cred] (-0.7,0) to (0.3,1);
 \draw (.3,1.16) node{$\scriptstyle \red{u}$};
	\draw[-,line width=1pt] (-0.08,0) to (-0.08,1);
	\draw[-,line width=1pt] (-0.1,0) to (-0.1,.2) to (-.9,1);
         \node at (-0.1,1.2) {$\scriptstyle b$};
        \node at (-0.95,1.18) {$\scriptstyle a$};
\end{tikzpicture}, 
\quad
\begin{tikzpicture}[baseline=-3.3mm,scale=.6,color=\clr]
\draw[-,line width=1pt,color=\cred] (0.4,.2) to (-0.6,-.8);
\draw (-.6,-1.05) node{$\scriptstyle \red{u}$};
\draw[-,line width=1pt] (0.08,0.2) to (0.08,-.75);
\draw[-,line width=1pt] (0.1,0.2) to (0.1,-.4) to (.5,-.8);
\node at (0.6,-1.05) {$\scriptstyle c$};
\node at (0.07,-1.05) {$\scriptstyle b$};
\end{tikzpicture}
\!\!=\!\!
\begin{tikzpicture}[baseline=-3.3mm,scale=.6,color=\clr]
\draw[-,line width=1pt,color=\cred] (0.7,0.2) to (-0.3,-.8);
\draw (-.3,-1.0) node{$\scriptstyle \red{u}$};
\draw[-,line width=1pt] (0.08,0.2) to (0.08,-.75);
\draw[-,line width=1pt] (0.1,0.2) to (0.1,0) to (.9,-.8);
\node at (1,-1.0) {$\scriptstyle c$};
\node at (0.1,-1.0) {$\scriptstyle b$};
\end{tikzpicture} ,
\quad
\begin{tikzpicture}[baseline=-3.3mm,scale=.6,color=\clr]
\draw[-,line width=1pt,color=\cred] (-0.4,0.2) to (0.6,-.8);
\draw (.6,-.91) node{$\scriptstyle \red{u}$};
\draw[-,line width=1pt] (-0.08,0.2) to (-0.08,-.75);
\draw[-,line width=1pt] (-0.1,0.2) to (-0.1,-.4) to (-.5,-.8);
\node at (-0.1,-1.0) {$\scriptstyle b$};
\node at (-0.6,-1.0) {$\scriptstyle a$};
\end{tikzpicture}
\!\!=\!\!
\begin{tikzpicture}[baseline=-3.3mm,scale=.6,color=\clr]
\draw[-,line width=1pt,color=\cred] (-0.7,0.2) to (0.3,-.8);
\draw (.3,-1) node{$\scriptstyle \red{u}$};
\draw[-,line width=1pt] (-0.08,0.2) to (-0.08,-.75);
\draw[-,line width=1pt] (-0.1,0.2) to (-0.1,0) to (-.9,-.8);
\node at (-0.1,-1.0) {$\scriptstyle b$};
\node at (-0.95,-1.0) {$\scriptstyle a$};
\end{tikzpicture} ,
\\
\label{redcross2}
\begin{split}
\begin{tikzpicture}[baseline = 7pt, scale=0.3, color=\clr]
\draw[-,line width=1pt,color=\cred](0.1,-.22)to (0,2.5);
\draw (0.1,-.5) node{$\scriptstyle \red{u}$};
\draw[-,line width=1.2pt](-.5,-.1) to[out=45, in=down] (.5,1.4);
\draw[-,line width=1.2pt](.5,1.4) to[out=up, in=300] (-.5,2.4);
\draw (-.5,-0.5) node{$\scriptstyle {a}$};
\end{tikzpicture}
 & =
\sum_{0\le t \le a} (-1)^t q^{-t(a-t)} u^t
\begin{tikzpicture}[baseline = 7pt, scale=0.3, color=\clr]
\draw[-,line width=1pt,color=\cred](-.5,-.25)to (-.5,2.15);
\draw (-.5,-.6) node{$\scriptstyle \red{u}$};
\draw[-,line width=1pt] (-1.3,2.15) to (-1.3,-0.2);   
\draw (-1.3,1) \bdot;
\draw (-2.5,1.4) node{$\scriptstyle \omega_{a-t}$};
\draw (-1.3,-.6) node{$\scriptstyle {a}$};
\end{tikzpicture} ~, 
\\    
\begin{tikzpicture}[baseline = 7pt, scale=0.3, color=\clr]
\draw[-,line width=1pt,color=\cred](-0.1,-.6)to (-0.1,2.2);
\draw (-0.1,-.9) node{$\scriptstyle \red{u}$};
\draw[-,line width=1.2pt](.5,-.5) to[out=135, in=down] (-.5,1);
\draw[-,line width=1.2pt](-.5,1) to[out=up, in=270] (.5,2.2);
\draw (.5,-0.9) node{$\scriptstyle {a}$};
\end{tikzpicture}
&=
\sum_{0\le t \le a} (-1)^t q^{-t(a-t)} u^t
\begin{tikzpicture}[baseline = 7pt, scale=0.3, color=\clr]
\draw[-,line width=1pt,color=\cred](-1.8,-.25)to (-1.8,2.15);
\draw (-1.8,-.6) node{$\scriptstyle \red{u}$};
\draw[-,line width=1pt] (-1,2.15) to (-1,-0.25);          
\draw (-1,1) \bdot;
\draw (0.4,1)  node{$\scriptstyle \omega_{a-t}$};
\draw (-1,-.6) node{$\scriptstyle {a}$};
\end{tikzpicture} ,
\end{split}
\\
 \label{redbraid}
\begin{tikzpicture}[anchorbase, scale=0.35, color=\clr]
\draw[-,line width=1pt](0,2) to (1.5,-.35);
\draw[-,line width=4pt, white](0,-.35) to (1.5,2); to (1.5,-.35);
\draw[-,line width=1pt] (0,-.35) to (1.5,2);
\draw[-,line width=1pt,color=\cred](.8,2) to[out=down,in=90] (0.2,1) to [out=down,in=up] (.8,-.4);
\draw (.8,-.8) node{$\scriptstyle \red{u}$};
\node at (0, -.7) {$\scriptstyle a$};
\node at (1.5, -.7) {$\scriptstyle b$};
\end{tikzpicture}
& = 
\sum_{s=0}^{\min (a,b)} (-1)^s q^{s(s-1)/2} (q^{-1} -q)^s [s]! 
\begin{tikzpicture}[anchorbase, scale=0.35, color=\clr]
\draw[-,line width=1.5pt](0,-.35) to   (0,0);
\draw[-,line width=1.5pt](0,1.65) to   (0,2);
\draw[-,line width=1.5pt](1.5,-.35) to   (1.5,0);
\draw[-,line width=1.5pt](1.5,1.65) to   (1.5,2);
\draw[-,thick](0,0) to   (0,1.65);
\draw[-,thick](1.5,0) to   (1.5,1.65);
\draw[-,line width=1pt](0,1.65) to   (1.5,0);
\draw[-,line width=4pt, white] (0.2,0.2) to (1.3,1.35);
\draw[-,line width=1pt] (0,0) to (1.5,1.65);
\draw[-,line width=1pt,color=\cred](.8,2) to
[out=-60,in=60]
(.8,-.4);
\draw (.8,-.7) node{$\scriptstyle \red{u}$};
\node at (0, -.7) {$\scriptstyle a$};
\node at (1.5, -.7) {$\scriptstyle b$};
\node at (-0.3, 1.1) {$\scriptstyle s$};
\node at (1.8, 1.1) {$\scriptstyle s$};
\draw (1.5, .7) \bdot;
\end{tikzpicture} . 
\end{align}
\end{definition}

By the symmetry of the defining relations in Definition \ref{def:qSchur}, there is an isomorphism of  categories:
\begin{equation} 
\label{eq:autoflip}
\div : \qASch\longrightarrow \mathpzc{SchuR}^{{\hspace{-.03in}}\bullet\text{op}} 
\end{equation} 
defined by rotating a string diagram by 180 degrees around a horizontal axis. In particular, $\div$ sends thick positive (resp. negative) crossings from $(a,b)$ to $(b,a)$ to thick positive (resp., negative) crossing from $(b,a)$ to $(a,b)$.

\subsection{More relations in $\qASch$}

\begin{lemma}
 \label{adamovecrossings}
The following relations hold in $\qASch$, for $u\in \kk$:
\begin{align*}
\begin{tikzpicture}[anchorbase, scale=0.4, color=\clr]
\draw[-,line width=1.2pt](-.5,1.5) to (.5,-.5);
\draw[-,line width=4pt, white](-.5,-.5) to (.5,1.5); 
\draw[-,line width=1.2pt](-.5,-.5) to (.5,1.5); 
\draw[-,line width=1pt,color=\cred] (.8,1.5) to (0.8,.8) to[out=down,in=up] (-.8,-.5);
\draw(-.85,-.7) node{$\scriptstyle \red u$};
\end{tikzpicture}
    &
    ~=~
\begin{tikzpicture}[anchorbase, scale=0.4, color=\clr]
\draw[-,line width=1.2pt](-.5,1.5) to (.5,-.5);
\draw[-,line width=4pt, white](-.5,-.5) to (.5,1.5); 
\draw[-,line width=1.2pt](-.5,-.5) to (.5,1.5); 
\draw[-,line width=1pt,color=\cred] (.8,1.5) to [out=down,in=45](-.7,.5) to [out=-135,in=up] (-.8,-.4);
\draw(-.85,-.7) node{$\scriptstyle \red u$};
\end{tikzpicture},  
    \qquad
\begin{tikzpicture}[anchorbase, scale=0.4, color=\clr]
\draw[-,line width=1.2pt](-.5,-.5) to (.5,1.5); 
\draw[-,line width=4pt, white](-.5,1.5) to (.5,-.5);
\draw[-,line width=1.2pt](-.5,1.5) to (.5,-.5);
\draw[-,line width=1pt,color=\cred] (.8,1.5) to (0.8,.8) to[out=down,in=up] (-.8,-.5);
\draw(-.85,-.7) node{$\scriptstyle \red u$};
\end{tikzpicture}
    &
    ~=~
\begin{tikzpicture}[anchorbase, scale=0.4, color=\clr]
\draw[-,line width=1.2pt](-.5,-.5) to (.5,1.5); 
\draw[-,line width=4pt, white](-.5,1.5) to (.5,-.5);
\draw[-,line width=1.2pt](-.5,1.5) to (.5,-.5);
\draw[-,line width=1pt,color=\cred] (.8,1.5) to [out=down,in=45](-.7,.5) to [out=-135,in=up] (-.8,-.4);
\draw(-.85,-.7) node{$\scriptstyle \red u$};
\end{tikzpicture}  ,
    \qquad
\begin{tikzpicture}[anchorbase, scale=0.4, color=\clr]
\draw[-,line width=1.2pt](-.5,1.5) to (.5,-.5);
\draw[-,line width=4pt,white](-.5,-.5) to (.5,1.5); 
\draw[-,line width=1.2pt](-.5,-.5) to (.5,1.5); 
\draw[-,line width=1pt,color=\cred] (-.8,1.5) to (-.8,.9) to [out=down,in=up] (.8,-.5);
\draw(.8,-.7) node{$\scriptstyle \red u$};
\end{tikzpicture}
    ~=~
\begin{tikzpicture}[anchorbase, scale=0.4, color=\clr]
\draw[-,line width=1.2pt](-.5,1.5) to (.5,-.5);
\draw[-,line width=4pt,white](-.5,-.5) to (.5,1.5); 
\draw[-,line width=1.2pt](-.5,-.5) to (.5,1.5); 
\draw[-,line width=1pt,color=\cred] (-.8,1.5) to [out=down,in=134] 
(.7,.5) to[out=-45,in=up](.8,-.4);
\draw(.8,-.7) node{$\scriptstyle \red{u}$};
\end{tikzpicture} ,  
    \qquad
\begin{tikzpicture}[anchorbase, scale=0.4, color=\clr]
\draw[-,line width=1.2pt](-.5,-.5) to (.5,1.5); 
\draw[-,line width=4pt,white](-.5,1.5) to (.5,-.5);
\draw[-,line width=1.2pt](-.5,1.5) to (.5,-.5);
\draw[-,line width=1pt,color=\cred] (-.8,1.5) to (-.8,.9) to [out=down,in=up] (.8,-.5);
\draw(.8,-.7) node{$\scriptstyle \red u$};
\end{tikzpicture}
    ~=~
\begin{tikzpicture}[anchorbase, scale=0.4, color=\clr]
\draw[-,line width=1.2pt](-.5,-.5) to (.5,1.5); 
\draw[-,line width=4pt,white](-.5,1.5) to (.5,-.5);
\draw[-,line width=1.2pt](-.5,1.5) to (.5,-.5);
\draw[-,line width=1pt,color=\cred] (-.8,1.5) to [out=down,in=134] 
(.7,.5) to[out=-45,in=up](.8,-.4);
\draw(.8,-.7) node{$\scriptstyle \red{u}$};
\end{tikzpicture} .
\end{align*}
\end{lemma}

\begin{proof}
    This follows from \eqref{redslider} and \eqref{cross via square}.
\end{proof}

\begin{lemma}
The following  relations hold in $\qASch$:
\begin{equation}
\label{dotmoveadaptor}
\begin{tikzpicture}[anchorbase, scale=0.4, color=\clr]
\draw[-,line width=1.2pt] (-1.5,2.2) to (-1.5,1);
\draw[-,line width=1,color=\cred](-2,2.2) to [out=down,in=up] (-1,0);
\draw (-1,-.5) node{$\scriptstyle \red{u}$};
\draw[-,line width=1.2pt] (-1.5,1) to (-1.5,0); 
\draw(-1.5,.5)\bdot;
\draw (-2.1,.5) node{$\scriptstyle \omega_r$};
\draw (-1.5,-.5) node{$\scriptstyle {a}$};
\end{tikzpicture}
 ~=~
\begin{tikzpicture}[anchorbase, scale=0.4, color=\clr]
\draw[-,line width=1.2pt] (-1.5,2.2) to (-1.5,1);
\draw[-,line width=1,color=\cred](-2.2,2) to [out=down,in=up] (-1,0);
\draw (-1,-.5) node{$\scriptstyle \red{u}$};
\draw[-,line width=1.2pt] (-1.5,1) to (-1.5,0); 
\draw(-1.5,1.5)\bdot;
\draw (-.9,1.5) node{$\scriptstyle \omega_r$};
\draw (-1.5,-.5) node{$\scriptstyle {a}$};
\end{tikzpicture}, 
\quad 
\begin{tikzpicture}[anchorbase, scale=0.4, color=\clr]
\draw[-,line width=1.2pt] (-1.5,2.2) to (-1.5,1);
\draw[-,line width=1,color=\cred](-1,2) to [out=down,in=up] (-2.2,0);
\draw (-2.2,-.5) node{$\scriptstyle \red{u}$};
\draw[-,line width=1.2pt] (-1.5,1) to (-1.5,0);
\draw(-1.5,.5)\bdot;
\draw (-.9,.5) node{$\scriptstyle \omega_r$};
\draw (-1.5,-.5) node{$\scriptstyle {a}$};
\end{tikzpicture}
~=~
\begin{tikzpicture}[anchorbase, scale=0.4, color=\clr]
\draw[-,line width=1.2pt] (-1.5,2.2) to (-1.5,1);
\draw[-,line width=1,color=\cred](-.8,2) to [out=down,in=up] (-2.2,0);
\draw (-2.2,-.5) node{$\scriptstyle \red{u}$};
\draw[-,line width=1.2pt] (-1.5,1) to (-1.5,0);
\draw(-1.5,1.5)\bdot;
\draw (-2.1,1.5) node{$\scriptstyle \omega_r$};
\draw (-1.5,-.5) node{$\scriptstyle {a}$};
\end{tikzpicture}
 \end{equation}
 for all admissible $a,r$ and any  $u\in \kk$.
 \end{lemma}
 \begin{proof}
     Similar to the proof of  \cite[Lemma 3.5]{SW24Schur}
     by using \eqref{redcross2} and induction on the number of dots.
 \end{proof}


\subsection{Definition of $\qASchC$}

To define a simpler $\C(q)$-linear variant of $\qASch$, we shall enlarge the $\C(q)$-linear category $\qAWC$ to $\qASchC$ by allowing more generating objects (with corresponding relations) as follows.

\begin{definition}
\label{def:qSchurC}
    The affine $q$-Schur category  $\qASchC$ is a strict $\C(q)$-linear monoidal category with generating objects $a\in \Z_{\ge 1}$ and $u\in \C(q)$. We denote the object $a\in \N$ by $\stra$ and denote the object $u \in \C(q)$ by a red strand labeled by $u$ as $\stru .$
The morphisms are generated by 
\begin{equation} 
\label{generator-affschur C}
  \begin{tikzpicture}[anchorbase, scale=0.4, color=\clr] 
                \draw[-,line width=1pt] (-2,0.2) to (-1.5,1);
	\draw[-,line width=1pt ] (-1,0.2) to (-1.5,1);
	\draw[-,line width=1.5pt] (-1.5,1) to (-1.5,1.8);
                \draw (-1,0) node{$\scriptstyle {b}$};
                \draw (-2,0) node{$\scriptstyle {a}$};
                \draw (-1.5,2.2) node{$\scriptstyle {a+b}$};
                \end{tikzpicture}, 
                \quad 
  \begin{tikzpicture}[anchorbase, scale=0.4, color=\clr]
                \draw[-,line width=1pt] (-2,1.8) to (-1.5,1);
	\draw[-,line width=1pt ] (-1,1.8) to (-1.5,1);
	\draw[-,line width=1.5pt] (-1.5,1) to (-1.5,0.2);
                \draw (-1,2.2) node{$\scriptstyle {b}$};
                \draw (-2,2.2) node{$\scriptstyle {a}$};
                \draw (-1.5,0) node{$\scriptstyle {a+b}$};
                \end{tikzpicture}, 
    \quad  \;
    \crossingpos , 
        \quad \;
    ~\dotgen ,
 \end{equation}
and
\begin{equation}
\label{redcrossingC}
\rightcrossing \;\; (\text{traverse-up}), 
                \qquad
\leftcrossing \;\; (\text{traverse-down}),    
\end{equation}
subject to relations \eqref{webassocC}--\eqref{dotmovecrossingC} for the generators in \eqref{generator-affschur C} and additional relations \eqref{redsliderC}--\eqref{redbraidC} involving also the traverse-ups and traverse-downs \eqref{redcrossing}:
\begin{align}
\label{redsliderC}
\begin{tikzpicture}[anchorbase,scale=.7,color=\clr]
	\draw[-,line width=1pt,color=\cred] (0.4,0) to (-0.6,1);
  \draw (-.6,1.2) node{$\scriptstyle \red{u}$};
	\draw[-,line width=1pt] (0.08,0) to (0.08,1);
	\draw[-,line width=1pt] (0.1,0) to (0.1,.6) to (.5,1);
        \node at (0.6,1.2) {$\scriptstyle a$};
        \node at (0.1,1.2) {$\scriptstyle b$};
\end{tikzpicture}
& =
\begin{tikzpicture}[anchorbase,scale=.7,color=\clr]
	\draw[-,line width=1pt,color=\cred] (0.7,0) to (-0.3,1);
  \draw (-.3,1.16) node{$\scriptstyle \red{u}$};
	\draw[-,line width=1pt] (0.08,0) to (0.08,1);
	\draw[-,line width=1pt] (0.1,0) to (0.1,.2) to (.9,1);
        \node at (0.9,1.13) {$\scriptstyle a$};
        \node at (0.1,1.16) {$\scriptstyle b$};
\end{tikzpicture},
\quad
\begin{tikzpicture}[anchorbase,scale=.7,color=\clr]
	\draw[-,line width=1pt,color=\cred] (-0.4,0) to (0.6,1);
  \draw (.6,1.16) node{$\scriptstyle \red{u}$};
	\draw[-,line width=1pt] (-0.08,0) to (-0.08,1);
	\draw[-,line width=1pt] (-0.1,0) to (-0.1,.6) to (-.5,1);
         \node at (-0.1,1.18) {$\scriptstyle b$};
        \node at (-0.6,1.16) {$\scriptstyle a$};
\end{tikzpicture}
\!\!=\!\!
\begin{tikzpicture}[anchorbase,scale=.7,color=\clr]
	\draw[-,line width=1pt,color=\cred] (-0.7,0) to (0.3,1);
 \draw (.3,1.16) node{$\scriptstyle \red{u}$};
	\draw[-,line width=1pt] (-0.08,0) to (-0.08,1);
	\draw[-,line width=1pt] (-0.1,0) to (-0.1,.2) to (-.9,1);
         \node at (-0.1,1.18) {$\scriptstyle b$};
        \node at (-0.95,1.16) {$\scriptstyle a$};
\end{tikzpicture}, 
\quad
\begin{tikzpicture}[baseline=-3.3mm,scale=.7,color=\clr]
\draw[-,line width=1pt,color=\cred] (0.4,.2) to (-0.6,-.8);
\draw (-.6,-1) node{$\scriptstyle \red{u}$};
\draw[-,line width=1pt] (0.08,0.2) to (0.08,-.75);
\draw[-,line width=1pt] (0.1,0.2) to (0.1,-.4) to (.5,-.8);
\node at (0.6,-1) {$\scriptstyle c$};
\node at (0.07,-1) {$\scriptstyle b$};
\end{tikzpicture}
\!\!=\!\!
\begin{tikzpicture}[baseline=-3.3mm,scale=.7,color=\clr]
\draw[-,line width=1pt,color=\cred] (0.7,0.2) to (-0.3,-.8);
\draw (-.3,-1) node{$\scriptstyle \red{u}$};
\draw[-,line width=1pt] (0.08,0.2) to (0.08,-.75);
\draw[-,line width=1pt] (0.1,0.2) to (0.1,0) to (.9,-.8);
\node at (1,-1) {$\scriptstyle c$};
\node at (0.1,-1) {$\scriptstyle b$};
\end{tikzpicture} ,
\quad
\begin{tikzpicture}[baseline=-3.3mm,scale=.7,color=\clr]
\draw[-,line width=1pt,color=\cred] (-0.4,0.2) to (0.6,-.8);
\draw (.6,-.91) node{$\scriptstyle \red{u}$};
\draw[-,line width=1pt] (-0.08,0.2) to (-0.08,-.75);
\draw[-,line width=1pt] (-0.1,0.2) to (-0.1,-.4) to (-.5,-.8);
\node at (-0.1,-.9) {$\scriptstyle b$};
\node at (-0.6,-.91) {$\scriptstyle a$};
\end{tikzpicture}
\!\!=\!\!
\begin{tikzpicture}[baseline=-3.3mm,scale=.7,color=\clr]
\draw[-,line width=1pt,color=\cred] (-0.7,0.2) to (0.3,-.8);
\draw (.3,-1) node{$\scriptstyle \red{u}$};
\draw[-,line width=1pt] (-0.08,0.2) to (-0.08,-.75);
\draw[-,line width=1pt] (-0.1,0.2) to (-0.1,0) to (-.9,-.8);
\node at (-0.1,-1) {$\scriptstyle b$};
\node at (-0.95,-1) {$\scriptstyle a$};
\end{tikzpicture} ,
\\
\label{redcross2C}
\begin{tikzpicture}[anchorbase, scale=0.35, color=\clr]
\draw[-,line width=1pt,color=\cred](0,-.22)to (0,2.5);
\draw (0,-.5) node{$\scriptstyle \red{u}$};
\draw[-,thick](-.5,-.1) to[out=45, in=down] (.5,1.4);
\draw[-,thick](.5,1.4) to[out=up, in=300] (-.5,2.4);
\draw (-.5,-0.5) node{$\scriptstyle {1}$};
\end{tikzpicture}
& =
\begin{tikzpicture}[anchorbase, scale=0.35, color=\clr]
\draw[-,line width=1pt,color=\cred](-.5,-.25)to (-.5,2.15);
\draw (-.5,-.6) node{$\scriptstyle \red{u}$};
\draw[-,thick] (-1.3,2.15) to (-1.3,-0.2);      
\draw (-1.3,1) \bdot;
\draw (-1.3,-.6) node{$\scriptstyle {1}$};
\end{tikzpicture}
- u  
\begin{tikzpicture}[anchorbase, scale=0.35, color=\clr]
\draw[-,line width=1pt,color=\cred](-.5,-.25)to (-.5,2.15);
\draw (-.5,-.6) node{$\scriptstyle \red{u}$};
\draw[-,thick] (-1.3,2.15) to (-1.3,-0.2);      
\draw (-1.3,-.6) node{$\scriptstyle {1}$};
\end{tikzpicture}
~, \qquad \qquad  
\begin{tikzpicture}[anchorbase, scale=0.35, color=\clr]
\draw[-,line width=1pt,color=\cred](0,-.6)to (0,2.2);
\draw (0,-.9) node{$\scriptstyle \red{u}$};
\draw[-,thick](.5,-.5) to[out=135, in=down] (-.5,1);
\draw[-,thick](-.5,1) to[out=up, in=270] (.5,2.2);
\draw (.5,-0.9) node{$\scriptstyle {1}$};
\end{tikzpicture}
~=~ 
\begin{tikzpicture}[anchorbase, scale=0.35, color=\clr]
\draw[-,line width=1pt,color=\cred](-1.8,-.25)to (-1.8,2.15);
\draw (-1.8,-.6) node{$\scriptstyle \red{u}$};
\draw[-,thick] (-1,2.15) to (-1,-0.25);          
\draw (-1,1) \bdot;
\draw (-1,-.6) node{$\scriptstyle {1}$};
\end{tikzpicture}
- u
\begin{tikzpicture}[anchorbase, scale=0.35, color=\clr]
\draw[-,line width=1pt,color=\cred](-1.8,-.25)to (-1.8,2.15);
\draw (-1.8,-.6) node{$\scriptstyle \red{u}$};
\draw[-,thick] (-1,2.15) to (-1,-0.25);          
\draw (-1,-.6) node{$\scriptstyle {1}$};
\end{tikzpicture},
\\
 \label{redbraidC}
\begin{tikzpicture}[anchorbase, scale=0.35, color=\clr]
\draw[-,thick](0,2) to (1.5,-.35);
\draw[-,line width=4pt, white](0,-.35) to (1.5,2); to (1.5,-.35);
\draw[-,thick] (0,-.35) to (1.5,2);
\draw[-,line width=1pt,color=\cred](.8,2) to[out=down,in=90] (0.2,1) to [out=down,in=up] (.8,-.4);
\draw (.8,-.8) node{$\scriptstyle \red{u}$};
\node at (0, -.7) {$\scriptstyle 1$};
\node at (1.5, -.7) {$\scriptstyle 1$};
\end{tikzpicture}
& =
\begin{tikzpicture}[anchorbase, scale=0.35, color=\clr]
\draw[-,thick](0,2) to   (1.5,-.35);
\draw[-,line width=4pt, white] (0,-.35) to (1.5,2);
\draw[-,thick] (0,-.35) to (1.5,2);
\draw[-,line width=1pt,color=\cred](.8,2) to
[out=down,in=up] (1.5,1) to [out=down,in=up] (.8,-.4);
\draw (.8,-.7) node{$\scriptstyle \red{u}$};
\node at (0, -.7) {$\scriptstyle 1$};
\node at (1.5, -.7) {$\scriptstyle 1$};
\end{tikzpicture}
- (q^{-1} -q)
\begin{tikzpicture}[anchorbase, scale=0.35, color=\clr]
\draw[-, thick] (0,-.2) to (0,1.8);
\draw[-, line width=1pt,color=\cred] (.5,1.8) to (.5,-.2);
\draw[-, thick] (1,1.8) to (1,-.2); 
\draw (.5,-.6) node{$\scriptstyle \red{u}$};
\node at (0, -.6) {$\scriptstyle 1$};
\node at (1, -.6) {$\scriptstyle 1$};
\draw (1, 0.8) \bdot;
\end{tikzpicture} .
\end{align}
\end{definition}

Assuming the relation \eqref{redcross2C}, we readily show that the relation \eqref{redbraidC} is equivalent to the relation 
\begin{align}
\label{eq:redbraidequiv}
    \begin{tikzpicture}[anchorbase, scale=0.35, color=\clr]
\draw[-,thick](0,2) to (1.5,-.35);
\draw[-,line width=4pt, white](0,-.35) to (1.5,2); to (1.5,-.35);
\draw[-,thick] (0,-.35) to (1.5,2);
\draw[-,line width=1pt,color=\cred](.8,2) to[out=down,in=90] (0.2,1) to [out=down,in=up] (.8,-.4);
\draw (.8,-.8) node{$\scriptstyle \red{u}$};
\node at (0, -.7) {$\scriptstyle 1$};
\node at (1.5, -.7) {$\scriptstyle 1$};
\end{tikzpicture}
~=~ 
\begin{tikzpicture}[anchorbase, scale=0.35, color=\clr]
\draw[-,thick] (0,-.35) to (1.5,2);
\draw[-,line width=4pt, white] (0,2) to   (1.5,-.35);
\draw[-,thick](0,2) to   (1.5,-.35);
\draw[-,line width=1pt,color=\cred](.8,2) to
[out=down,in=up] (1.5,1) to [out=down,in=up] (.8,-.4);
\draw (.8,-.7) node{$\scriptstyle \red{u}$};
\node at (0, -.7) {$\scriptstyle 1$};
\node at (1.5, -.7) {$\scriptstyle 1$};
\end{tikzpicture}
~-~ 
u (q^{-1} -q)
\begin{tikzpicture}[anchorbase, scale=0.35, color=\clr]
\draw[-, thick] (0,-.2) to (0,1.8);
\draw[-, line width=1pt,color=\cred] (.5,1.8) to (.5,-.2);
\draw[-, thick] (1,1.8) to (1,-.2); 
\draw (.5,-.6) node{$\scriptstyle \red{u}$};
\node at (0, -.6) {$\scriptstyle 1$};
\node at (1, -.6) {$\scriptstyle 1$};
\end{tikzpicture} .
\end{align}

\subsection{Isomorphism of $\qASchC$ and $\qASch$ over $\C(q)$}

We define the morphisms $\xdota$ and $\zdota$ (for $a\ge 1$) as in \eqref{eq:dota} in $\qASchC$; it is convenient to formally include them as generating morphisms in $\qASchC$ for discussion below. 

\begin{theorem}
\label{th:qASchiso}
    The $\C(q)$-linear monoidal categories $\qASchC$ and $\qASch$ are isomorphic by matching the generating objects and generating morphisms in the same symbols. 
\end{theorem}

\begin{proof}
    There is a natural functor $\Psi: \qASchC \rightarrow \qASch$ which matches the generating objects and generating morphisms in the same symbols. 

    Note that the defining relations \eqref{webassocC}--\eqref{dotmovecrossingC} and \eqref{redsliderC}--\eqref{redbraidC} for $\qASchC$ are weaker than the defining relations \eqref{webassoc}--\eqref{intergralballon} and \eqref{redslider}--\eqref{redbraid} for $\qASch$. 

    In the proof of Theorem~\ref{th:qAWiso}, we have already shown that the relations \eqref{webassoc}--\eqref{intergralballon} hold in $\qAWC$, and hence hold in $\qASchC$ as well since the defining relations for $\qAWC$ is part of defining relations for $\qASchC$. Note that the slider relations \eqref{redslider} are the same as \eqref{redsliderC}. It remains to show that the relations \eqref{redcross2}--\eqref{redbraid} hold for $\qASchC$, whose proofs will be given in Lemma~\ref{lem:redcross2thick} and Lemma~\ref{lem:redbraidthick} below. This completes the proof of the theorem. 
\end{proof}

Below we supply the detailed proofs that the two relations \eqref{redcross2}--\eqref{redbraid} hold for $\qASchC$. 
\begin{lemma}
    \label{lem:redcross2thick}
    The relation \eqref{redcross2} hold in $\qASchC$, for $u\in \C(q)$ and $a\ge 1$, that is,
    \begin{align*}
 .
    \end{align*}
\end{lemma}

\begin{proof}
We will prove the first identity only, as the second identity is proved similarly. We proved by induction on $a$, with the base case when $a=1$ being  \eqref{redcross2C}. For $a\ge 2$, we have
\begin{align*}
 [a]
 %
 ,
\end{align*}
where we have changed variables $t\mapsto t-1$ in the second summand of the RHS above. 
Note the last diagram above can be simplified using Lemma ~\ref{lem:rightdot}
 as follows:
\begin{align}
\label{cal2}
%
 .
\end{align}
The first identity in the lemma follows now by plugging the formula \eqref{cal2} back into the previous computation and simplifying via
the identity $q^t [a-t] +q^{-(a-t)} [t] = [a]$. 
\end{proof}


\begin{lemma}
    \label{lem:dotredcross}
    The following relations hold in $\qASchC$, for $u\in \C(q)$ and $a \ge 1$:
    \begin{align*}
%
 .
    \end{align*}
\end{lemma}

\begin{proof}
    Follows by the same argument as for \cite[Lemma 3.5]{SW24Schur}. 
\end{proof}

\begin{lemma}
\label{lem:redbraidthick}
    The relation \eqref{redbraid} hold in $\qASchC$, for $u\in \C(q)$ and $a,b \ge 1$, that is,
 \begin{align*}
%
 . 
 \end{align*}
\end{lemma}

\begin{proof}
We prove the formula by a double induction on $a, b$. First let us prove inductively on $b$ the desired formula for $a=1$, i.e., 
\begin{align}
\label{redbraid1b}
    %
 ,
\end{align}
with the base case for $b=1$ being \eqref{redbraidC}. Then, for $b\ge 2$, we have
\begin{align*}
[b]
%
 .
\end{align*}
The desired identity \eqref{redbraid1b} now follows from this and noting that $q^{-1} [b-1] + q^{b-1} = [b]$.

Now we prove the formula stated in the lemma by induction on $a$ with the base case for $a=1$ being \eqref{redbraid1b}. Then, for $a\ge 2$, we have
\begin{align}
[a]
%
 .
\label{cal1}
\end{align}

The first summand on the RHS of \eqref{cal1} can be simplified as follows:
\begin{align*}
      %
 .
\end{align*}
On the other hand, the second summand on the RHS of \eqref{cal1} can be simplified as follows:
\begin{align*}
-\eps
%
 . 
\end{align*} 
The lemma follows by inserting the above two simplified formulas for the two summands back to the identity \eqref{cal1} and using the identity $q^{-s} [a-s] +q^{a-s} [s] = [a]$.
\end{proof}

\subsection{A polynomial representation of $\qASch$}
\label{sec:polyrep}
For any \( m \in \mathbb{N} \), a composition (respectively, partition) \( \mu = (\mu_1, \mu_2, \ldots, \mu_n) \) of \( m \) is defined as a sequence of (respectively, weakly decreasing) non-negative integers such that \( \sum_{i \geq 1} \mu_i = m \). The length of \( \mu \) is denoted by \( l(\mu) \). A composition \( \mu \) is called \textit{strict} if \( \mu_i > 0 \) for all \( i \). Let \( \Lambda_{\text{st}}(m) \) (respectively, \( \Par(m) \)) represent the set of all strict compositions (respectively, partitions) of \( m \), and let \( \Lambda_{\text{st}} \) (respectively, \( \Par \)) represent the set of all compositions (respectively, partitions).

Recall that the affine Hecke algebra $\Hd$ is the associative unital $\kk$-algebra generated by $H_i$ and $X^{\pm 1}_l$, for $i=1,\ldots,d-1$ and $l=1,\ldots,d$, subject to the relations for all admissible $i,j,l$:
\begin{align}
    (H_i-q^{-1})(H_i+q)=0,\quad &H_iH_{i+1}H_i=H_{i+1}H_iH_{i+1},\quad
    H_iH_j=H_jH_i\ (|i-j|>1)\\
    X_iX_i^{-1}=1,\quad&X_iX_j=X_jX_i,\quad X^{\pm 1}_jH_i=H_iX^{\pm 1}_j\ (j\neq i,i+1),\\
    &H_iX_iH_i=X_{i+1}.
    \label{HXH}
\end{align}
The Hecke algebra $\mathcal{H}_d$ is the subalgebra of $\Hd$ generated by $H_i$ for $i=1,\ldots,d-1$.

For any $a\in \N$, let 
\[
\text{Sym}_a:=\kk[X^{\pm}_1,X^{\pm}_2,\ldots, X^{\pm}_a]^{\mathfrak S_a}
\]
be the ring of symmetric Laurent 
 polynomials in $X_1,\ldots, X_a$.
More generally we consider the ring of partially symmetric Laurent  polynomials
\[\text{Sym}_\lambda:=\kk[X^{\pm 1}_1,\ldots, X^{\pm 1}_a]^{\mathfrak S_{\lambda}}\cong \text{Sym}_{\lambda_1} \otimes \ldots\otimes \text{Sym}_{\lambda_s}
\]
for any $\lambda=(\lambda_1,\ldots, \lambda_s)\in \Lambda_{\text{st}}(a)$.

Recall that (cf. \cite[Proposition 3.6]{Lus89}) the affine Hecke algebra $\Hd$ acts on $\kk[X^{\pm 1}_1,\ldots,X^{\pm 1}_d]$ by letting $X^{\pm 1}_l$ act by multiplication by $X^{\pm 1}_l$ and $H_i$ act by
 \begin{equation}
 \label{actionHi}
     H_i. f= q^{-1}f+(qX_{i+1}-q^{-1}X_i)\frac{ f-f^{s_i}}{X_i-X_{i+1}}
 \end{equation}
for all 
$f\in \kk[X^{\pm 1}_1,\ldots, X^{\pm 1}_d] $ and all admissible $i,l$,
where 
\[ f^{s_i}(X^{\pm 1}_1,\ldots, X^{\pm 1}_d)=f(X^{\pm 1}_1, \ldots, X^{\pm 1}_{i+1}, X^{\pm 1}_i,\ldots,X^{\pm 1}_d).\]

For any reduced expression $w=s_{i_1}\cdots s_{i_j}$ in the symmetric group $\mathfrak S_{d}$, we define $H_w:=H_{i_1}\cdots H_{i_j}$. For any $a,b\in \N$, we define
    \begin{align}  \label{eq:sigmaab}
    \sigma_{a,b} := \sum_{w\in (\mathfrak S_{a+b} /\mathfrak S_a \times \mathfrak S_b)_{\min} }q^{\ell(w_{a,b})-\ell(w)}H_w,
\end{align}
where $w_{a,b}$ is the longest element of $(\mathfrak S_{a+b} /\mathfrak S_a \times \mathfrak S_b)_{\min}$.
 
Now we introduce $\kk$-linear maps for each generator of $\qASch$ in \eqref{generator-affschur}--\eqref{redcrossing} denoted by the same symbol as follows: 
\begin{align}
\begin{split}
        &\splits\colon \text{Sym}_{a+b}\longrightarrow \text{Sym}_{(a,b)}, \quad f\mapsto f,
        \\
        &\merge\colon \text{Sym}_{(a,b)}\longrightarrow \text{Sym}_{a+b}, \quad f\mapsto \sigma_{a,b}.f,
        \end{split}
         \label{map_SM}
\\ 
\begin{split}
     &\wkdotr\colon \text{Sym}_r\longrightarrow \text{Sym}_r, \quad f\mapsto X_1X_2\cdots X_r f,\\
     &\wkzdotr\colon \text{Sym}_r\longrightarrow \text{Sym}_r, \quad f\mapsto X^{-1}_1X^{-1}_2\cdots X^{-1}_r f,
     \end{split}
     \label{eq:wrblackwhite}
\\
     \begin{split}
     &\rightcrossing: \text{Sym}_a \longrightarrow \text{Sym}_a, \quad  f \mapsto f,
     \\
     &\leftcrossing: \text{Sym}_a\longrightarrow \text{Sym}_a, \quad f\mapsto \prod_{1\le j\le a}(X_j-u)f.
     \end{split}
\label{map_redcross}
\end{align} 
where the action of $\sigma_{a,b}$ is given by \eqref{actionHi}. Finally, the linear maps associated with the crossings $\crossingpos$ and $\crossingneg$ are uniquely determined by those of splits and merges by \eqref{cross via square}. 
For example, the linear map for 
$ \begin{tikzpicture}[anchorbase,scale=1,yscale=.6, color=\clr]
\draw[-,thick] (0,1) to (.6,0);
  \draw[wipe] (0,0) to (.6,1);
	\draw[-,thick] (0,0) to (.6,1);
        \node at (0,-.2) {$\scriptstyle 1$};
        \node at (0.6,-0.2) {$\scriptstyle 1$};
\end{tikzpicture} 
$ is given by the action of $H_1$ in \eqref{actionHi}. In general,
the action of $\crossingpos$ on $\text{Sym}_{(a,b)}$ is given by the action of $H_{w_{a,b}}$.

Denote by $\mathpzc{Vec}_\kk$ the category of free $\kk$-modules. 

\begin{theorem}
    \label{th:ASchRep}
There is a $\kk$-linear monoidal functor 
\[
\mathcal F: \qASch \longrightarrow \mathpzc{Vec}_\kk, 
\]
which sends objects $a\in \N$ to $\text{Sym}_a$, and  $u\in \kk$ to $\kk$, and sends the generating morphisms to the linear maps in the same symbols defined in \eqref{map_SM}--\eqref{map_redcross} above.
\end{theorem}

\begin{proof}
    The proof is completely analogous to the one for \cite[Theorem~3.6]{SW24Schur} (which is a degenerate counterpart of the current theorem); we skip the detail. 
\end{proof}

\section{Bases for affine $q$-web and $q$-Schur categories}
\label{sec:bases_affineWS}

In this section, we establish bases for arbitrary Hom-spaces for the categories $\qAW$ and $\qASch$.

\subsection{Elementary diagrams of $\qAW$}

The Hom-spaces for the degenerate version of $\qAW$ have a basis given by the \emph{elementary chicken foot diagrams} (see \cite[\S 3.1]{SW24web}).  In this subsection, we shall construct a spanning set of $\qAW$, which is a $q$-analogue of the  elementary chicken foot diagram. 

The difference between web diagrams of $\qAW$ and its degenerate analogues lies in the distinction of the \emph{positive} crossings $\crossingpos$ and the \emph{negative} crossings $\crossingneg$. From the relation of positive and negative crossings as in Lemma~\ref{lem:cross+=-}, only one of them is needed to construct a spanning set for $\qAW$.

Let $\lambda,\mu\in \Lambda_{\text{st}}(m)$. A $\lambda\times \mu$ \emph{elementary chicken foot ribbon diagram} is a diagram in $\Hom_{\qW}(\mu,\lambda)$ consisting of three horizontal  parts such that
\begin{itemize}
    \item the bottom part consists of only splits,
    \item the top part consists of only merges,
    \item the middle part consists of only positive crossings of the thinner strands (i.e., legs).
\end{itemize}
A chicken foot ribbon diagram (CFRD) is called \emph{reduced} if there is at most one
intersection or one join between every pair of the legs. 

For $\lambda,\mu\in \Lambda_{\text{st}}(m)$, let 
$\Mat_{\lambda,\mu}$ be the set of all $l(\lambda)\times l(\mu)$
non-negative integer matrices $A=(a_{ij})$ with row sum vector $\la$ and column sum vector $\mu$, i.e.,  $\sum_{1\le h\le l(\mu)}a_{ih}=\lambda_i$ and $\sum_{1\le h\le l(\lambda)}a_{hj}=\mu_j$ for all $1\le i\le l(\lambda) $ and $1\le j\le l(\mu)$. 

For each reduced chicken foot ribbon diagram, we may forget about the positivity of each crossing and recover a reduced chicken foot diagram. Thus the combinatorial information of a chicken foot ribbon diagram is really the same as the underlying chicken foot diagram, which is encoded by a matrix $A\in \Mat_{\lambda,\mu}$ ,cf. \cite[(3.2)]{SW24web}.  
Recall the morphisms $\omega_{a,r}$ from \eqref{eq:wkdota} for $1\leq |r| \leq a$ and the set of rational partitions $\RPar_a$ from \eqref{eq:RPar}. We call such a morphism {\em the $r$-th elementary dot packet (of thickness $a$)}. We may also write $\omega_r=\omega_{a,r}$ if $a$ is clear in the context. For any tuple $\nu=(\nu_1,\nu_2,\ldots, \nu_k)\in \RPar_a$, the {\em elementary dot packet (of thickness $a$)} is defined to be 
\[
\omega_{a, \nu}:= \omega_{a,\nu_1}\omega_{a,\nu_2}\cdots \omega_{a,\nu_k}\in \End_{\qAW}(a).
\]
We write $\omega_\nu=\omega_{a,\nu}$ if $a$ is clear from the context and draw it as $
\begin{tikzpicture}[baseline = 3pt, scale=0.4, color=\clr]
\draw[-,line width=1.2pt] (0,-.2) to[out=up, in=down] (0,1.2);
\draw(0,0.5) \bdot; \node at (0.6,0.5) {$ \scriptstyle \nu$};
\node at (0,-.4) {$\scriptstyle a$};
\end{tikzpicture}$. Moreover, a CFRD with elementary dot packets on each segment of its legs is referred to as a \emph{dotted chicken foot ribbon diagram}.

\begin{definition} \label{def:dottedribbon}
    For $\lambda,\mu\in \Lambda_{\text{st}}(m)$, a $\lambda\times \mu$ elementary chicken foot ribbon diagram (or simply elementary ribbon diagram) is a $\lambda\times \mu$ reduced chicken foot ribbon diagram with an elementary dot packet $\omega_{\nu}$, for some tuple $\nu\in \RPar_a$, attached at the bottom of each leg with thickness $a$. 
\end{definition}

Denote  
\begin{equation}\label{dottedreduced}
 \RPMat_{\lambda,\mu}:=\{ (A, P))\mid  A=(a_{ij})\in \Mat_{\lambda,\mu}, P=(\nu_{ij}), \nu_{ij}\in \RPar_{a_{ij}} \}.   
\end{equation}
We will identify the set of all elementary chicken foot ribbon diagrams from $\mu$ to $\lambda$ with $\RPMat_{\lambda,\mu}$.

\begin{example}
    We have the following elementary chicken foot ribbon diagram:
    \begin{align}
    \label{examplereduced}
          \begin{tikzpicture}
        [anchorbase,color=\clr,scale=1.5]
        \draw[-,line width=1pt] (-.5,1) to (-1,0)
                                (-.5,1) to (0,0)
                                (-.5,1) to (1,0);
        \draw[wipe]             (.2,.8) to (-.7,.2)
                                (.4,.8) to (.1,.2)
                                (.6,.8) to (.8,.4);
        \draw[-,line width=1pt] (.5,1) to (-1,0)
                                (.5,1) to (0,0)
                                (.5,1) to (1,0);
        \draw[-,line width=1.5pt] (-.5,1) to (-.5,1.15)
                                (.5,1) to (.5,1.15)
                                (-1,0) to (-1,-.15)
                                (1,0) to (1,-.15)
                                (0,0) to (0,-.15);
        \node at  (-.8,.8){$\scriptstyle 2$};
        \node at  (-.5,.6){$\scriptstyle 3$};
        \node at  (-.2,1){$\scriptstyle 4$};
        \node at  (.2,1){$\scriptstyle 3$};
        \node at  (.5,.6){$\scriptstyle 2$};
        \node at  (.8,.8){$\scriptstyle 2$};
        \draw (-.85,.3) \bdot;
        \node at (-1.1,.3){$\scriptstyle \nu^1$};
        \draw (-.52,.3) \bdot;
        \node at (-.52,.1){$\scriptstyle \nu^2$};
        \draw (-.15,.3) \bdot;
        \node at (-.18,.1){$\scriptstyle \nu^3$};
        \draw (.15,.3) \bdot;
        \node at (.18,.1){$\scriptstyle \nu^4$};
        \draw (.58,.3) \bdot;
        \node at (.58,.1){$\scriptstyle \nu^5$};
        \draw (.88,.3) \bdot;
        \node at (1.13,.3){$\scriptstyle \nu^6$};
    \end{tikzpicture}
    \end{align}
    where $\nu^i\in \RPar_{a_i}$, with the thickness of legs are $(a_1,\ldots,a_6)=(2,3,3,2,4,2)$. Associated with \eqref{examplereduced} we have
    \[
    A=\begin{pmatrix}
        2& 3& 4\\
        3& 2& 2
    \end{pmatrix},\qquad 
    P=\begin{pmatrix}
        \nu^1& \nu^3 & \nu^5\\
        \nu^2& \nu^4& \nu^6
    \end{pmatrix}.
    \]
\end{example}

\begin{definition}
\label{def:Cdeg}
     Assign degree $a+b$ to $\crossingpos$ and $\crossingneg$ and $0$ to all other generating morphisms \eqref{generator-affschur}--\eqref{redcrossing}.  We define the \emph{crossing degree}, denoted by $\deg^\times$, of a dotted ribbon diagram as the sum of the degrees of its local components.
\end{definition}
Let $\Hom_{\qAW} (\mu,\la)_{<k}$ (resp. $\Hom_{\qAW} (\mu,\la)_{\le k}$) denote the $\kk$-span of all dotted ribbon diagrams with crossing degree $<k$ (resp. $\leq k$). For any $R$ and $R'$ in $\Hom_{\qAW} (\mu,\la)_{\le k}$, we write 
\[
R \equivc R' \text{ if }R=R' \mod \Hom_{\qAW} (\mu,\la)_{<k}.
\]

\begin{lemma}
    \label{lem:movedotsfreely}
    The following $\equivc$-relations hold in $\qAW$:
    \begin{align*}
    \crossingpos
    &\equivc
    \crossingneg,\\
\begin{tikzpicture}[anchorbase, scale=0.5, yscale=.5,color=\clr]
\draw[-,line width=1.2pt] (0,-.2) to  (1,2.2);
\draw[wipe] (1,-.2) to  (0,2.2);
\draw[-,line width=1.2pt] (1,-.2) to  (0,2.2);
\draw(0.2,1.6)\bdot;
\draw(-.4,1.6) node {$\scriptstyle \omega_r$};
\node at (0, -.5) {$\scriptstyle a$};
\node at (1, -.5) {$\scriptstyle b$};
\end{tikzpicture}
&\equivc~ 
\begin{tikzpicture}[anchorbase, yscale=.5,scale=0.5, color=\clr]
\draw[-, line width=1.2pt] (0,-.2) to (1,2.2);
\draw[wipe] (1,-.2) to (0,2.2);
\draw[-,line width=1.5pt] (1,-.2) to (0,2.2);
\draw(.8,0.3)\bdot;
\draw(1.35,0.3) node {$\scriptstyle \omega_r$};
\node at (0, -.5) {$\scriptstyle a$};
\node at (1, -.5) {$\scriptstyle b$};
\end{tikzpicture}, 
\quad 
\begin{tikzpicture}[anchorbase, yscale=.5,scale=0.5, color=\clr]
\draw[-, line width=1.2pt] (0,-.2) to (1,2.2);
\draw[wipe]  (1,-.2) to(0,2.2);
\draw[-,line width=1.5pt] (1,-.2) to(0,2.2);
\draw(.2,0.2)\bdot;
\draw(-.4,.2)node {$\scriptstyle \omega_r$};
\node at (0, -.5) {$\scriptstyle b$};
\node at (1, -.5) {$\scriptstyle a$};
\end{tikzpicture}
~\equivc~
\begin{tikzpicture}[anchorbase, yscale=.5,scale=0.5, color=\clr]
\draw[-,line width=1.2pt] (0,-.2) to  (1,2.2);
\draw[wipe] (1,-.2) to  (0,2.2);
\draw[-,line width=1.5pt] (1,-.2) to  (0,2.2);
\draw(0.8,1.6)\bdot;
\draw(1.4,1.6)node {$\scriptstyle \omega_r$};
\node at (0, -.5) {$\scriptstyle b$};
\node at (1, -.5) {$\scriptstyle a$};
\end{tikzpicture},\quad |r|\leq b.
    \end{align*}
\end{lemma}

\begin{proof}
    Follows from Lemma~\ref{lem:cross+=-} and Lemma~\ref{lem:wrdotcross} directly.
\end{proof}

\begin{proposition} \label{prop:spanqAW}
For any $\lambda,\mu\in \Lambda_{\text{st}}(m)$, $\Hom_{\qAW}(\mu,\lambda)$ is spanned by $\RPMat_{\la,\mu}$ consisting of all $\lambda \times \mu$ elementary chicken foot diagrams from $\mu$ to $\la$. 
\end{proposition}
\begin{proof}
Let $D_{\lambda,\mu}$ denote the set of all dotted CFRDs (not necessarily reduced) from $\mu$ to $\lambda$. Clearly we have $\RPMat_{\la,\mu}\subset D_{\lambda,\mu}$.  We prove the proposition in two steps.
\begin{enumerate}
    \item $\Hom_{\qAW}(\mu,\lambda)$ is spanned by $D_{\lambda,\mu}$,
    \item  $D_{\lambda,\mu}$ is contained in the span of $\RPMat_{\la,\mu}$.
\end{enumerate}

To prove (1), it suffices to show that $fg$ can be written as a linear combination of elements in $D_{\lambda,\mu}$  for any $g\in D_{\lambda,\mu}$ and a generating morphism $f$ of the following types: 
\begin{equation*}
    1_* \wkdotaa 1_*,\qquad  1_*\merge 1_*,\qquad  1_*\splits 1_*,
\end{equation*}
where $1_*$ represents suitable identity morphisms. 

When $f=1_* \wkdotaa 1_*$, the claim follows from Lemma~\ref{lem:dotSM} directly. When $f=1_*\merge 1_*$, then the claim follows from the associativity relation \eqref{webassoc}. When $f=1_*\splits 1_*$, the proof is similar to Part (3) of \cite[Proposition 3.6]{SW24web} using \eqref{mergesplit} and \eqref{sliders}.

Now to prove (2), we show that any $g\in D_{\la,\mu}$ can be written as a linear combination of elements in $\RPMat_{\la,\mu}$. We proceed via induction on the crossing degree $k:=\cdeg(g)$. Assume that $k=0$, i.e. there is no crossing in $g$. Then $g$ is either already an element of $\RPMat_{\la,\mu}$  or there exists local components of $g$  of the form $\begin{tikzpicture}
[baseline = -1mm,scale=.8,color=\clr]
\draw[-,line width=1.5pt] (0.08,-.8) to (0.08,-.5);
\draw[-,line width=1.5pt] (0.08,.3) to (0.08,.6);
\draw[-,thick] (0.1,-.51) to [out=45,in=-45] (0.1,.31);
\draw[-,thick] (0.06,-.51) to [out=135,in=-135] (0.06,.31);
\draw(-.1,0) \bdot;
\draw(-.4,0)node {$\scriptstyle \omega_\lambda$};
\draw(.25,0) \bdot;
\draw(.55,0)node {$\scriptstyle \omega_\mu$};
\node at (-.22,-.35) {$\scriptstyle a$};
\node at (.42,-.35) {$\scriptstyle b$};
\end{tikzpicture}$. By Lemma~\ref{lem:blamu}, such a local component lies in the span of \eqref{eq:wkdotalambda}. Hence we can rewrite $g$ as a linear combination of diagrams in $\RPMat_{\la,\mu}$.

Now suppose $k>0$ and the claim is true for $\Hom_{\qAW} (\mu,\la)_{<k}$. By Lemma~\ref{lem:movedotsfreely} we may swap positive and negative crossings and move dots freely modulo lower crossing degree terms. Moreover, we may move the merges in $g$ up and the splits in $g$ down via \eqref{sliders} without changing the crossing degree. Thus $g$ is equivalent to certain linear combination of diagrams in  $\RPMat_{\la,\mu}$ modulo $\Hom_{\qAW} (\mu,\la)_{<k}$. Hence the claim (2) follows.

Now the proposition follows from Claims (1) and (2).
\end{proof}

\subsection{Elementary ribbon diagrams of $\qASch$}

In this subsection, we construct a spanning set of $\qASch$, which is the $q$-analogue of the elementary diagram basis in \cite[\S 3.4]{SW24Schur}.

 By definition, an arbitrary object of $\qASch$ is of the form 
\begin{equation}
\label{equ:object-of-Aqschur}
 \lambda^{(0)} \red{u_1} \lambda^{(1)} \red{u_2}  \ldots \lambda^{(\ell-1)} \red{u_{\ell}} \lambda^{(\ell)}    
\end{equation}
where $\lambda^{(i)}$ is a strict composition and $\bfu=(u_1,u_2,\ldots,u_{\ell})$ is a vector in $ \kk^{\ell}$, for some $\ell\in \N$. The case $\ell=0$ should be understood as there is no $u_i$. For $\ell\in \N$, an $(1+\ell)$-multicomposition of $m$ is an ordered $(1+\ell)$-tuple of compositions 
$\mu=(\mu^{(0)}, \mu^{(1)}, \ldots, \mu^{(\ell)})$ 
such that $\sum_{i=1}^{1+\ell}|\mu^{(i)}|=m$. 

We call an $(1+\ell)$-multicomposition $\mu$ strict if all its components are strict compositions (where the partition $\emptyset$ is understood to be strict). Let $\Lambda_{\text{st}}^{1+\ell}(m)$ be the set of all strict $(1+\ell)$-multicompositions of $m$. Write 
\[
\Lambda_{\text{st}}^{1+\ell}:=\bigcup_{m\in \N}\Lambda_{\text{st}}^{1+\ell}(m).
\]

Thus, the set of objects of $\qASch$
can be identified with $ \bigcup_{\ell\in \N} \big(\Lambda_{\text{st}}^{1+\ell}\times \kk^{\ell} \big)$.

\begin{lemma}
\label{lem:Homzero}
Let $(\lambda,\bfu )\in \Lambda_{\text{st}}^{1+\ell}(m)\times \kk^{\ell}$ and $ (\mu,\red{\bfu'})\in \Lambda_{\text{st}}^{1+\ell'}(m')\times \kk^{\ell'}$.
Then $\Hom_{\qASch}((\lambda,\bfu ),  (\mu,\red{\bfu'})) =0$ unless
$m=m',  \ell=\ell',$ and $\bfu=\red{\mathbf u'}.
$
\end{lemma}

\begin{proof}
It follows from the forms of generating morphisms \eqref{generator-affschur}--\eqref{redcrossing} for $\qASch$.
\end{proof}

Thus we only need to consider the case for fixed $m,\ell$ and $\mathbf{u}\in \kk^\ell$. 
We define a forgetful map
\begin{align} \label{muforget}
    \Lambda_{\text{st}}^{1+\ell}(m) \longrightarrow \Lambda_{\text{st}}(m),
    \qquad \mu \mapsto \overline{\mu},
\end{align}
by viewing $\mu$ as a strict $1$-multicomposition (by omitting the empty components).

We shall refer to the diagrams obtained by compositions and tensor products of \eqref{generator-affschur}--\eqref{redcrossing} (and identity morphisms) as {\em red strand dotted ribbon diagrams}. In particular, by our construction, there is no crossing between red strands and a red strand only intersects with leg segments (connected components after removing all the intersections and dots).

\begin{example}
Let $\lambda=((\varnothing), (9),(7))$ and $\mu=((5),(5),(6))$, we have two red strand dotted ribbon diagrams as follows.
     \begin{align}
    \label{examplered}
          \begin{tikzpicture}
        [anchorbase,color=\clr,scale=1.6]
        \draw[-,line width=1pt] (-.5,1) to (-1,0)
                                (-.5,1) to (0,0)
                                (-.5,1) to (1,0);
        \draw[wipe]             (.2,.8) to (-.7,.2)
                                (.4,.8) to (.1,.2)
                                (.6,.8) to (.8,.4);
        \draw[-,line width=1pt] (.5,1) to (-1,0)
                                (.5,1) to (0,0)
                                (.5,1) to (1,0);
        \draw[-,line width=1.5pt] (-.5,1) to (-.5,1.15)
                                (.5,1) to (.5,1.15)
                                (-1,0) to (-1,-.15)
                                (1,0) to (1,-.15)
                                (0,0) to (0,-.15);
        \node at  (-.8,.8){$\scriptstyle 2$};
        \node at  (-.5,.6){$\scriptstyle 3$};
        \node at  (-.2,1){$\scriptstyle 4$};
        \node at  (.2,1){$\scriptstyle 3$};
        \node at  (.5,.6){$\scriptstyle 2$};
        \node at  (.8,.8){$\scriptstyle 2$};
        \draw (-.85,.3) \bdot;
        \node at (-1.1,.3){$\scriptstyle \nu^1$};
        \draw (-.52,.3) \bdot;
        \node at (-.52,.1){$\scriptstyle \nu^2$};
        \draw (-.15,.3) \bdot;
        \node at (-.18,.1){$\scriptstyle \nu^3$};
        \draw (.1,.2) \bdot;
        \node at (.12,.05){$\scriptstyle \nu^4$};
        \draw (.58,.3) \bdot;
        \node at (.58,.1){$\scriptstyle \nu^5$};
        \draw (.88,.3) \bdot;
        \node at (1.13,.3){$\scriptstyle \nu^6$};
        \draw[-,line width=1pt,red] (-.7,1) to (-.7,-.2);
        \node[red] at (-.7,1.15){$\scriptstyle u_1$};
        \node[red] at (-.7,-.35){$\scriptstyle u_1$};
        \draw[-,line width=1pt,red] (0,1) to [out=down,in=up] (.15,.45) to [out=down,in=up] (.5,-.2); 
        \node[red] at (0,1.15){$\scriptstyle u_2$};
        \node[red] at (.5,-.35){$\scriptstyle u_2$};
    \end{tikzpicture},
    \qquad     
    \begin{tikzpicture}
        [anchorbase,color=\clr,scale=1.6]
        \draw[-,line width=1pt] (-.5,1) to (-1,0)
                                (-.5,1) to (0,0)
                                (-.5,1) to (1,0);
        \draw[wipe]             (.2,.8) to (-.7,.2)
                                (.4,.8) to (.1,.2)
                                (.6,.8) to (.8,.4);
        \draw[-,line width=1pt] (.5,1) to (-1,0)
                                (.5,1) to (0,0)
                                (.5,1) to (1,0);
        \draw[-,line width=1.5pt] (-.5,1) to (-.5,1.15)
                                (.5,1) to (.5,1.15)
                                (-1,0) to (-1,-.15)
                                (1,0) to (1,-.15)
                                (0,0) to (0,-.15);
        \node at  (-.8,.8){$\scriptstyle 2$};
        \node at  (-.5,.6){$\scriptstyle 3$};
        \node at  (-.2,1){$\scriptstyle 4$};
        \node at  (.2,1){$\scriptstyle 3$};
        \node at  (.5,.6){$\scriptstyle 2$};
        \node at  (.8,.8){$\scriptstyle 2$};
        \draw (-.85,.3) \bdot;
        \node at (-1.1,.3){$\scriptstyle \nu^1$};
        \draw (-.52,.3) \bdot;
        \node at (-.52,.1){$\scriptstyle \nu^2$};
        \draw (-.15,.3) \bdot;
        \node at (-.18,.1){$\scriptstyle \nu^3$};
        \draw (.1,.2) \bdot;
        \node at (.12,.05){$\scriptstyle \nu^4$};
        \draw (.58,.3) \bdot;
        \node at (.58,.1){$\scriptstyle \nu^5$};
        \draw (.88,.3) \bdot;
        \node at (1.13,.3){$\scriptstyle \nu^6$};
        \draw[-,line width=1pt,red] (-.7,1) to (-.7,-.2);
        \node[red] at (-.7,1.15){$\scriptstyle u_1$};
        \node[red] at (-.7,-.35){$\scriptstyle u_1$};
        \draw[-,line width=1pt,red] (0,1) to [out=down,in=up] (.4,.45) to [out=down,in=up] (.5,-.2); 
        \node[red] at (0,1.15){$\scriptstyle u_2$};
        \node[red] at (.5,-.35){$\scriptstyle u_2$};
    \end{tikzpicture}
    \end{align}
\end{example}

We extend the notion of crossing degree for dotted ribbon diagrams from Definition~\ref{def:Cdeg} for red strand dotted ribbon diagrams.
\begin{definition}
\label{def:Cdegred}
We define the \emph{crossing degree}, denoted by $\cdeg$, of a red strand dotted ribbon diagram as the sum of the crossing degrees of its local components, where we declare locally that $\cdeg (\crossingpos ) = \cdeg (\crossingneg )=a+b$ and $\cdeg =0$ for all other generating morphisms \eqref{generator-affschur}--\eqref{redcrossing}.  
\end{definition}


Let $\Homlesk$ (resp. $\Homleqk$) denote the $\kk$-span of all red strand dotted ribbon diagram with crossing degree $<k$ (resp. $\leq k$). For any $R$ and $R'$ in $\Homleqk$, we write \[R \equivc R' \text{ if }R=R' \mod \Homlesk.\]

\begin{lemma}
    \label{lem:movedotsfreelyred}
    The following $\equivc$-relations hold in $\qASch$:
    \begin{align*}
\begin{tikzpicture}[anchorbase, scale=0.35, color=\clr]
\draw[-,line width=1pt](0,2) to (1.5,-.35);
\draw[-,line width=4pt, white](0,-.35) to (1.5,2); to (1.5,-.35);
\draw[-,line width=1pt] (0,-.35) to (1.5,2);
\draw[-,line width=1pt,color=\cred](.8,2) to[out=down,in=90] (0.2,1) to [out=down,in=up] (.8,-.4);
\draw (.8,-.8) node{$\scriptstyle \red{u}$};
\node at (0, -.7) {$\scriptstyle a$};
\node at (1.5, -.7) {$\scriptstyle b$};
\end{tikzpicture}
~\equivc~ 
\begin{tikzpicture}[anchorbase, scale=0.35, color=\clr]
\draw[-,line width=1pt] (0,-.35) to (1.5,2);
\draw[-,line width=4pt, white] (0,2) to   (1.5,-.35);
\draw[-,line width=1pt](0,2) to   (1.5,-.35);
\draw[-,line width=1pt,color=\cred](.8,2) to
[out=down,in=up] (1.5,1) to [out=down,in=up] (.8,-.4);
\draw (.8,-.7) node{$\scriptstyle \red{u}$};
\node at (0, -.7) {$\scriptstyle a$};
\node at (1.5, -.7) {$\scriptstyle b$};
\end{tikzpicture}
~\equivc~ 
\begin{tikzpicture}[anchorbase, scale=0.35, color=\clr]
\draw[-,line width=1pt](0,2) to   (1.5,-.35);
\draw[wipe] (0,-.35) to (1.5,2);
\draw[-,line width=1pt] (0,-.35) to (1.5,2);
\draw[-,line width=1pt,color=\cred](.8,2) to
[out=down,in=up] (1.5,1) to [out=down,in=up] (.8,-.4);
\draw (.8,-.7) node{$\scriptstyle \red{u}$};
\node at (0, -.7) {$\scriptstyle b$};
\node at (1.5, -.7) {$\scriptstyle a$};
\end{tikzpicture}
~\equivc~
\begin{tikzpicture}[anchorbase, scale=0.35, color=\clr]
\draw[-,line width=1pt] (0,-.35) to (1.5,2);
\draw[wipe](0,2) to (1.5,-.35);
\draw[-,line width=1pt](0,2) to (1.5,-.35);
\draw[-,line width=1pt,color=\cred](.8,2) to[out=down,in=90] (0.2,1) to [out=down,in=up] (.8,-.4);
\draw (.8,-.8) node{$\scriptstyle \red{u}$};
\node at (0, -.7) {$\scriptstyle b$};
\node at (1.5, -.7) {$\scriptstyle a$};
\end{tikzpicture}.
    \end{align*}
\end{lemma}

\begin{proof}
    Follows from \eqref{redbraid} directly.
\end{proof}

By Lemmas~\ref{lem:movedotsfreely} and \ref{lem:movedotsfreelyred}, we may identify positive crossings and negative crossings modulo lower crossing degree terms.  By \eqref{dotmovesplitabr} and Lemma~\ref{lem:movedotsfreely}, we may freely move the dots through crossings, merges and splits modulo lower crossing degree terms. We will use these facts frequently in what follows without further reference.

Recall from \eqref{examplereduced} and Definition \ref{def:dottedribbon} the set of elementary (chicken foot) ribbons. Recall the forgetful map \eqref{muforget} which sends $\mu$ to $\overline{\mu}$. 

\begin{definition}
Suppose that $\lambda,\mu \in \Lambda_{\text{st}}^{1+\ell}(m)$.
A red strand dotted ribbon diagram $R$ from $\mu$ to $\lambda$ is  called an elementary ribbon diagram if  
 \begin{enumerate}
     \item
     the ribbon $\overline R $ obtained from $R$ by removing the red strands is an elementary chicken foot ribbon diagram in  $\RPMat_{\bar\lambda,\bar\mu}$, and 
     \item
     there is at most one intersection between each red strand and any other strand.
 \end{enumerate}   
In this case, we call $R$ a $\la \times\mu$-ornamentation of $\overline R$. 
\end{definition}

For example, \eqref{examplered} are two examples of $\lambda\times\mu$-ornamentations of \eqref{examplereduced}. In particular, they are equivalent modulo lower crossing degrees according to Lemma~\ref{lem:movedotsfreelyred}. 

To construct the spanning set of $\Hom_{\qASch}(\mu,\lambda)$, we first fix an $\lambda \times \mu$-ornamentation for each elementary chicken foot ribbon diagram in $\RPMat_{\bar\lambda,\bar\mu}$ up to lower crossing degree terms. 

For $\lambda,\mu \in \Lambda_{\text{st}}^{1+\ell}(m)$, we introduce
\begin{align}  \label{RPMatblock}
\RPMat_{\lambda,\mu} = \left\{\big(A=(A_{pq})_{0\le p,q\le \ell}, P=(P_{pq})_{0\le p,q\le \ell} \big) \mid (A , P ) \in  \RPMat_{\bar\lambda ,\bar\mu }  \right\}. 
\end{align}
Recalling the forgetful map \eqref{muforget}, we define the following forgetful map (which is indeed a bijection) 
\begin{align}  \label{bijPM}
    \RPMat_{\la,\mu} \longrightarrow \RPMat_{\overline\lambda, \overline\mu},
    \qquad (A,P) \mapsto (A,P),
\end{align}
by viewing a block matrix as a standard matrix; the map is bijective as there is a unique way to view a matrix in $\PMat_{\overline\lambda, \overline\mu}$ as a block matrix in $\PMat_{\lambda,\mu}$ where the blocks are determined by the fixed $(1+\ell)$-multicompositions.

For each elementary chicken foot ribbon diagram $\overline{R}$  encoded by $(A,P)\in\RPMat_{\overline\lambda, \overline\mu}$, there is a unique elementary ribbon diagram (up to lower crossing degree terms) encoded by the preimage of $(A,P)$ under \eqref{bijPM}.

\begin{example}
    Let $\lambda=((\varnothing),(9),(7))$ and $\mu=((5),(5),(6))$. Then \eqref{examplered} are two (equivalent) ornamentations of $(A,P)$ with
\[A=\left(\begin{array}{c|c|c}
2 & 3 & 4 \\
\hline
3& 2 & 2 \\  
\end{array}\right),
\quad  
P=  \left(\begin{array}{c|c|c}
\nu^1 & \nu^3 & \nu^5 \\
\hline
\nu^2& \nu^4 & \nu^6 \\  
\end{array}\right).\]
The forgetful map in \eqref{bijPM} sends $(A,P)$
to 
\[A=\left(\begin{array}{ccc}
2 & 3 & 4 \\
3& 2 & 2 \\  
\end{array}\right),\quad  P=  \left(\begin{array}{ccc}
\nu^1 & \nu^3 & \nu^5 \\
\nu^2& \nu^4 & \nu^6 \\  
\end{array}\right).\]
\end{example}

\begin{definition}
\label{def:dotdegred}
    For any red strand dotted ribbon diagram, we define its dot degree, denoted by $\ddeg$, to be the sum of the dot degrees of its local components, where we declare locally that $\ddeg (\leftcrossing) =a$ and $\ddeg (\omega_{a,\nu}) =|\nu_1|+\ldots+|\nu_k|$, for $\nu\in \RPar_a$.
\end{definition}

\begin{proposition} \label{prop:spanqASch}
For any $\lambda,\mu\in \Lambda_{\text{st}}^{1+\ell}(m)$, $\Hom_{\qASch}(\mu,\lambda)$ is spanned by $\RPMat_{\la,\mu}$ consisting of all $\lambda \times \mu$ elementary chicken foot diagrams from $\mu$ to $\la$. 
\end{proposition}

\begin{proof}
    It suffices to show that $fg$ can be written as a linear combination of elementary ribbon diagrams up to lower crossing degree and dot degree terms  for an arbitrary elementary ribbon diagram $g$ and a generating morphism $f$ of the following five types: 
\[ 
1_* \wkdota 1_*,\qquad  1_*\merge 1_*,\qquad  1_*\splits 1_*,   \qquad  1_*\rightcrossing1_*, \qquad   1_*\leftcrossing1_*,
\]
where $1_*$ represents suitable identity morphisms.
   We proceed by induction on $k=\ddeg(g)$. When $k=0$, for merges and splits the claim follows from \cite[(6.18)--(6.19)]{Bru24}. Note that all ornamentations of a given elementary  chicken foot ribbon diagram represent the same morphism by \eqref{redslider} and Lemma \ref{adamovecrossings} since there are no traverse-down.
Thus it is clear that the left multiplication of the traverse-up is still an ornamentation of some dot degree zero elementary ribbon diagram and hence the claim holds for $k=0$. For $k>0$, the rest of the proof follows similarly the proof in \cite[Theorem 3.13]{SW24Schur}, and we skip the detail.
\end{proof}

\subsection{Basis for $\qASch$}

\begin{theorem}
  \label{thm:basisASchur}
$\RPMat_{\lambda,\mu}$ forms a basis for  $\Hom_{\qASch}(\mu,\lambda)$, for any $\lambda,\mu \in \Lambda_{\text{st}}^{1+\ell}(m)$.    
\end{theorem}

\begin{proof}
    By Proposition~\ref{prop:spanqASch}, it remains to prove that $\RPMat_{\lambda,\mu}$ is linearly independent. To that end, it suffices to show that every finite subset in $\RPMat_{\lambda,\mu}$ is linearly independent. Fix a subset $\mathcal{T}\subset\RPMat_{\lambda,\mu}$. Let 
    \[ \overline{\mathcal{T}}=\{(A_{\alpha},(\nu_{ij})_{\alpha})\in\RPMat_{\overline{\lambda},\overline{\mu}}\mid \alpha=1,\ldots,c\} \]
    denote the collection of images of elements in $\mathcal{T}$ under the forgetful map.

    Let $B$ be an ornamentation associated with $\overline{B}=(A_{\alpha},(\nu_{ij})_{\alpha})\in \overline{\mathcal{T}}$. Suppose $l(\mu^{(i)})=h_i$, $l(\lambda^{(i)})=t_i$. Let $h:=l(\overline{\mu})=\sum_{k=0}^\ell h_k$ and $t:=\sum_{k=0}^\ell t_k$. Define $p_0=0$ and $p_j=\sum_{i=1}^j\overline{\mu}_i$ for $1\leq j\leq h$. Then we consider the monomial
    \[m=\prod_{k=1}^h X^{kN}_{p_{k-1}+1}\cdots X^{kN}_{p_{k-1}+2}X^{kN}_{p_{k}}\in \Sym_{\overline{\mu}}\]
    where $N\gg 0$ is chosen such that the leading term of $B'$ acting on $m$ is a polynomial for any $B'\in \mathcal{T}$. (Such $N$ exists since $\mathcal{T}$ is finite.) 

    By \eqref{eq:wrblackwhite}, the leading term of $\omega_{a,\nu}. X_b^lX_{b+1}^l\cdots X_{b+a}^l$ is \[q^*X_b^{l+\nu_1'}X_{b+1}^{l+\nu_2'}\cdots X_{b+a}^{l+\nu_a'}+(q-q^{-1})f\] where $q^*$ is certain power of $q$, $f$ is a symmetric Laurent polynomial in $X_b,\ldots,X_{b+a}$ and $\nu'_i$ is the number of positive components $\nu_k$ in $\nu$ such that $|\nu_k|\geq i$ subtracting the number of negative ones such that $|\nu_k|\geq i$. In particular, for any rational partition $\nu=(\nu_1,\cdots,\nu_l)\in\RPar_{\overline{\lambda},\overline{\mu}}$, we can define its dual to be $\nu'=(\nu'_1,\cdots,\nu'_{l'})$.

    Using this together with the actions of splits, merges, crossings, traverse-downs and traverse-ups (see \eqref{map_SM}--\eqref{map_redcross}), we conclude that the leading term of $B.m$  is
    \[
q^*X^{((hN)^{a_{1h}})+\nu'_{1h}+((r_{1h})^{a_{1h}})}_{1,\ldots, a_{1h}}
X^{((h-1)N)^{a_{1h-1}}+ \nu'_{1h-1}+(r_{1h-1})^{a_{1h-1}})}_{a_{1h}+1, \ldots, a_{1h}+a_{1h-1}}
\ldots X_{p_h-a_{t1}+1,\ldots, p_h}^{(N)^{a_{t1}} + \nu'_{t1} +(r_{t1}^{a_{t1}})}+(q-q^{-1})f',
\]
where 
\begin{itemize}
        \item $A_{\alpha}=(a_{ij})_{\alpha}$ and $f'$ is a symmetric Laurent polynomial,
       \item $X_{a+1,\ldots, a+b}^{(d_1,\ldots, d_b)}=X_{a+1}^{d_1}\ldots X_{a+b}^{d_b}$,
    \item $r_{ij}$ is the number of traverse-downs which produced by the strand $a_{ij}$.
\end{itemize}
This shows that for different $\overline B$'s, the ornamentation $B$ will produce different leading terms which are clearly linearly independent. This implies the linear independence of $\RPMat_{\lambda,\mu}$. 
\end{proof}

For any $\ell, m \ge 1$ and $\bfu=(u_1,\ldots, u_\ell)$, the path algebras of the category $\qASch$ are defined to be 
\begin{align} \label{affineSchur_red}
    \qASch (m):=\oplus_{\lambda,\mu\in \Lambda^{1+\ell}_{st}(m)}\Hom_{\qASch}(\lambda,\mu). 
\end{align}
    The following can be read off from the proof of the above basis theorem. 
\begin{corollary}
\label{Cor:faithfulrep}
 The functor $\mathcal F$ in Theorem  \ref{th:ASchRep} gives a faithful representation of $\qASch (m)$, for any $\bfu$ and $m$.
\end{corollary}

\subsection{Basis for $\qAW$}

We now derive a basis theorem for $\qAW$ from its counterpart for $\qASch$. 

\begin{theorem} 
 \label{thm:basisqAW}
For any $\lambda,\mu \in \Lambda_{\text{st}}(m)$, $\Hom_{\qAW}(\mu,\lambda)$ has a basis $\RPMat_{\lambda,\mu}$ which consists of all elementary ribbons from $\mu$ to $\la$. Moreover, $\qAW$ is a full subcategory of $\qASch$. 
\end{theorem}

\begin{proof}
It is proved in Proposition \ref{prop:spanqAW} that $\Hom_{\qAW}(\mu,\lambda)$ is spanned by $\RPMat_{\lambda,\mu}$. Furthermore, considering the obvious functor from $\qAW$ to $\qASch$ sending the generating morphisms with the same names, we see that $\RPMat_{\lambda,\mu}$ are also linear independent by  Theorem \ref{thm:basisASchur}. The last statement is now clear. 
\end{proof}

For any $m \ge 1$, the path algebras of the category $\qAW$ are defined to be 
\begin{align} \label{affineSchur_black}
    \qAW (m):=\oplus_{\lambda,\mu\in \Lambda_{st}(m)}\Hom_{\qAW}(\lambda,\mu). 
\end{align}

Recall the algebras $D_a$ from Definition~\ref{def:Da}. A special case of \cref{thm:basisqAW} provides detailed structures on the endomorphism algebra of $\stra$. 

\begin{corollary}
\label{cor:End1a}
Let $a\ge 1$. Then
\begin{enumerate}
    \item 
    The $\kk$-algebra $\End_{\qAW}(\stra)$ has a basis $\{\omega_{a,\lambda} \mid \lambda\in \RPar_a\}.$
    \item
    $\End_{\qAW}(\stra)$  is a commutative $\kk$-algebra generated by $\omega_{a,r}$, for $-a\le r \le a$.
\end{enumerate}
\end{corollary}
\begin{proof}
Part (1) follows from \cref{thm:basisASchur} for $\la=\mu=(a)$ and $\RPMat_{(a),(a)} =\RPar_a$. 
    Part (2) follows by (1) and Lemma~\ref{lem:spanDa}.
\end{proof}

\subsection{Some algebra isomorphisms}
\label{subsec:iso-MS}

A $\kk$-linear monoidal category called higher level affine Hecke category was introduced in \cite{MS21}, whose path algebras are called higher level affine Hecke algebras \cite{Web20}. This category appears as a monoidal subcategory of $\qASch$ with generating objects being a thin strand $\str$ and $\stru$ ($u\in \kk^*$) and generating morphisms 
\[
\begin{tikzpicture}[baseline=-1mm, color=\clr]
	\draw[-,thick] (0.3,-.3) to (-.3,.4);
	\draw[-,line width=4pt,white] (-0.3,-.3) to (.3,.4);
	\draw[-,thick] (-0.3,-.3) to (.3,.4);
        \node at (-0.3,-.45) {$\scriptstyle 1$};
        \node at (0.26,-.45) {$\scriptstyle 1$};
\end{tikzpicture}
\text{ (with inverse}
\begin{tikzpicture}[baseline=-1mm, color=\clr]
 \draw[-,thick] (-0.3,-.3) to (.3,.4);
	\draw[-,line width=4pt,white] (0.3,-.3) to (-.3,.4);
 \draw[-,thick] (0.3,-.3) to (-.3,.4);
        \node at (-0.33,-.45) {$\scriptstyle 1$};
        \node at (0.3,-.45) {$\scriptstyle 1$};
\end{tikzpicture}), 
\quad 
\dotgen,  \quad 
\begin{tikzpicture}[anchorbase, scale=1, color=\clr]
 \draw[-,thick] (-0.3,.3) to (.3,1);
\draw[-,line width=1pt,color=\cred] (0.3,.3) to (-.3,1);
\draw(-.3,0.2) node{$\scriptstyle 1$};
\draw (.3, 0.2) node{$\scriptstyle \red{u}$};
\end{tikzpicture} , 
\quad
\begin{tikzpicture}[anchorbase, scale=1, color=\clr]
 \draw[-,line width=1pt,color=\cred] (-0.3,.3) to (.3,1);
\draw[-,thick] (0.3,.3) to (-.3,1);
\draw(-.3,0.2) node{$\scriptstyle \red{u}$};
\draw (.3, 0.2) node{$\scriptstyle 1$};
\end{tikzpicture}. 
\]
(To match their convention of affine Hecke algebra with ours, one identifies their $H_i, X_i$ and $q$ with our $q^{-1}H_i, X_i$ and $q^{-2}$, respectively.)

Following \cite{DJM98}, the authors in \cite{MS21} further defined a higher level affine $q$-Schur algebra as the endomorphism algebra of a direct sum of permutation modules over a higher level affine Hecke algebra. They formulated a set of generating morphisms including the merges, splits, and (mix-color) crossings, but did not formulate the thick dots $\wkdotaa, \wkdota$ as we do; they did not specify all relevant relations either. 

Recall that earlier R. Green \cite{Gre99} formulated the affine $q$-Schur algebra as the endomorphism algebra of a direct sum of permutation modules over the affine Hecke algebra associated with the symmetric group $S_m$. 

\begin{proposition}
\label{prop:path_MSchur}
\begin{enumerate}
    \item 
    The path algebras $\qAW (m)$ in \eqref{affineSchur_black} are isomorphic to the affine $q$-Schur algebras. 
    \item
The path algebras $\qASch (m)$ in \eqref{affineSchur_red} are isomorphic to the higher level affine $q$-Schur algebras (of rank $m$ with the same parameter $\bfu$). 
\end{enumerate}
\end{proposition}

\begin{proof}
By \cref{thm:basisqAW}, $\qAW$ is a full subcategory of $\qASch$ and hence Part (1) becomes a special case of (2). 

Let us prove (2). According to \cite[Lemma~ 4.40]{MS21}, the higher level affine $q$-Schur algebra (of rank $m$ with same parameter $\bfu$) has a faithful representation on the symmetric Laurent polynomials such that the image is generated by splits, merges, right-crossings (= traverse-downs), left-crossings (= traverse-ups), left multiplications of symmetric Laurent polynomials and hence the same as $\qASch (m)$. Thus, they are isomorphic by Corollary
~\ref{Cor:faithfulrep}.
\end{proof}
\section{Cyclotomic $q$-web and cyclotomic $q$-Schur categories}
\label{sec:cyclotomic}

In this section, we introduce the cyclotomic $q$-web category $\qW_{\bfu}$ and  cyclotomic $q$-Schur category $\qSchu$ associated with any fixed parameter  $\bfu \in (\kk^*)^\ell.$ We construct an elementary ribbon diagram basis and a double SST basis for the cyclotomic $q$-Schur category $\qSchu$. We establish a functor from $\qSchu$ to a small $\kk$-linear category $\qSchuDJM$ reformulated from the Dipper-James-Mathas' cyclotomic $q$-Schur algebras, and prove that it is an isomorphism by matching the double SST basis of $\qSchu$ with the cellular basis of $\qSchuDJM$.

\subsection{Definitions}

For any $r\ge 1$ and $u\in \kk$, we define
\begin{equation}
    \label{def-gau}
 g_{r}(u):= 
\sum_{0\le i\le r} (-u)^i q^{-i(r-i)} \:\:
 \begin{tikzpicture}[anchorbase,scale=.7,color=\clr]
\draw[-,line width=1.5pt] (0.08,-.6) to (0.08,.5);
\node at (.08,-.8) {$\scriptstyle r$};
\draw(0.08,0) \bdot;
\draw(.7,0)node {$\scriptstyle \omega_{r-i}$};
\end{tikzpicture}.
\end{equation}

\begin{lemma}
 \label{lem:gru}
 For any $r\ge 1$ and $u\in \kk$, the following relation (called a {\em balloon relation}) holds in $\qAW$ and  $\qASch$:
\[
\begin{tikzpicture}[baseline = 5pt, scale=.5, color=\clr]
\draw[-, line width=1.2pt] (0.5,2) to (0.5,2.3);
\draw[-, line width=1.2pt] (0.5,0) to (0.5,-.3);
\draw[-,thin](0.5,2) to[out=left,in=up] (-.5,1) to[out=down,in=left] (0.5,0);
\draw[-,thin]  (0.5,2) to[out=left,in=up] (0,1) to[out=down,in=left] (0.5,0);   
\draw[-,thin] (0.5,0)to[out=right,in=down] (1.5,1)to[out=up,in=right] (0.5,2);
\draw[-,thin] (0.5,0)to[out=right,in=down] (1,1)
 to[out=up,in=right] (0.5,2);
\node at (0.5,.7){$\scriptstyle \cdots$};
\draw (-0.5,1) \bdot; 
 \node at (-.75,1) {$\scriptstyle g$}; 
\draw (0,1) \bdot; 
\node at (0.3,1) {$\scriptstyle g$};
\draw (1,1) \bdot;
\draw (1.5,1) \bdot; 
\node at (1.8,1) {$\scriptstyle g$};
\node at (-.4,0) {$\scriptstyle 1$};
 \node at (.2,0.3) {$\scriptstyle 1$};
\node at (.7,0.3) {$\scriptstyle 1$};
\node at (1.2,0) {$\scriptstyle 1$};
\draw (.5,-.5) node{$\scriptstyle {r}$};
\end{tikzpicture}
=
[r]!\, g_{r}(u),
\qquad \text{ where }
 \begin{tikzpicture}[baseline = 3pt, scale=0.5, color=\clr]
\draw[-,thin] (0,0) to[out=up, in=down] (0,1.4);
\draw(0,0.6) \bdot; 
\draw (0.65,0.6) node {$\scriptstyle g$};
\node at (0,-.3) {$\scriptstyle 1$};
\end{tikzpicture} 
=
\begin{tikzpicture}[baseline = 3pt, scale=0.5, color=\clr]
\draw[-,thin] (0,0) to[out=up, in=down] (0,1.4);
\draw(0,0.6) \bdot; 
\draw (0.65,0.6) node {$\scriptstyle \omega_1$};
\node at (0,-.3) {$\scriptstyle 1$};
\end{tikzpicture}
-u 
\begin{tikzpicture}[baseline = 3pt, scale=0.5, color=\clr]
\draw[-,thin] (0,0) to[out=up, in=down] (0,1.4);
\node at (0,-.3) {$\scriptstyle 1$};
\end{tikzpicture} .
\]
\end{lemma}

\begin{proof}
    Follows by induction similar to the proof of \cite[Lemma 2.11]{SW24web}.
\end{proof}

Let
\begin{equation}
\label{equ:parau}
    \bfu=(  u_1, \ldots, u_\ell)\in (\kk^*)^\ell.
\end{equation}

\begin{definition}
 \label{def:cycWeb}
    Suppose  $\ell\ge 1$. For any $\bfu=(u_1,\ldots, u_\ell) $ in \eqref{equ:parau}, the cyclotomic web category $\qW_\bfu$ is the quotient category of $\qAW$ by the right tensor ideal generated by $\prod_{1\le j\le \ell }g_{r}(u_j)$, for $r\ge 1$. 
\end{definition}

In the following, we will identify 
$\lambda^{(0)}\red{u_1}\ldots \red{u_\ell} \lambda^{(\ell)}$ with $\lambda=(\lambda^{(0)},\ldots, \lambda^{(\ell)})$ once $\bfu$ is fixed in \eqref{equ:parau}.

Let $\qSchu$ be the subcategory of $\qASch$ with objects $\Lambda_{st}^{1+\ell}$.
\begin{definition}
\label{def:qcSchur}
    For any fixed $\bfu$ in \eqref{equ:parau}, the cyclotomic $q$-Schur category is the quotient of the $\kk$-linear category $\qSchu$
    by the relations 
    \begin{equation}
    \label{eq:unitlambda}
       \unit{\lambda}=0 \text{ for all } \lambda =(\la^{(0)},\la^{(1)},\ldots,\la^{(\ell)}) \in \Lambda_{\text{st}}^{1+\ell} \text{ with }
\lambda^{(0)}\neq \emptyset.
    \end{equation}
\end{definition}

\begin{rem}
    The commuting relations via obvious diagrammatic isotopes that do not change the combinatorial type of the diagram 
    are implicitly imposed in Definition~\ref{def:qcSchur}. For example,  
\begin{equation}
\label{commurelation}
     ~\begin{tikzpicture}[anchorbase, scale=0.45, color=\clr] 
               \draw[-,line width=1pt] (-2,0.2) to (-1.5,1);
	\draw[-,line width=1pt] (-1,0.2) to (-1.5,1);
\draw[-,line width=1.5pt] (-1.5,1) to (-1.5,1.8);
 \draw (-1,0) node{$\scriptstyle {b}$};
\draw (-2,0) node{$\scriptstyle {a}$};
\draw (-1.5,2.2) node{$\scriptstyle {a+b}$};
               \end{tikzpicture}~
\begin{tikzpicture}[anchorbase, scale=0.45, color=\clr] 
\draw[-,line width=1pt] (-2,0.2) to (-1,1.8);
\draw[wipe](-1,0.2) to (-2,1.8);
\draw[-,line width=1pt] (-1,0.2) to (-2,1.8); 
\draw (-1,0) node{$\scriptstyle {c}$};
 \draw (-2,0) node{$\scriptstyle {d}$};
\draw (-1,2.2) node{$\scriptstyle {d}$};
\draw (-2,2.2) node{$\scriptstyle {c}$};
\end{tikzpicture}~ ~
:= ~  
~\begin{tikzpicture}[anchorbase, scale=0.4, color=\clr] 
 \draw[-,line width=1pt] (-2,0.7) to (-1.5,1.5);
\draw[-,line width=1pt] (-2,0.7) to (-2,-.3);  \draw[-,line width=1pt] (-1,0.7) to (-1,-.3);   
\draw[-,line width=1pt] (-1,0.7) to (-1.5,1.5);
\draw[-,line width=1.5pt] (-1.5,1.5) to (-1.5,2.3);
 \end{tikzpicture}
~\begin{tikzpicture}[anchorbase, scale=0.4, color=\clr] 
\draw[-,line width=1pt] (-2,-.3) to (-1,.7);
\draw[-,line width=1pt] (-1,.7) to (-1,2.3);
\draw[-,line width=1pt] (-2,.7) to (-2,2.3);
\draw[wipe] (-1,-0.3) to (-2,.7); 
\draw[-,line width=1pt] (-1,-0.3) to (-2,.7); 
\end{tikzpicture}~
~=~ 
~\begin{tikzpicture}[anchorbase, scale=0.4, color=\clr] 
 \draw[-,line width=1pt] (-2,-0.3) to (-1.5,.6);  
\draw[-,line width=1pt] (-1,-0.3) to (-1.5,.6);
\draw[-,line width=1.5pt] (-1.5,.6) to (-1.5,2.3);
 \end{tikzpicture}~
\begin{tikzpicture}[anchorbase, scale=0.4, color=\clr] 
 \draw[-,line width=1pt] (-2,.8) to (-1,2.3);
 \draw[wipe]  (-1,.8) to (-1,-.3);
 \draw[-,line width=1pt] (-1,.8) to (-1,-.3);
  \draw[wipe] (-2,.8) to (-2,-.3);
\draw[-,line width=1pt] (-2,.8) to (-2,-.3);
  \draw[wipe] (-1,.8) to (-2,2.3); 
\draw[-,line width=1pt] (-1,.8) to (-2,2.3); 
\end{tikzpicture}~.
\end{equation} 
\end{rem}

\begin{rem}
    Any diagram we will use is obtained from the identity morphism by replacing some parts with some other diagrams.  
To simplify the notation, we will only draw the different parts to denote the whole diagram if it is clear for us to draw the whole diagram from this.
 For example, a diagram of the form 
 \[
 \begin{tikzpicture}[baseline = 10pt, scale=0.5, color=\clr]
\draw[-,line width=1.5pt] (-8,0.5)to[out=up,in=down](-8,2.2);
\draw (-7,1) node {$\ldots$};
\draw (-8,0) node{$ \scriptstyle{\mu^{(0)}_1}$};   
\draw[-,line width=1.5pt] (-6,0.5)to[out=up,in=down](-6,2.2);
\draw (-6,0) node{$\scriptstyle {\mu^{(0)}_{h_0}}$};
\draw[-,line width=1pt, color=\cred](-5,0.2) to (-5,2.3);
\draw (-5,-.2) node{$\scriptstyle \red{u_1}$};
\draw[-,line width=1.5pt] (-4,0.5)to[out=up,in=down](-4,2.2);
\draw (-3,1) node {$\ldots$};
\draw (-4,0) node{$ \scriptstyle{\mu^{(1)}_1}$};
\draw[-,line width=1pt] (-2,1)to[out=up,in=down](-2,1.8) to (2,1.8)to (2,1) to (-2,1);
\draw (-2,0) node{$ \scriptstyle{\mu^{(i)}_j}$}; 
\draw (2,0) node{$ \scriptstyle{\mu^{(i)}_l}$};
\draw[-,line width=1.5pt](-1.7,.5) to(-1.7,1);
\draw[-,line width=1.5pt](1.5,.5) to(1.5,1);
\draw[-,line width=1.5pt](-1.7,1.8) to(-1.7,2.2);
\draw[-,line width=1.5pt](1.5,1.8) to(1.5,2.2);
\draw (0,.6) node {$\ldots$};
\draw (0,2) node {$\ldots$};
\draw[-,line width=1pt, color=\cred](4,0.2) to (4,2.3);
\draw (4,-.2) node{$\scriptstyle \red{u_\ell}$};
\draw (3,1) node {$\ldots$};
\draw[-,line width=1.5pt] (5,0.5)to[out=up,in=down](5,2.2);
\draw (6,1) node {$\ldots$};
\draw (5,0) node{$ \scriptstyle{\mu^{(\ell)}_1}$};   
\draw[-,line width=1.5pt] (7,0.5)to[out=up,in=down](7,2.2);
\draw (7,0) node{$\scriptstyle {\mu^{(\ell)}_{h_\ell}}$};
\end{tikzpicture},  
\]
can be expressed concisely as 
\[
\begin{tikzpicture}[baseline = 10pt, scale=0.5, color=\clr]
\draw[-,line width=1pt] (-2,1)to[out=up,in=down](-2,1.8) to (2,1.8)to (2,1) to (-2,1);
\draw (-2,0) node{$ \scriptstyle{\mu^{(i)}_j}$}; 
\draw (2,0) node{$ \scriptstyle{\mu^{(i)}_l}$};
\draw[-,line width=1.5pt](-1.7,.5) to(-1.7,1);
\draw[-,line width=1.5pt](1.5,.5) to(1.5,1);
\draw[-,line width=1.5pt](-1.7,1.8) to(-1.7,2.2);
\draw[-,line width=1.5pt](1.5,1.8) to(1.5,2.2);
\draw (0,.6) node {$\ldots$};
\draw (0,2) node {$\ldots$};
\end{tikzpicture}~.
\]
\end{rem}

 By \eqref{eq:unitlambda}, it will be convenient to identify the set of objects $\Lambda_{\text{st}}^{\ell}(m)$ of $\qSchu$ with the subset 
$\Lambda_{\text{st}}^{\emptyset,\ell}(m)$ of $\Lambda_{\text{st}}^{1+\ell}(m)$, where 
\begin{align}
\label{def:barlambdaum}
\Lambda_{\text{st}}^{\emptyset,\ell}(m) &:=\{\lambda\in \Lambda_{\text{st}}^{1+\ell}(m)
\mid \lambda^{(0)}=\emptyset\},
\\
\label{def:barlambdau}
\Lambda_{\text{st}}^{\emptyset,\ell} &:=\bigcup_{m\in \N} \Lambda_{\text{st}}^{\emptyset,\ell}(m).    
\end{align}

For $1\le i\le \ell$ and $r\in\Z_{> 0}$, we define 
\begin{align}
  \label{eq:gri}
g_{r,i}=\prod_{1\le j\le i}g_{r}(u_j).
\end{align}
Then we may draw $g_{r,i}$ as a diagram
\begin{tikzpicture}[baseline = 10pt, scale=0.4, color=\clr]
\draw[-,line width=1pt] (-1,2) to (-1,0.2);
\draw (-1,1) \bdot;
\draw (-0.25,1) node{$\scriptstyle {g_{r,i}}$};
\draw (-1,-.2) node{$\scriptstyle {r}$}; 
\end{tikzpicture} 
which is the vertical concatenation of 
\begin{tikzpicture}[baseline = 10pt, scale=0.4, color=\clr]
\draw[-,line width=1pt] (-1,2) to (-1,0.2);
\draw (-1,1) \bdot;
\draw (0.2,1) node{$\scriptstyle {g_{r}(u_j)}$};
\draw (-1,-.2) node{$\scriptstyle {r}$}; 
\end{tikzpicture}, for $1\le j\le i$.
Note that $\begin{tikzpicture}[baseline = 10pt, scale=0.4, color=\clr]
\draw[-,line width=1pt] (-1,2) to (-1,0.2);
\draw (-1,1) \bdot;
\draw (0.2,1) node{$\scriptstyle {g_{a}(u)}$};
\draw (-1,-.2) node{$\scriptstyle {a}$}; 
\end{tikzpicture}$ has already appeared in \eqref{redcross2} when we define $\qASch$.
\begin{lemma}
\label{lem:cycpolyvanish}
The following relation holds in $\qSchu$:
\begin{equation}
\label{cyclotomicpolyi}
~
\begin{tikzpicture}[baseline = 10pt, scale=0.5, color=\clr]
\draw[-,line width =1pt,color=\cred] (-4.5,.2) to (-4.5,2);
\draw (-4.5,-.2) node{$\scriptstyle \red{u_1}$};    
\draw[-,line width =1pt,color=\cred] (-3.5,.2) to (-3.5,2);
\draw (-3.5,-.2) node{$\scriptstyle \red{u_2}$};
\draw(-2.6,.8) node{$\ldots$};
\draw[-,line width =1pt,color=\cred] (-2,.2) to (-2,2);
\draw (-2,-.2) node{$\scriptstyle \red{u_{i}}$};
\draw[-,line width=1pt] (-1,2) to (-1,0.2);
\draw (-1,1) \bdot;
\draw (-.25,1) node{$\scriptstyle {g_{r,i}}$};
\draw (-1,-.2) node{$\scriptstyle {r}$}; 
\end{tikzpicture}
~1_*
=0, 
\qquad 1\le i\le \ell,
\end{equation}
where $1_*$ represents any suitable identity morphisms. 
\end{lemma}

\begin{proof}
    Follows by induction on $i$ with the initial case for $i=1$ proved by using \eqref{redcross2}. The detailed argument is the same as for \cite[Lemma~4.6]{SW24Schur} and will be skipped. 
\end{proof}

\subsection{Cyclotomic $q$-Schur algebras}

The cyclotomic Hecke algebra $\Hu$ with parameter $\bfu$ is by definition the quotient of the affine Hecke algebra $\widehat{\mathcal{H}}_m$ by the two-sided ideal generated by $\prod_{1 \le i \le \ell}(X_1 - u_i)$; cf. \cite{Ar96}.     

Let $\lambda = (\lambda^{(1)}, \ldots, \lambda^{(\ell)}) \in \Lambda_{\text{st}}^\ell(m)$ be a multicomposition. For each $i$, set 
$a_i=\sum_{1\le j\le i}|\lambda^{(j)}|$, for $1\le i\le \ell-1$, and define 
\begin{equation}\label{def-pilambda}
 \pi_\lambda:= \pi_{a_1,1} \pi_{a_2,2} \cdots \pi_{a_{\ell-1},\ell-1},   
\end{equation}
where $\pi_{a,i}=\prod_{1\le j\le a}(X_j-u_{i+1})$.
Recall the Young subgroup $\mathfrak{S}_\lambda$ of $\mathfrak{S}_m$ associated with $\lambda$ and let $\x_\lambda = \sum_{w \in \mathfrak{S}_\lambda} q^{-\ell(w)}H_w$. Define \begin{equation} \label{def-mlambda} m_\lambda := \pi_{\lambda} \x_{\lambda} = \x_\lambda \pi_\lambda. \end{equation}
For any $\lambda\in \Lambda_{\text{st}}^\ell(m)$, we refer to the right $\Hu$-module 
\[
M^\lambda=m_\lambda \Hu
\]
as a permutation module.

An $\ell$-multicomposition $\la$ is an $\ell$-multipartition if each of its components $\la^{(i)}$'s is a partition. Let $\Par^\ell(m)$  be the set of all $\ell$-multipartitions  of $m$. We have $\Par^\ell(m)\subset \Lambda_{\text{st}}^\ell(m)$.

We can define the cyclotomic $q$-Schur algebra as in \cite[Definition (6.1)]{DJM98} since $\Par^\ell(m)\subset \Lambda_{\text{st}}^\ell(m)$. Suppose that $\bfu =(u_1,\ldots, u_{\ell})\in \kk^{\ell}$ is fixed as for the cyclotomic Hecke algebra $\Hu$. The cyclotomic $q$-Schur algebra is the endomorphism algebra 
     \begin{equation} \label{eq:SalgDJM}
        \Smu:= \End_{\Hu}\Big(\bigoplus_{\mu\in \Lambda_{\text{st}}^\ell(m)}M^\mu\Big).
     \end{equation}
Clearly we have $\Smu\cong \oplus_{\mu,\nu\in\Lambda_{\text{st}}^\ell(m) }\Hom_{\Hu}(M^\nu,M^\mu)$. Hence we may regard  any $\phi\in \Hom_{\Hu}(M^\nu,M^\mu)$ as an element of $\Smu$ by extension of zero. We further denote 
   \begin{align}
   \Sc_\bfu=\bigoplus_{m\in \N}\Smu.  
   \end{align}

 For any $\mu\in \Lambda_{\text{st}}^\ell(m)$, let $\unit{\mu}\in \Hom_{\Hu}(M^\mu,M^\mu)$ be the identity morphism of $M^\mu$.
Then the algebra $\Sc_\bfu$ is a locally unital algebra with 
$\{1_\mu\mid \mu\in \Lambda_{\text{st}}^{\ell}\}$ as a family of mutually orthogonal idempotents such that 
\[
\Sc_\bfu=\bigoplus _{\mu,\nu\in\Lambda_{\text{st}}^{\ell}} \unit{\nu}\Sc_\bfu 1_\mu.
\]
The algebra $\Sc_\bfu$ can be identified with a small $\kk$-linear category $\qSchuDJM$, with the object set $\Lambda_{\text{st}}^{\ell}$ and morphisms $\Hom_{\qSchuDJM}(\mu,\nu):=\unit{\nu}\Sc_\bfu 1_\mu$, which is  $ \Hom_{\Hu}(M^\mu,M^\nu) $ (respectively, $0$ ) if $\mu\in \Lambda_{\text{st}}^\ell(m), \nu \in\Lambda_{\text{st}}^\ell(m')$ with $m=m'$ (respectively, $m\neq m'$).

For any $\lambda\in\Par^\ell(m)$ and $\mu\in \Lambda_{st}^\ell(m)$, we refer to \cite{DJM98} for the notion of  standard $\lambda$-tableaux and semistandard $\lambda$-tableau of type $\mu$, etc.
Let $\std(\lambda)$ be the set of all standard $\lambda$-tableaux. 
Recall $\t^\lambda$ is maximal $\lambda$-tableau and for any $\t\in \std(\lambda)$, $d(\t)$ is the minimal length representative in $\mathfrak S_\lambda/ \mathfrak S_m$ such that $\t=\t^\lambda d(\t)$.  
For $\s,\t \in\std(\lambda)$, let 
 \begin{align}  \label{def-mst}
 m_{\s\t}:= H_{d(\s)}^* m_\lambda H_{d(\t)},  
 \end{align}
 where $*$ is the anti-involution of $\mathcal H_{m,\bfu}$ fixing all the generators $H_j$ and $X_i$.

Let $\SST(\lambda,\mu)$ be the set of all semistandard $\lambda$-tableaux of type $\mu$. For any $\mathbf S\in \SST(\lambda,\mu)$ and $\mathbf T\in \SST(\lambda,\nu)$, let 
    \begin{equation}
        m_{\mathbf S\mathbf T}=\sum_{\s\in \mu^{-1}(\mathbf S),\t\in\nu^{-1}(\mathbf T) }m_{\s\t}.
    \end{equation}
   Set $\phi_{\mathbf S\mathbf T}\in \Hom_{\Hu}(M^\nu,M^\mu)$ by 
 \begin{equation} \label{phi}
     \phi_{\mathbf S\mathbf T}(m_\nu h)=m_{\mathbf S\mathbf T}h
 \end{equation}
 for $h\in \Hu$. 
Then $\Smu$ has a cellular basis \cite{DJM98} given by 
\begin{equation}
    \label{basisofdjm}
    \{\phi_{\mathbf S\mathbf T}\mid \mathbf S\in \SST(\lambda,\mu), \mathbf T\in \SST(\lambda,\nu), \mu,\nu\in \Lambda_{\text{st}}^\ell(m), \lambda\in \Par^\ell(m)\}. 
    \end{equation}
    
\subsection{A functor $\mathcal G$ from $\qSchu$ to $\qSchuDJM$}
Recall the set $\Lambda_{\text{st}}^{\emptyset,\ell} $ in \eqref{def:barlambdau}. We formally extend the object set of  $\qSchuDJM$ to be $\Lambda_{\text{st}}^{1+\ell}$ such that $ \Hom_{\qSchuDJM}(\lambda,\mu)=0$ if $\lambda\notin \Lambda_{\text{st}}^{\emptyset,\ell} $ or $\mu\notin \Lambda_{\text{st}}^{\emptyset,\ell} $.

For any $\lambda =(\lambda_1, \ldots,\lambda_{k})\in \Lambda_{\text{st}}$, the {\em $i$-th merge} of $\lambda$ is defined to be     
$$
\la^{\vartriangle_i}:=(\lambda_1, \lambda_2,\ldots, \lambda_{i-1},\lambda_{i}+\lambda_{i+1},\lambda_{i+2},\ldots, \lambda_k),$$ 
for $1\le i\le k-1$. Recall the sum $\sigma_{\lambda_i, \lambda_{i+1}}$ defined by \eqref{eq:sigmaab}.
Then we have 
\begin{equation}\label{glambdai}
    \x_{\la^{\vartriangle_i}}=\x_\lambda \sigma_{\lambda_i,\lambda_{i+1}} = \sigma_{\lambda_i,\lambda_{i+1}}^* \x_\lambda,
\end{equation}
where $*$ is the $\kk$-linear anti-involution of the (finite) Hecke algebra $\mathcal H_m$ fixing all generators $H_i$.

\begin{proposition}
[{\cite[Theorem 5.4]{Bru24}}]
\label{prop:isom-qschur} Suppose $\bfu=(u_1)$, then $\Hu\cong \mathcal H_m$ and
 we have an algebra isomorphism  
\[
\phi: \bigoplus_{\lambda,\mu \in \Lambda_{\text{st}}(m)}\Hom_{\qW}(\lambda,\mu) \stackrel{\cong}{\longrightarrow} \End_{\mathcal H_m}(\oplus _{\lambda\in \Lambda_{\text{st}}(m)} M^\lambda)
\]
such that the generators are sent by
$$\begin{aligned}\phi\Big(1_{*}\begin{tikzpicture}[anchorbase,color=\clr]
	\draw[-,line width=1pt] (0.28,-.3) to (0.08,0.04);
	\draw[-,line width=1pt] (-0.12,-.3) to (0.08,0.04);
	\draw[-,line width=2pt] (0.08,.4) to (0.08,0);
        \node at (-0.22,-.4) {$\scriptstyle \lambda_i$};
        \node at (0.4,-.45) {$\scriptstyle \lambda_{i+1}$};\end{tikzpicture}1_{*}\Big) &: M^{\lambda} \rightarrow M^{\la^{\vartriangle_i} }, \quad  \x_\lambda h\mapsto \x_{\la^{\vartriangle_i}}h, \\     
\phi\Big(1_{*}\begin{tikzpicture}[anchorbase,color=\clr]
	\draw[-,line width=2pt] (0.08,-.3) to (0.08,0.04);
	\draw[-,line width=1pt] (0.28,.4) to (0.08,0);
	\draw[-,line width=1pt] (-0.12,.4) to (0.08,0);
        \node at (-0.22,.6) {$\scriptstyle \lambda_i$};
        \node at (0.36,.6) {$\scriptstyle \lambda_{i+1}$};
\end{tikzpicture}1_{*} \Big)&: M^{\la^{\vartriangle_i}}\mapsto M^{\lambda},\quad  \x_{\la^{\vartriangle_i}}h\mapsto \x_{\la^{\vartriangle_i}} h,
\end{aligned}
$$
for any $h\in \mathcal H_m$, where $1_*$ stands for suitable identity morphisms.    
\end{proposition}

 Associated with each generator of $\qSchu$ in \eqref{generator-affschur}--\eqref{redcrossing}, we introduce certain elements in $\qSchuDJM$ as follows.

 Suppose $ \mu ,\lambda \in  \Lambda_{\text{st}}^{\emptyset,\ell}$. Recall $m_\lambda$ and $m_\mu$ from \eqref{def-mlambda}.

 \begin{enumerate}
     \item $\merge\in \Hom_{\qSchu}(\mu,\lambda)$: In this case, $\lambda$ is obtained from $\mu$  by contracting $a=\mu^{(i)}_j$ and $b=\mu^{(i)}_{j+1}$ as $\mu^{(i)}_j+\mu^{(i)}_{j+1}$ for some $i$th component and some $j$.  Then
$\pi_{\mu}=\pi_{\lambda}$
and 
$
\x_{\lambda}=\x_{\mu}\sigma_{a,b}=\sigma_{a,b}^*\x_\mu
$, where 
$\sigma_{a,b}$
is given as in \eqref{eq:sigmaab}. Moreover, we have 
$
\sigma_{a,b}\pi_{\mu}=\pi_{\mu}\sigma_{a,b}
$
and
$
m_{\lambda}=\sigma_{a,b}^*m_{\mu}=m_\mu \sigma_{a,b} 
$.
Then we define 
$
\brY\in \Hom_{\Hu}(M^\mu,M^{\lambda})
$
to be the left multiplication by $\sigma_{a,b}^* $.

\item $\splits \in \Hom_{\qSchu}(\lambda,\mu)$: we define $\bY$ to be  the inclusion from $M^{\lambda}$ to $M^\mu$. 

\item $
  \upliftuip\in \Hom_{\qSchu}(\mu,\lambda)
  $: we define $\upliftuip\in \Hom_{\Hu}(M^\mu,M^{\lambda})$ to be inclusion from $M^\mu$ to $M^{\lambda}$.
  

\item $\downliftuip\in \Hom_{\qSchu}(\lambda,\mu)$: we define $\downliftuip\in \Hom_{\Hu}(M^\mu,M^{\lambda})$ to be the left multiplication by
$\prod_{1\le j\le \mu^{(i)}_{h_i}}(X_{k+j}-u_{i+1})$, where $k=\sum_{j=1}^i |\mu_j^{(i)}|-\mu_{h_i}^{(i)}$.   

\item  $
 \wkdotr,\ \wkzdotr
~\in \Hom_{\qSchu}(\mu,\mu)
$: In this case we have $r=\mu^{(i)}_j$ for some $i,j$ with $i\ge1$. Set 
$s=\sum_{1\le h\le i-1}|\mu^{(h)}|+\sum_{1\le l\le j-1}\mu^{(i)}_l$. We define 
$
\wkdotr
\in \Hom_{\Hu}(M^\mu,M^\mu)
$
to be the left multiplication by $\prod_{s+1\le k\le s+\mu^{(i)}_j}X_k$.   Similarly, we define 
$
\wkzdotr
\in \Hom_{\Hu}(M^\mu,M^\mu)
$
to be the left multiplication by $\prod_{s+1\le k\le s+\mu^{(i)}_j}X^{-1}_k$.    In general, the map associated to $\wkdota$ with $|r|<a$ is determined by \eqref{splitmerge}.

\item $\crossingpos\in  \Hom_{\qSchu}(\mu,\lambda)$: In this case $\lambda $
is obtained from $\mu$ by swapping 
$a=\mu^{(i)}_j$ 
and 
$b=\mu^{(i)}_{j+1}$ 
for some $i,j$ with $i\ge 1$.
Then $\pi_\mu=\pi_{\lambda}$.
We define the morphism 
$
\crossingpos
\in \Hom_{\Hu}(M^\mu,M^{\lambda}) 
$ 
to be
\[
\crossingpos\colon M^\mu \longrightarrow M^{\lambda}, \quad m_\mu \mapsto \pi_{\mu} \phi\left(
\crossingpos\right)(\x_\mu),    
\]
where $\phi$ is the isomorphism given in 
Proposition ~\ref{prop:isom-qschur}.
Note that by \eqref{cross via square} and  Proposition \ref{prop:isom-qschur}, the morphism  
$
\crossingpos
$ 
is determined by 
$\bY$'s 
and
$\brY$'s.
 \end{enumerate}
For any generator as described in \eqref{generator-affschur}--\eqref{redcrossing}, when viewed within $\Hom_{\qSchu}(\lambda, \mu)$ for some $\lambda \notin \Lambda_{\text{st}}^{\emptyset, \ell}$ or $\mu \notin \Lambda_{\text{st}}^{\emptyset, \ell}$, we define the corresponding morphisms in $\qSchuDJM$ to be zero. 

\begin{theorem}
 \label{thm:G}
There is a functor $\mathcal G: \qSchu\rightarrow \qSchuDJM$, which sends an object $\mu$ to $\mu$, and the generating morphisms in  \eqref{generator-affschur}--\eqref{redcrossing} to the corresponding morphisms 
$\brY$, 
$\bY$, 
$\upliftui$, 
$\downliftui$,
$\wkdota$, 
$\crossingpos$, respectively.
\end{theorem}

\begin{proof}
   The proof of this theorem is completely analogous to the proof of its degenerate
    counterpart \cite[Theorem 4.7]{SW24Schur}, and we will give an outline. It suffices to prove the results over $\C(q)$, i.e., to check the relations \eqref{webassocC}--\eqref{dotmovecrossingC}, \eqref{redsliderC}--\eqref{redbraidC}, and \eqref{eq:unitlambda}  are satisfied  under $\mathcal G$.

The relations  \eqref{webassocC}--\eqref{splitbinomialC} follow from Proposition \ref{prop:isom-qschur}. Since the positive crossing and negative crossing (resp., thin dot) correspond to left multiplication of $H_i$ and $H_i^{-1}$ (resp., $X_i$), The relation \eqref{dotmovecrossingC} holds by the relation \eqref{HXH} for the affine Hecke algebras.
 Finally, the remaining relations holds by similar arguments as that for the degenerate analog relations in \cite[Theorem 4.7]{SW24Schur} and we omit the details. (For example, we use the identity $H_i(X_i-u)(X_{i+1}-u)= (X_i-u)(X_{i+1}-u)H_i$ when we check \eqref{redsliderC}.)
\end{proof}


\subsection{Elementary ribbon diagrams of $\qSchu$}
 For $\lambda,\mu\in \Lambda_{\text{st}}(m)$, 
 recall from \eqref{dottedreduced} the set $\RPMat_{\lambda,\mu}$. Denote  
\begin{equation}
\label{euq:defofparm}
 \PMat_{\lambda,\mu}:=\{ (A, P))\mid  A=(a_{ij})\in \Mat_{\lambda,\mu}, P=(\nu_{ij}), \nu_{ij}\in \Par_{a_{ij}} \}.
 \end{equation}

For $\lambda,\mu\in \Lambda_{\text{st}}^{\emptyset,\ell}(m)$,  we define the set $\PMat_{\lambda,\mu}$ to be the collection of pairs of $\ell \times \ell$-block matrices $(A=(A_{pq})_{p,q=1}^\ell, P=(P_{pq})_{p,q=1}^\ell)$ such that $(A,P)\in \PMat_{\bar\lambda,\bar\mu}$. Such a pair is usually denoted as
\[
\big(A=(a_{(p,i),(q,j)}), P=(\eta_{(p,i),(q,j)}) \big)
\]
 
 Moreover, we define the set of bounded-partition-enhanced $\ell\times \ell$-block $\N$-matrices 
 \begin{align}
     \label{PM0}
\PMat_{\lambda,\mu}^\flat:= \big\{ (A,P)\in\PMat_{\lambda,\mu}\mid   l(\eta_{(p,i),(q,j)})\le \min\{p,q\}-1, \forall i,j,p,q 
 \big \}.
\end{align}

Recall the crossing degree and dot degree from Definition~\ref{def:Cdegred} and Definition~\ref{def:dotdegred} of a red strand dotted ribbon diagram.

\begin{proposition}  
\label{pro:spanqSchu}
For $\lambda,\mu\in \Lambda_{\text{st}}^{\emptyset,\ell}(m)$, the Hom-space $\Hom_{\qSchu}(\mu,\lambda)$
is spanned by $\PMat_{\lambda,\mu}^\flat$.
\end{proposition}

\begin{proof}
We prove by induction on the crossing degree $\cdeg$ and the dotted degree $\ddeg$ in the same time. All the $\equiv$ symbol in this proof means equivalent modulo terms with lower crossing degrees and lower dotted degrees.

Recall that 
\begin{align}
\label{symmetric}
\mathord{
\begin{tikzpicture}[baseline = -1mm,scale=0.8,color=\clr]
	\draw[-,thick] (0.28,0) to[out=90,in=-90] (-0.28,.6);
	\draw[-,thick] (-0.28,0) to[out=90,in=-90] (0.28,.6);
	\draw[-,thick] (0.28,-.6) to[out=90,in=-90] (-0.28,0);
	\draw[-,thick] (-0.28,-.6) to[out=90,in=-90] (0.28,0);
        \node at (0.3,-.75) {$\scriptstyle b$};
        \node at (-0.3,-.75) {$\scriptstyle a$};
\end{tikzpicture}
}&=
\mathord{
\begin{tikzpicture}[baseline = -1mm,scale=0.8,color=\clr]
	\draw[-,thick] (0.2,-.6) to (0.2,.6);
	\draw[-,thick] (-0.2,-.6) to (-0.2,.6);
        \node at (0.2,-.75) {$\scriptstyle b$};
        \node at (-0.2,-.75) {$\scriptstyle a$};
\end{tikzpicture}
}\:,\end{align}
By Lemma~\ref{lem:wrdotcross} and \eqref{symmetric}, we have  for any $-b\le r\le b$ that
\begin{equation}
\label{movedottofirststrand}
  \begin{tikzpicture}[baseline = -1mm,scale=0.8,color=\clr]
	\draw[-,thick] (0.2,-.6) to (0.2,.6);
	\draw[-,thick] (-0.2,-.6) to (-0.2,.6);
        \node at (0.2,-.75) {$\scriptstyle b$};
        \node at (-0.2,-.75) {$\scriptstyle a$};
        \draw (.2, 0) \bdot;
        \node at (0.55,0) {$\scriptstyle \omega_r$};
\end{tikzpicture}
\equiv
\begin{tikzpicture}[baseline = -1mm,scale=0.8,color=\clr]
	\draw[-,thick] (0.28,0) to[out=90,in=-90] (-0.28,.6);
    \draw[wipe]  (-0.28,0) to[out=90,in=-90] (0.28,.6);
	\draw[-,thick] (-0.28,0) to[out=90,in=-90] (0.28,.6);
    \draw[-,thick] (-0.28,-.6) to [out=90,in=-90] (0.28,0);
        \draw[wipe] (0.28,-.6) to[out=90,in=-90] (-0.28,0);
	\draw[-,thick] (0.28,-.6) to[out=90,in=-90] (-0.28,0);
        \node at (0.3,-.75) {$\scriptstyle b$};
        \draw (-.28, 0) \bdot;
        \node at (-0.6,0) {$\scriptstyle \omega_r$};
        \node at (-0.3,-.75) {$\scriptstyle a$};
\end{tikzpicture}.
\end{equation}

 By \eqref{redcross2} and \eqref{dotmoveadaptor} we see that
   \begin{equation}
\label{eq:movewhite}
   \begin{tikzpicture}[baseline = 7pt, scale=0.35, color=\clr]
\draw[-,line width=1pt,color=\cred](-0.1,-.6)to (-0.1,2.2);
\draw (-0.1,-.9) node{$\scriptstyle \red{u}$};
\draw[-,line width=1pt](.5,-.5) to[out=135, in=down] (-.5,1);
\draw[-,line width=1pt](-.5,1) to[out=up, in=270] (.5,2.2);
\draw  (-.5,1) \zdot;
\draw (.5,-0.9) node{$\scriptstyle {1}$};
\end{tikzpicture}
~=~ 
\begin{tikzpicture}[baseline = 7pt, scale=0.35, color=\clr]
\draw[-,line width=1pt,color=\cred](-1.8,-.25)to (-1.8,2.15);
\draw (-1.8,-.6) node{$\scriptstyle \red{u}$};
\draw[-,line width=1pt] (-1,2.15) to (-1,-0.25);   
\draw (-1,-.6) node{$\scriptstyle {1}$};
\end{tikzpicture}
~-u~
\begin{tikzpicture}[baseline = 7pt, scale=0.35, color=\clr]
\draw[-,line width=1pt,color=\cred](-1.8,-.25)to (-1.8,2.15);
\draw (-1.8,-.6) node{$\scriptstyle \red{u}$};
\draw[-,line width=1pt] (-1,2.15) to (-1,-0.25);   
\draw (-1,1) \zdot;
\draw (-1,-.6) node{$\scriptstyle {1}$};
\end{tikzpicture}
~\equiv~
-u\begin{tikzpicture}[baseline = 7pt, scale=0.35, color=\clr]
\draw[-,line width=1pt,color=\cred](-1.8,-.25)to (-1.8,2.15);
\draw (-1.8,-.6) node{$\scriptstyle \red{u}$};
\draw[-,line width=1pt] (-1,2.15) to (-1,-0.25);   
\draw (-1,1) \zdot;
\draw (-1,-.6) node{$\scriptstyle {1}$};
\end{tikzpicture} .
   \end{equation}
By \eqref{eq:movewhite}, \eqref{movedottofirststrand} and Lemma~\ref{lem:rightdot}, in general we have
\[
\begin{tikzpicture}[baseline = 7pt, scale=0.35, color=\clr]
\draw[-,line width=1pt,color=\cred](-1.8,-.25)to (-1.8,2.15);
\draw (-1.8,-.6) node{$\scriptstyle \red{u}$};
\draw[-,line width=1.2pt] (-1,2.15) to (-1,-0.25); 
\draw (-1,1) \zdot;
\draw (0,1)  node{$\scriptstyle \omega_{r}$};
\draw (-1,-.6) node{$\scriptstyle {a}$};
\end{tikzpicture}
~\equiv~z
\begin{tikzpicture}[baseline = 7pt, scale=0.35, color=\clr]
\draw[-,line width=1pt,color=\cred](-0.1,-.6)to (-0.1,2.2);
\draw (-0.1,-.9) node{$\scriptstyle \red{u}$};
\draw[-,line width=1pt](.5,-.5) to[out=135, in=down] (-.5,1);
\draw[-,line width=1pt](-.5,1) to[out=up, in=270] (.5,2.2);
\draw[-,line width=1pt] (.5,-.5) -- (.5,2.2);
\node at (1.5,.85){$\scriptstyle{a-r}$};
\draw  (-.5,1) \zdot;
\node at (-1.2,1){$\scriptstyle{\omega_r}$};
\draw (.5,-0.9) node{$\scriptstyle {a}$};
\end{tikzpicture}\quad \text{ for some }z\in \kk^\times .
\]
Thus whenever a red strand dotted ribbon diagram contains $\omega_r$ for some $r<0$, we may move the strand containing $\omega_r$ through all the red strand \red{$u_i$} to its left (up to lower degree terms). Hence such a red strand dotted ribbon diagram results as $0$ by \eqref{eq:unitlambda}.

Therefore, the proof of this proposition reduces to the case where all the elementary dot packets $\omega_{a,\nu}$ satisfy $\nu\in \Par_a$. The proof in this case is entirely similar to that for \cite[Proposition~ 5.1]{SW24Schur}, and we skip the detail.
\end{proof}

Later in Theorem~\ref{thm:basisqSchu} we will show that $\PMat_{\lambda,\mu}^\flat$ indeed forms a basis of $\Hom_{\qSchu}(\mu,\lambda)$.

\subsection{The SST-ribbon diagrams}
  \label{subsec:SST}
Recall we identify $\Lambda_{\text{st}}^{\ell}$ with the subset $\Lambda_{\text{st}}^{\emptyset,\ell} $ of $\Lambda_{\text{st}}^{1+\ell}$ by identifying $\mu=(\mu^{(1)}, \ldots, \mu^{(\ell)})$ with $(\emptyset, \mu^{(1)}, \ldots, \mu^{(\ell)})$.
Suppose that $\lambda\in \Par^\ell(m)$ and $\mu\in \Lambda_{\text{st}}^\ell(m)$. Recall that the forgetful map takes $\mu$ to the composition 
\[
\bar \mu =(\mu^{(1)}_1,\ldots, \mu^{(1)}_{h_1}, \mu^{(2)}_1,\ldots, \mu^{(2)}_{h_2},\ldots,\mu^{(r)}_1,\ldots, \mu^{(r)}_{h_{\ell}}),
\]
of $m$, where 
$ h_i=l(\mu^{(i)})$ for $1\le i\le \ell$. We may identify $\bar \mu $ with a strict composition by omitting the empty components (i.e., those corresponding to $h_i=0$).

 Associated to any semistandard tableau 
$\bT\in \SST(\lambda,\mu)$  we can construct a matrix $A_{\bT}\in \Mat_{\overline{\mu},\overline{\lambda}}$ as in \cite[\S 4.5]{SW24Schur} such that 
$a_{(p,i),(q,j)}$ 
is the number of $i_p$ in $j$th row of
$\mathbf T^{(q)}$.   Here we index the row  and column of any matrix in $\Mat_{\bar\mu,\bar\lambda}$
in the order 
\begin{align*}
 &
 (1,1),(1,2),\ldots, (1, l(\mu^{(1)}), (2,1), \ldots, (2, l(\mu^{(2)})), \ldots, (\ell,1),\ldots, (\ell, l(\mu^{(\ell)})),\quad \text{ and }
  \\
 & 
 (1,1),(1,2),\ldots, (1, l(\lambda^{(1)}), (2,1), \ldots, (2, l(\lambda^{(2)})), \ldots, (\ell,1),\ldots, (\ell, l(\lambda^{(\ell)})), 
\end{align*}
 respectively.

Thus, we have a reduced chicken foot ribbon diagram $[A_{\bT}]$  of shape $A_{\bT}$. Let $[\bT]$ denote the ornamentation associated with $(A_{\bT},(\emptyset))$. We also denote by $\div$ the anti-involution on $\qSchu$ induced by $\div$ on $\qASch$ in \eqref{eq:autoflip}. Applying the symmetry $\div$ gives us $[{\bT}]^\div \in \Hom_{\qSchu}(\mu,\la)$. The diagrams of the form $[{\bT}]$ or $[{\bT}]^\div$ are called {\em SST-ribbon diagrams (or SST-morphisms)}. 

Recall from Theorem \ref{thm:G} the functor $\mathcal G: \qSchu\rightarrow \qSchuDJM$. Recall the cellular basis $\{ \phi_{\bS\bT}\}$ from \eqref{phi}--\eqref{basisofdjm} for the cyclotomic  $q$-Schur algebras.

\begin{proposition}
 \label{pro:propG}
Suppose that $\lambda\in \emph{\Par}^\ell(m)$ and $\mu,\nu\in \Lambda_{\text{st}}^\ell(m).$ Then the functor $\mathcal G$ sends $[\bT]\circ [\bS]^\div$ to $\phi_{\bS\bT}$, for any $\bS\in \emph{\SST}(\lambda,\mu)$ and $\bT\in \emph{\SST}(\lambda,\nu)$.
\end{proposition}

\begin{proof}
    The proof is completely analogue to \cite[Proposition~4.9]{SW24Schur}, and we skip the detail here.    
\end{proof}

\subsection{The double SST basis}

In this subsection, we shall construct a double SST basis for the path algebra of the cyclotomic $q$-Schur category $\qSchu$:
 \begin{align} \label{eq:qSchum}
 \qSchu(m):= \bigoplus_{\mu,\nu\in \Lambda_{\text{st}}^\ell(m)}\Hom_{\qSchu}(\mu,\nu),
 \qquad \text{ for } m \ge 1.
 \end{align}
 Moreover, we shall identify the double SST basis for $\qSchu(m)$ with the cellular basis $\{\phi_{\bS \bT}\}$ for $\qSchuDJM$ in \eqref{basisofdjm}. 

 For any $\lambda\in \Par^\ell(m)$, let $\qSch^{\rhd\lambda}(m)$ be the span of all morphisms factoring  through another object $\gamma\in \Par^\ell(m)$ with $\gamma\rhd \lambda$. Equivalently, $\qSch^{\rhd\lambda}(m)$ is the two-sided ideal of the algebra $\qSchu(m)$ generated by all $1_{\gamma}$, for $\gamma\in \Par^\ell(m)$ with $\gamma\rhd \lambda$.

\begin{proposition}
 \label{pro:spancycdjm}
For any $\mu\in \Lambda^\ell_{\text{st}}(m)$ and $\lambda\in \emph{\Par}^\ell(m)$, we have
\[
\Hom_{\qSchu}(\lambda,\mu)+\qSch^{\rhd\lambda}(m)= \kk\text{-span} \big\{ [\bT] + \qSch^{\rhd\lambda}(m)\mid \bT\in \emph{\SST}(\lambda,\mu) \big\}.
\]
\end{proposition}

\begin{proof}
Denote the span on the right-hand side in the formula above by $\text{Span}_{\lambda,\mu}$.
    By Proposition~\ref{pro:spanqSchu}, it suffices to show  that 
$(A,P)  +\qSch^{\rhd\lambda}(m)\in \text{Span}_{\lambda,\mu}$ for any 
$(A,P)\in\PMat_{\mu,\lambda}^\flat$. The proof of this statement is completely analogous to \cite[Proposition 5.2]{SW24Schur}. 
\end{proof}

 \begin{theorem} [Double SST basis]
 \label{thm:SSTbasis}
Let $m\ge 1$. 
\begin{enumerate}
\item 
For $\mu,\nu\in \Lambda_{\text{st}}^\ell(m)$, $\Hom_{\qSchu}(\mu,\nu)$ has a basis given  by 
\begin{align}  \label{doubleSST Hom}
\bigcup_{\lambda\in \Par^\ell(m)}
\big\{ [\bT]\circ [\bS]^\div \mid \bT\in \SST(\lambda,\nu), \bS\in \SST(\lambda,\mu) \big\}. 
\end{align}
\item 
The algebra $\qSchu(m)$ has a basis given  by 
\begin{align}  \label{doubleSST}
\bigcup_{\overset{\lambda\in \Par^\ell(m)}{\mu,\nu\in \Lambda_{\text{st}}^\ell(m)} }
\big\{ [\bT]\circ [\bS]^\div \mid \bT\in \SST(\lambda,\nu), \bS\in \SST(\lambda,\mu) \big\}. 
\end{align}
\end{enumerate}
\end{theorem}

Recall the functor $\mathcal G: \qSchu\rightarrow \qSchuDJM$ from \cref{thm:G}. 

\begin{theorem} 
 \label{thm:isoDJM}
Let $m\ge 1$. 
\begin{enumerate}
\item The functor $\mathcal G: \qSchu\rightarrow \qSchuDJM$ is an isomorphism.
\item 
For any $m\ge 1$, we have an algebra isomorphism $\mathcal G: \qSchu(m)\cong \Smu$, which sends the double SST basis \eqref{doubleSST} to the cellular basis \eqref{basisofdjm}.    
\end{enumerate}
\end{theorem}

\begin{proof} [Proof of \cref{thm:SSTbasis,thm:isoDJM}]
We first prove \cref{thm:SSTbasis}. Parts (1) and (2) are clearly equivalent by definition. Applying \(\div\) to the equation in Proposition~\ref{pro:spancycdjm} results in a similar equation for \(\Hom_{\qSchu}(\mu, \lambda) + \qSch^{\rhd\lambda}(m)\). Combining this equation with the equation in Proposition~\ref{pro:spancycdjm} (for \(\mu = \nu\)), we obtain that \(\qSchu(m) 1_\lambda \qSchu(m) + \qSch^{\rhd\lambda}(m)\) is spanned by 
\[
\big\{[\bT] \circ [\mathbf{S}]^\div + \qSch^{\rhd\lambda}(m) \mid \bT \in \SST(\lambda,\nu), \mathbf{S} \in \SST(\lambda,\mu), \mu, \nu \in \Lambda_{\text{st}}^\ell(m)\big\}.
\]
Therefore, it follows by induction on \(\lambda\) that \(\qSchu(m)\) is spanned by the elements in \eqref{doubleSST}. By Proposition~\ref{pro:propG}, this spanning set \eqref{doubleSST} is mapped by \(\mathcal{G}: \qSchu(m) \to \Sc_{m,\bfu}\) to the cellular basis \eqref{basisofdjm} for \(\Sc_{m,\bfu}\) and hence must be linearly independent. This proves that \eqref{doubleSST} forms a basis for \(\qSchu(m)\).

Now we turn to the proof of \cref{thm:isoDJM}, where Part (1) follows from (2). Part (2) follows since the homomorphism \(\mathcal{G}: \qSchu(m) \to \Sc_{m,\bfu}\) matches the two bases. 
\end{proof}

Recall $\PMat_{\nu,\mu}^\flat$ from \eqref{PM0}.
 \begin{theorem}
     \label{thm:basisqSchu}
    Suppose that $\mu,\nu\in \Lambda_{\text{st}}^\ell(m)$. Then  $\PMat_{\nu,\mu}^\flat$ forms a basis for $\Hom_{\qSch_{\bfu}}(\mu,\nu)$.
\end{theorem}

\begin{proof}
    By \cref{thm:SSTbasis}, $\Hom_{\qSchu}(\mu,\nu)$ has a basis given by \eqref{doubleSST Hom}, which is parametrized by $\SST^{\,2}_{\nu,\mu} := \bigcup_{\lambda\in \Par^\ell(m)}\{(\mathbf S, \mathbf T) \mid \mathbf S \in \SST(\lambda,\nu), \mathbf T\in \SST(\lambda,\mu) \}$. It is known (see \cite[Theorem 5.6]{SW24Schur}) that $\PMat^\flat_{\nu,\mu}$ and $\SST^{\,2}_{\nu,\mu}$ have the same cardinality. As  $\PMat^\flat_{\nu,\mu}$ spans $\Hom_{\qSchu}(\mu,\nu)$ (see Proposition~\ref{pro:spanqSchu}),  it must be a basis for it.
\end{proof}

\begin{rem}
We may also define the cyclotomic higher level affine Hecke category as the quotient of the higher level affine Hecke category in \S~\ref{subsec:iso-MS} by the relation \eqref{eq:unitlambda}. Then the path algebra of this quotient category is the same as the cyclotomic $\ell$-Hecke algebra in \cite[Definition 6.1]{MS21} (also see \cite{Web20}), which is defined to be the quotient of the higher level affine Hecke algebra by the ideal generated by all  $1_{\lambda}$ such that  $\lambda^{(0)}\neq \emptyset$. 
 
Maksimau and Stroppel defined a cyclotomic $q$-Schur algebra as the endomorphism ring of certain permutation modules of the cyclotomic $\ell$-Hecke algebra \cite[Definition~ 6.15]{MS21}, which is a quotient of the higher level affine Schur algebra. They also showed that their cyclotomic $q$-Schur algebra is isomorphic to Dipper-James-Mathas'  cyclotomic $q$-Schur algebra \cite[Remark 6.16]{MS21}, and hence the cyclotomic $q$-Schur algebra in \cite{MS21} is isomorphic to our $\qSchu(m)$  by Theorem \ref{thm:basisqSchu}(3).  
\end{rem}
\subsection{Bases for $\qW_\bfu$}
Recall the set
 $\PMat_{\lambda,\mu}$ in \eqref{euq:defofparm} which is a subset of  $\RPMat_{\lambda,\mu}$ \eqref{dottedreduced} and hence maybe viewed as   morphisms in $\qW_\bfu$. 
 For any matrix  $ P=(\nu_{ij})_{l(\lambda)\times l(\mu)}$ of partitions, denote 
\[
l(P)=\max\{l(\nu_{ij})\mid 1\le i\le l(\lambda), 1\le j\le l(\mu)\}.
\]
We define
\begin{equation}
\label{Def-spansetof-cycqweb}    
\PMat_{\lambda,\mu}^\ell:=\{(A,P)\in\PMat_{\lambda,\mu}\mid l(P)\le \ell-1\}.
\end{equation}

\begin{lemma}
\label{lem:cycqweb-span}
For any $\lambda,\mu \in\Lambda_{\text{st}}(m)$, $\Hom_{\qW_{\bfu}}(\mu,\lambda)$ is spanned  by $ \PMat_{\lambda,\mu}^\ell$.    
\end{lemma}
\begin{proof}
    First we note that by \eqref{eq:zdotmovecrossing} and \eqref{unfold}, we have 
    \begin{equation}
\label{movewdottofirststrand}
  \begin{tikzpicture}[baseline = -1mm,scale=0.8,color=\clr]
	\draw[-,thick] (0.2,-.6) to (0.2,.6);
	\draw[-,thick] (-0.2,-.6) to (-0.2,.6);
        \node at (0.2,-.75) {$\scriptstyle b$};
        \node at (-0.2,-.75) {$\scriptstyle a$};
        \draw (.2, 0) \zdot;
\end{tikzpicture}
=
\begin{tikzpicture}[baseline = -1mm,scale=0.8,color=\clr]
	\draw[-,thick] (-0.28,0) to[out=90,in=-90] (0.28,.6);
    \draw[wipe]  (0.28,0) to[out=90,in=-90] (-0.28,.6);
	\draw[-,thick] (0.28,0) to[out=90,in=-90] (-0.28,.6);
    \draw[-,thick] (-0.28,-.6) to [out=90,in=-90] (0.28,0);
        \draw[wipe] (0.28,-.6) to[out=90,in=-90] (-0.28,0);
	\draw[-,thick] (0.28,-.6) to[out=90,in=-90] (-0.28,0);
        \node at (0.3,-.75) {$\scriptstyle b$};
        \draw (-.28, 0) \zdot;
        \node at (-0.3,-.75) {$\scriptstyle a$};
\end{tikzpicture}.
\end{equation}
We claim that each local component  $\zdota$ can be written as a linear combination of diagrams without white dots. 
In fact,
We may prove the claim by induction on $a$ using   \eqref{movewdottofirststrand}
and the fact $\Pi_{1\le j\le \ell }g_{a}(u_j)1_*=0$, for $a\in \Z_{\ge 1}$ in $\qW_{\bfu}$.
 Suppose $a=1$. if $\zdota$ is on  the left most of the diagram, then 
multiplying   $\zdota 1_*$ on both sides of 
$\Pi_{1\le j\le \ell }g_{1}(u_j)1_*=0 $ we see that the claim holds. 
If it is not on the leftmost, we may use   \eqref{movewdottofirststrand} to move it to the leftmost and the claim holds, too.

Suppose $a>1$. By \eqref{movewdottofirststrand} we may assume that   $\zdota$ is at the leftmost without loss of any generality.  Then  multiplying   $\zdota 1_*$ on both $\Pi_{1\le j\le \ell }g_{a}(u_j)1_*=0$ and use the fact that 
\[ \begin{tikzpicture}[baseline = -2.4mm,scale=1, color=\clr]
\draw[-,line width=1.2pt] (0.08,-.3) to (0.08,-.5);
\draw[-,line width=1.2pt] (0.08,.7) to (0.08,.3);
\draw (-.1,0) \bdot;
\draw (.07,0.5) \zdot;
\draw[-,line width=1pt](0.08,.32) to [out=left,in=left] (0.08, -.32);
\draw[-,line width=1pt](0.08,.32) to [out=right,in=right] (0.08, -.32);
\node at (-0.16,-.28) {$\scriptstyle r$};
\node at (0.08,-.59) {$\scriptstyle a$};
\end{tikzpicture}\overset{\eqref{eq:dotSM2}}=\begin{tikzpicture}[baseline = -2.4mm,scale=1, color=\clr]
\draw[-,line width=1.2pt] (0.08,-.3) to (0.08,-.5);
\draw[-,line width=1.2pt] (0.08,.7) to (0.08,.3);
\draw[-,line width=1pt](0.08,.32) to [out=left,in=left] (0.08, -.32);
\draw[-,line width=1pt](0.08,.32) to [out=right,in=right] (0.08, -.32);
\node at (-0.16,-.28) {$\scriptstyle r$};
\node at (0.08,-.59) {$\scriptstyle a$};
\draw (.25,0) \zdot;
\end{tikzpicture},  \]
we see that the claim holds by inductive hypothesis on $a-r$. This completes the proof of the claim.

The claim implies that   $\Hom_{\qW_{\bfu}}(\mu,\lambda)$ is spanned by diagrams without white dots. Then the remaining proof of the result is similar to \cite[Lemma 4.2]{SW24web}. We skip the details here.
\end{proof}

\begin{theorem}
\label{thm:basis-cyc-qweb} 
Suppose that ${\bfu} =(u_1, \ldots, u_\ell) \in(\kk^*)^\ell$. Then $\Hom_{\qW_{\bfu}}(\mu,\lambda)$ has a basis given by $\PMat_{\lambda,\mu}^\ell$, for any $\lambda,\mu \in\Lambda_{\text{st}}(m)$. Consequently, $\qW_\bfu$ is a full subcategory of $\qSchu$ with objects $\{(\emptyset^\ell,\mu)\mid \mu\in \Lambda_{\text{st}}\}$.   
\end{theorem}

\begin{proof}
There is a natural functor $\mathcal F: \qAW\rightarrow \qASch_{\bfu}$ of $\kk$-linear categories, which sends an object $\lambda\in \Lambda_{\text{st}}$ to $(\emptyset^\ell,\lambda)\in \Lambda_{\text{st}}^{1+\ell}$ and sends any morphism $f$ to 
 $\begin{tikzpicture}[baseline = 10pt, scale=0.5, color=\clr]
\draw[-,line width =1pt,color=\cred] (-4.2,.2) to (-4.2,1.6);
\draw (-4.2,0) node{$\scriptstyle \red{u_1}$};    
\draw[-,line width =1pt,color=\cred] (-3.5,.2) to (-3.5,1.6);
\draw (-3.4,0) node{$\scriptstyle \red{u_2}$};
\draw(-2.6,.8) node{$\ldots$};
\draw[-,line width =1pt,color=\cred] (-2,.2) to (-2,1.6);
\draw (-1.9,0) node{$\scriptstyle \red{u_{\ell}}$};
\draw (-1.2,1) node{$f$};
\end{tikzpicture}
$. 
By Lemma~ \ref{lem:cycpolyvanish}, $\mathcal F$ factors through $\qW_\bfu$, resulting to a functor $\mathcal F: \qW_\bfu \rightarrow \qSch_\bfu$. Since $\mathcal F$ sends the spanning set $\PMat_{\lambda,\mu}^\ell$ for $\Hom_{\qW_{\bfu}}(\mu,\lambda)$ (see Lemma~ \ref{lem:cycqweb-span}) onto the basis $\PMat_{(\emptyset^{\ell},\lambda),(\emptyset^\ell,\mu)}^\flat$ of the corresponding Hom-space for $\qSch_\bfu$, $\PMat^\ell_{\lambda,\mu}$ must be a basis. The last statement follows from this.
\end{proof}

\bibliographystyle{alpha}
\bibliography{qWeb}

\end{document}